\newif\ifshortversion
\shortversionfalse  

\ifshortversion
  \documentclass[table,x11names,a4paper]{amsart} 
  \usepackage[margin=1in]{geometry}
\else
  \documentclass[table,x11names]{amsart}
  \usepackage[margin=1.25in]{geometry}
  \usepackage{etoolbox}
  \makeatletter
  \pretocmd{\section}{\addtocontents{toc}{\protect\addvspace{5\p@}}}{}{}
  \makeatother
\fi

\usepackage{va, sty}
\usepackage{bm}

\usepackage{textcomp}
\usepackage{array,multirow}

\usepackage{float}
\usepackage{tikz}
\usetikzlibrary{arrows}
\usetikzlibrary{cd}

\usepackage[pdftex]{hyperref}
\hypersetup{colorlinks,citecolor=blue,linkcolor=blue,linktocpage,hyperindex=true,backref=true}
\allowdisplaybreaks

\usepackage{textcomp}
\usepackage{bm}

\theoremstyle{plain}
\newtheorem{innercustomthm}{Theorem}
\newenvironment{customthm}[1]
  {\renewcommand\theinnercustomthm{#1}\innercustomthm}
  {\endinnercustomthm}

\date{October 6, 2019} 
\title{ADE surfaces and their moduli}
\author{Valery Alexeev}
\email{valery@uga.edu}
\address{Department of Mathematics, University of Georgia, Athens GA
  30602, USA}
\author{Alan Thompson}
\email{a.m.thompson@lboro.ac.uk}
\address{Department of Mathematical Sciences, Loughborough University, 
Loughborough, Leicestershire, LE11 3TU, UK}

\begin{document}
\numberwithin{equation}{section}

\begin{abstract}
  We define a class of surfaces corresponding to the 
  \ade root lattices and construct compactifications of their moduli
  spaces as quotients of projective varieties for Coxeter fans,
  generalizing Losev-Manin spaces of curves. We exhibit
  modular families over these moduli spaces, which extend to
  families of stable pairs over the compactifications.
  One simple application is a geometric compactification of the moduli of
  rational elliptic surfaces that is a finite quotient of a projective
  toric variety.
\end{abstract}

\maketitle 

\tableofcontents

\section{Introduction} 

There are two sources of motivation for this work: Losev-Manin spaces
\cite{losev2000new-moduli-spaces} and degenerations of K3
surfaces with a nonsymplectic involution \cite{alexeev2019stable-pair}. 

Let $L_{n+3}$ be the moduli space parameterizing weighted stable
curves $(Z, Q_0+Q_\infty + \epsilon\sum_{i=1}^{n+1} P_i)$ of genus 0
with $n+3$ points, where $0<\epsilon\ll1$.  Equivalently, the
singularity condition is that the $n+1$ points $P_i$ are allowed to
collide while the remaining two may not collide with any others. One
has $\dim L_{n+3}=n$.  Quite remarkably, $L_{n+3}$ is a projective
toric variety for the Coxeter fan (also called the Weyl chamber fan)
for the root lattice $A_n$, formed by the mirrors to the roots. Of
course it comes with an action of the Weyl group $W(A_n)=S_{n+1}$
permuting the points $P_i$. The moduli space of the pairs
$(Z,Q_0+Q_\infty + \epsilon R)$ for the divisor
$R=\sum_{i=1}^{n+1} P_i$ with unordered points is then
$L_{n+3}/S_{n+1}$.

There are other ways in which $L_{n+3}$ corresponds
to the root lattice $A_n$. For example, its interior, over which the
fibers are $Z\simeq\bP^1$, is the torus $\Hom(A_n,\bC^*)$, and the
discriminant locus, where some of the points $P_i,P_j$ coincide, is a
union of root hypertori $\cup_\alpha \{e^\alpha=1\}$ with $\alpha=e_i-e_j$
going over the roots of~$A_n$. Additionally, the worst singularity that the
divisor $\sum P_i$ can have is $(x-1)^{n+1}=0$, which is an
$A_n$-singularity.

Losev and Manin asked in \cite{losev2000new-moduli-spaces} if similar
moduli spaces existed for other root lattices. This was partially
answered by Batyrev and Blume in \cite{batyrev2011on-generalizations}
where they constructed compact moduli spaces for the $B_n$ and $C_n$
lattices as moduli of certain pointed rational curves with an
involution. Batyrev-Blume's method works only for infinite series of
root lattices, such as $ABCD$, and it breaks down for $D_n$ where it
leads to non-flat families (most fibers have dimension 1 but some have 2).

In this paper, we generalize Losev-Manin spaces to the $D_n$ and
$E_n$ lattices by replacing stable curve pairs
$(Z, Q_0+Q_\infty + \epsilon R)$ by (KSBA) stable slc
\emph{surface} pairs \xdr and constructing their compact moduli.

\smallskip

Namely, we define a class of surface pairs \xdr naturally associated
with the root lattices $A_n$, $D_n$, and $E_n$. We call these pairs
\emph{\ade double covers}, as all of them are double covers
$\pi\colon X\to Y$ of surface pairs \ycb.  Here, $C$ and $D$ are
reduced boundaries (downstairs and upstairs), $R$ is the ramification
divisor, and $B$ is the branch divisor of $\pi$.  We call the pairs
\ycb the \emph{\ade pairs}, and the underlying pairs $(Y,C)$ the
\emph{\ade surfaces} (with reduced boundary $C$).

We prove that the moduli space $M$ of \ade pairs (equivalently of \ade
double covers) of a fixed type is a torus for the associated \ade lattice
$\Lambda$ modulo a Weyl group $W$, and that the
normalization of the moduli compactification $\oM\slc$ is the
$W$-quotient of a projective toric variety for a generalized
Coxeter fan corresponding to~$\Lambda$. Moreover, for each type
we construct an explicit modular family of \ade pairs 
over $M$ and show that, after a suitable coordinate change,
the discriminant locus in $M$,
where $B$ is singular, is a union of root hypertori
$\cup_\alpha\{e^\alpha=1\}$ with $\alpha$ going over the roots of
$\Lambda$. Additionally, the worst singularity appearing in the double cover $X$ is the surface
Du Val singularity of type $\Lambda$.

For $\Lambda=A_n$ we get the standard Coxeter fan and
$\oM\slc = L_{n+3}/S_{n+1}$.  The ramification curve $R=B$ in this
case is hyperelliptic, a double cover $f\colon B\to Z$ of a rational
genus 0 curve. The boundary $C$ has two irreducible components
defining the boundary $Q_0+Q_\infty$ of $Z$, and the ramification
points of $f$ provide the remaining $n+1$ points in the data for a
stable Losev-Manin curve $(Z, Q_0+Q_\infty + \epsilon\sum_{i=1}^{n+1} P_i)$.

For $\Lambda=D_n$ and $E_n$ the fan is a generalized Coxeter fan, a
coarsening of the standard Coxeter fan. It is the normal fan of a
permutahedron given by a classical Wythoff construction.

\smallskip

We found these \ade surfaces and pairs by studying degenerations of K3
surfaces of degree 2.  A polarized K3 surface $(X,L)$ of degree
$L^2=2$ comes with a canonical double cover $\pi\colon X\to Y$. The
ramification divisor $R$ of $\pi$ is intrinsic to $(X,L)$, and the
pair $(X,\epsilon R)$ is a stable slc pair. Thus, the moduli of (KSBA)
stable slc pairs provides a canonical moduli compactification
$\oF_2\slc$ of the moduli space $F_2$ of K3 surfaces of degree 2.

On the other hand, there exists a nice toroidal compactification
$\oF_2\utor$ defined by the Coxeter fan for the reflection group of
the root lattice associated to $F_2$. The type~III strata of
$\oF_2\utor$ are products of $W$-quotients of projective toric
varieties for the Coxeter fans of certain \ade root lattices.  These
strata look like the moduli spaces of degenerate stable slc pairs \ycb
whose irreducible components $(Y_i, C_i + \freps B_i)$ are some of the
\ade surface pairs discussed above. Indeed, we confirmed this 
in many examples.
We determine the precise connection between $\oF_2\slc$ and $\oF_2\utor$ in
\cite{alexeev2019stable-pair}, which is a continuation of this paper.

\medskip

We work over the field $\bC$ of complex numbers.
Throughout, $\epsilon$ will denote a
sufficiently small real number: $0<\epsilon\ll1$. This means that for
fixed numerical invariants there exists an $\epsilon_0>0$ such
that the stated conditions hold for any $0<\epsilon\le\epsilon_0$.
Now let us explain the main results and the structure of the present paper
in more detail.

\smallskip

In Section~\ref{sec:double-covers} we define $(K+D)$-trivial polarized
involution pairs $(X,D,\iota)$ and study their basic properties. Roughly speaking,
such pairs consist of a normal surface $X$ with an anticanonical divisor $D$ and
 an involution $\iota\colon X \to X$ that preserves $D$.
They naturally appear when studying stable degenerations of K3
surfaces with a nonsymplectic involution. We prove that the quotient
$(Y,C)=(X,D)/\iota$ of an involution pair
is a log del Pezzo surface of index 2, i.e. the divisor $-2(K_Y+C)$ is
Cartier and ample.

Denoting by $\pi\colon X\to Y$ the double cover, $B \in |-2(K_Y+C)|$ 
the branch divisor and $R\subset X$ the
ramification divisor, one has
$K_X + D + \epsilon R = \pi^*(K_Y+ C+ \freps B)$. Then the pair
$(X, D + \epsilon R)$ is a (KSBA) stable slc pair iff the pair
$(Y,C + \freps B)$ is such.

By analogy with Kulikov degenerations of K3 surfaces, we divide the
pairs $(X,D,\iota)$ and their quotients $(Y,C)$ into types I, II, III. 
For type I, one has
$C=D=0$, the surface $X$ is an ordinary K3 surface with Du Val
singularities, and the pair $(Y,C + \freps B)$ is klt. For types II
and III, the pairs $(X, D + \epsilon R)$ and $(Y,C + \freps B)$ are
both not klt; these types are distinguished by the properties of the 
boundary $D$, which is a disjoint union of smooth elliptic curves in type II
and a cycle of rational curves in type III.

With this motivation, we set out to investigate log canonical non-klt
del Pezzo surfaces with boundary $(Y,C)$ of index 2, and the moduli
spaces of log canonical pairs \ycbend, with $B\in |-2(K_Y+C)|$.

\smallskip

In Section~\ref{sec:adepairs} we explicitly define many examples of 
such surfaces $(Y,C)$ in an \emph{ad
hoc} way. Since the word \emph{type} is already used for ``types I,
II, III'', we call the
combinatorial classes of such surfaces \emph{shapes}. Those of type III
we call \ade shapes, and of type II we call \wade shapes. We call the
corresponding surfaces $(Y,C)$ \ade resp. \wade \emph{surfaces},
the stable pairs $(Y,C + \freps B)$  \ade resp. \wade \emph{pairs},
and their covers $(X, D + \epsilon R)$  \ade resp. \wade \emph{double covers}.
To each shape we associate a decorated \adeend, resp. \wade 
Dynkin diagram, which we use to label the shape, and a
corresponding \adeend, resp. \wade lattice. The main reason for this association
comes later, when considering the moduli spaces and their
compactifications.

In the simplest cases, the surfaces $Y$ are toric and $C$ is a part of the toric
boundary, with two components $C_1,C_2$ in type III and one component
in type II.  These shapes are labeled by diagrams of types $A_n$, $D_n$, $E_n$, 
$\wD_{2n}$, $\wE_7$ and $\wE_8$. At this point there is a clear motivation behind 
this labeling scheme, as the defining lattice polytopes of the toric surfaces $Y$
contain the corresponding Dynkin diagrams in an obvious way.
In type II we also introduce several nontoric shapes,
which we call $\wA_{2n-1}$,
$\wA_1^*$, and $\wA_0^-$. Interestingly there is no $\wE_6$ shape;
Remark~\ref{rem:no-wE6} discusses some reasons for that. 

Next we define a procedure, which we call \emph{priming}, for
producing a new lc nonklt del Pezzo pair $(\oY',\oC')$ of index 2 from an
old such pair $(Y,C)$. The procedure consists of making weighted
blowups $Y'\to Y$ at a collection of up to 4 points on the boundary
$C$, and then performing a contraction $Y'\to \oY'$ defined by the
divisor $-2(K_{Y'}+C')$ (where $C'$ is the strict transform of $C$), 
provided that it is big and nef. 

We list all the \ade and \wade shapes,
together with their basic numerical invariants and singularities in
Tables~\ref{tab:primed-III} and \ref{tab:primed-II}. In all, there are
43 \ade shapes and 17 \wade shapes, some of which define infinite families.
Whilst this list seems rather large, most are obtained by applying the priming
operation to a very short list of fundamental shapes. We call these fundamental
shapes \emph{pure shapes}, and call the ones obtained from them by priming
\emph{primed shapes}.

\smallskip

In Section~\ref{sec:nakayama} we prove our first main result, which 
justifies our interest in the \ade and \wade surfaces.

\begin{customthm}{A}\label{thm:logdP=ade}
  The log canonical non-klt del Pezzo surfaces $(Y,C)$ with $2(K_X + C)$ Cartier 
  and $C$ reduced (or possibly empty) are exactly the same as the \ade and \wade
  surfaces $(Y,C)$, pure and primed.
\end{customthm}

 Most of the proof can be extracted from
the work of Nakayama~\cite{nakayama2007classification-of-log}, with
additional arguments necessary in genus~1. Nakayama's
classification of log del Pezzo pairs of index 2 was done in very
different terms and the connection with root lattices did not appear in it.

\smallskip

In Section~\ref{sec:moduli}, for each shape we describe the moduli
spaces of \ade (i.e. type III) pairs and their double covers.
For each shape we have a root lattice $\Lambda$ of \ade type.
It has an associated torus $T_{\Lambda} := \Hom(\Lambda,\bC^*)$ and
Weyl group $W_{\Lambda}$. 
Then our second main result is as follows.

\begin{customthm}{B}\label{thm:moduli-summary}
  The moduli stack of \ade pairs of a fixed \ade shape is 
  \begin{eqnarray*}
    &[\Hom(\Lambda^*, \bC) : \mu_\Lambda\times\mu_2] =
      [T_\Lambda : W_\Lambda \times \mu_2]    
    &\text{for pure $A$ shapes,} \\
    &[\Hom(\Lambda^*, \bC) : \mu_\Lambda] = [T_\Lambda : W_\Lambda]    
    &\text{for pure $D$ and $E$ shapes,} \\
    &    [\Hom(\Lambda^*, \bC) : \mu_{\Lambda'} \times W_0] = [T_{\Lambda'} :
    W_\Lambda \rtimes W_0] &\text{for primed shapes.}
  \end{eqnarray*}
  Here, $\Lambda$ is an \ade root lattice, $\Lambda^*$ is its dual
  weight lattice, 
  $\Lambda'$ is a lattice
  satisfying $\Lambda\subset\Lambda'\subset\Lambda^*$ given explicitly in
  Theorem~\ref{thm:moduli-ade-primed},  $T_{\Lambda'} := \Hom(\Lambda',\bC^*)$,
   $\mu_{\Lambda'} := \Hom(\Lambda^*/\Lambda',\bC^*)$,
  and the additional Weyl group $W_0$ is given in
  Theorem~\ref{thm:reconstruction}, with action described in
  Theorem~\ref{thm:W0-action}.
\end{customthm}
 
This result is proved as Thms.~\ref{thm:moduli-ade-pure} (for pure
shapes) and \ref{thm:moduli-ade-primed} (for primed shapes).
To conclude Section~\ref{sec:moduli}, for each pure \ade shape
we construct a Weyl group invariant modular family of \ade pairs, 
which we call the \emph{naive
  family}, over the torus~$T_{\Lambda^*}$. 
\smallskip

In Section~\ref{sec:compactifications} for each \ade (i.e. type III)
shape we construct a modular compactification of the moduli space of \ade
pairs of this shape.  In ~\ref{sec:stable-pairs-general} we begin with
a general discussion of moduli compactifications using stable pairs,
and we define stable \ade pairs. Next, for each \ade shape we
construct a Weyl group invariant family of stable slc pairs \ycb over a
projective toric variety $V_M\cox$ for the Coxeter fan of an
appropriate over-lattice $M\supset\Lambda^*$ of index $2^k$
(Thms.~\ref{thm:compactified-family-A},
\ref{thm:compactified-family-A'DE}, 
\ref{thm:compactified-family-primed}). 
These theorems also describe the
combinatorial types of the stable pairs over each point of $V_M\cox$.
For the \ade surfaces where $C$ has two components, the irreducible
components of these pairs are again \ade pairs for Dynkin
subdiagrams. For some of the primed \ade shapes where $C$ has one 
or zero components, new ``folded'' shapes appear.

Next, we define a generalized Coxeter fan as a coarsening of the
Coxeter fan, corresponding to a decorated Dynkin diagram, and the
corresponding projective toric variety $V_M\semi$. We prove that our
family is constant on the fibers of $V_M\cox\to V_M\semi$ and the
types of degenerations are in a bijection with the strata of
$V_M\semi$, with the moduli of the same dimension. As a consequence,
we obtain our third main theorem. This theorem follows from 
Thm.~\ref{thm:compact-ade-moduli}, which is a slightly stronger result.

\begin{customthm}{C}\label{thm:compact-ade-moduli-custom}
  For each \ade shape the moduli space $M\slc_{ADE}$ is proper and the
  stable limit of \ade pairs are stable \ade pairs. 
  \begin{enumerate}
  \item For the pure \ade shapes, the normalization of $M\slc_{ADE}$
    is $V\semi_\Lambda/W_\Lambda$, a $W_\Lambda$-quotient of the
    projective toric variety for the generalized Coxeter fan.
  \item For the primed shapes, the normalization
    $(M\slc_{ADE})^\nu$ is $V\semi_{\Lambda'}/W_\Lambda\rtimes W_0$,
    for a lattice extension $\Lambda'\supset\Lambda$. The lattice
    $\Lambda'$ and the Weyl group $W_0$ are as in Theorem~\ref{thm:moduli-summary}.
  \end{enumerate}
\end{customthm}

\smallskip 

The moduli spaces described in Theorem~\ref{thm:moduli-summary}
have many automorphisms, some of which extend to
automorphisms of our compactification. In
Section~\ref{sec:canonical-families} we prove that there exists an
essentially unique deformation of the naive family such that its
pullback to the torus $T_{\Lambda^*}$ has the following wonderful
property: the discriminant locus becomes the union of the root
hypertori $\{e^\alpha=1\}$, with $\alpha$ going over the roots of the
corresponding \ade root lattice.  We also prove that this
deformation extends to the compactification. This is our
fourth main theorem.

\begin{customthm}{D}\label{thm:canonical-families-summary}
  For each \ade shape
  there exists a unique deformation of the equation $f$ of the naive
  family such that $\Discr(f) = \Discr(\Lambda)$.  The resulting
  \emph{canonical family} of \ade pairs extends to a family of stable
  pairs on the compactification for the generalized Coxeter fan.  The
  restriction of this compactified canonical family to a boundary
  stratum is the canonical family for a smaller Dynkin diagram.
\end{customthm}

This theorem is proved in two parts, as Theorems~\ref{thm:two-discrs}
and \ref{thm:canonical-family-compactified}. In the final subsection~\ref{sec:sings-ade-pairs} 
we use these canonical families to explicitly
determine all the possible singularities of the branch divisor $B$ that
can appear in our \ade pairs.

In Section~\ref{sec:connections} we discuss an application of our results
and its connections with other work. 
In Section~\ref{sec:elliptic}, as an application  we
construct a compactification $\oM\dell$ of the moduli space of
rational elliptic surfaces with section and a distinguished $I_1$ fiber
(i.e. irreducible rational with one node). The compactification
is by the stable slc pairs $(X, D + \epsilon R)$ where $D$ is the $I_1$
fiber and $R$ is the fixed locus of the elliptic involution. We prove
that the normalization of $\oM\dell$ is a $W_{E_8}$-quotient of a
projective toric variety for the generalized Coxeter fan for the $E_8$
lattice. In Section~\ref{sec:GHK} we discuss the relationship
of our work to that of Gross-Hacking-Keel on moduli of anticanonical pairs 
\cite{gross2015moduli-of-surfaces}, 
and in Section~\ref{sec:cremona} we discuss its relationship with the
classification of birational involutions in the Cremona group
$\Bir(\bP^2)$ \cite{bayle2000birational-involutions}.

\section{Log del Pezzo index 2 pairs and their double covers}
\label{sec:double-covers}

\begin{definition}\label{def:invo-pair}
  A \emph{$(K+D)$-trivial polarized involution pair} $(X, D, \iota)$
  consists of a normal surface $X$ with an effective reduced divisor
  $D$, and an involution $\iota\colon X\to X$, $\iota(D)=D$ such that
  \begin{enumerate}
  \item $K_X + D \sim 0$ is a Cartier divisor linearly equivalent to 0,
  \item the fixed locus of $\iota$ consists of an ample Cartier
    divisor $R$, henceforth called the \emph{ramification divisor},
    possibly along with some isolated points, and
  \item the pair $(X, D+\epsilon R)$ has log canonical (lc)
    singularities for $0<\epsilon\ll1$.
  \end{enumerate}
\end{definition}

\begin{remark}
  Such pairs naturally appear when studying degenerations of K3
  surfaces with an involution. In
  \cite{alexeev2019stable-pair} we show that for any one parameter
  degeneration of K3 surfaces $\cS\to (Z,0)$ with a nonsymplectic
  involution $\iota_\cS$ and a ramification divisor $\cR_\cS$, if
  $(\cS_0, \epsilon \cR_0)$ is the stable slc limit of the pairs
  $(\cS_t,\epsilon \cR_t)$ for $0<\epsilon\ll1$, 
  then each irreducible component $X$ of the
  normalization of $(\cS_0, \cR_0)$ comes with an involution $\iota$
  and, denoting by $D$ its double locus, the pair $(X,D,\iota)$ is a
  $(K+D)$-trivial polarized involution pair as in
  \eqref{def:invo-pair}.
\end{remark}

Let $\omega$ be a global generator of the 1-dimensional space
$H^0\big(\cO_X(K_X+D)\big)$. The ramification divisor $R$ is nonempty
by ampleness and has no components in common with $D$ by the lc
condition. For a generic point $x\in R$ there are local parameters
$(u,v)$ such that $\iota(u,v)=(u,-v)$. Then
$\iota^*(du\wedge dv)=-du\wedge dv$. Thus, the involution $\iota$ is
\emph{non-symplectic}, meaning $\iota(\omega)=-\omega$.

Let $\pi\colon X\to Y=X/\iota$ be the quotient map, $C=\pi(D)$ the
boundary and $B=\pi(R)$ the branch divisors. By Hurwitz formula,
$K_X + D \equiv \pi^* (K_Y + C + \frac12 B)$. 

\begin{lemma}\label{lem:pairs-and-covers}
  There is a one-to-one correspondence between $(K+D)$-trivial
  polarized involution pairs 
  $(X,D,\iota)$ and pairs $(Y,C+\frac{1+\epsilon}2 B)$ such that
  \begin{enumerate}
  \item $Y$ is a normal surface and $C,B$ are reduced effective Weil
    divisors on it.
  \item $(Y,C)$ is a (possibly singular) del Pezzo surface with
    boundary of index $\le2$, i.e. $-2(K_Y+C)$ is an ample Cartier
    divisor.
  \item $B\in |-2(K_Y+C)|$; in particular $B$ is Cartier.
  \item The pair $(Y, C + \frac{1+\epsilon}2 B)$ has lc singularities
    for $0<\epsilon\ll1$. 
  \end{enumerate}
  Moreover, if (1)--(4) hold then one also has
  \begin{enumerate}
    \setcounter{enumi}{4}
  \item For any singular point $y\in Y$: if $y\in B$ then $y$ is Du
    Val and $y\not\in C$.
  \end{enumerate}
\end{lemma}
\begin{proof}
  Suppose (1)--(4) hold and $y\in B$ is a non Du Val singularity of
  $Y$ or a Du Val singularity with $y\in C$.  Then on a minimal
  resolution $g\colon\wY\to Y$ there exists an exceptional divisor $E$
  whose discrepancy with respect to $K_Y+C$ is $<0$. Since $2(K_Y+C)$
  is Cartier, one has $a_E(K_Y+C)\le -\frac12$. But $B$ is Cartier, so
  \begin{displaymath}
    a_E\left( K_Y+C + \frac{1+\epsilon}2 B \right)
    \le -\frac12 - \frac{1+\epsilon}2 < -1,
  \end{displaymath}
  and the pair $(Y, C + \frac{1+\epsilon}2 B)$ is not
  lc, a contradiction. This proves (5). 

  Now let $(X,D,\iota)$ be a $(K+D)$-trivial polarized involution pair. Using
  $\iota^*(\omega)=-\omega$, it follows 
  by \cite[Prop.2.50(4)]{kollar2013singularities-of-the-minimal} 
  that for any $x\in X$ \'etale-locally $(X,x)\to (Y,\pi(x))$
  is the index-1 cover for the pair $(Y, C + \frac12 B)$. Thus,
  $\pi_* \cO_X = \cO_Y \oplus \omega_Y(C)$, the divisor $2(K_X+C)$ is
  Cartier, and $B=(s)$, $s\in H^0\big( \cO_Y(-2(K_Y+C)) \big)$. From
  the identity
  $K_X + D + \epsilon R \equiv \pi^* (K_Y + C + \frac{1+\epsilon}2 B)$
  it follows that the divisor $K_Y + C + \frac{1+\epsilon}2 B$ is
  ample and the pair $(Y, C + \frac{1+\epsilon}2 B)$ has lc
  singularities.

  Vice versa, let $(Y,C+\frac{1+\epsilon}2 B)$ be a pair as above, and let
  $X:= \Spec_Y \cO_Y \oplus \omega_Y(C)$ be the double cover
  corresponding to a section $s\in H^0\big( \cO_Y(-2(K_Y+C)) \big)$,
  $B = (s)$. Thus, \'etale-locally it is the index-1 cover for the pair
  $(Y, C + \frac12 B)$.  Then $K_X+D \sim 0$, $K_X + D + \epsilon R$
  is ample and lc, and $2R=\pi^*(B)$ is an ample Cartier divisor.

  We claim that $R$ itself is Cartier. Pick a point $x\in R$ and let
  $y=\pi(x)\in B$.  The cover $\pi$ corresponds to
  the divisorial sheaf $\cO_Y(K_Y+C)$, which is locally free at~$y$ by
  (5). Then the double cover is given by a local equation $u^2=s$, and
  $R$ is given by one local equation $u=0$, so it is Cartier.
\end{proof}

Thus, the classification of $(K+D)$-trivial polarized involution pairs is reduced to that of del
Pezzo surfaces $(Y,C)$ with reduced boundary of index $\le2$ plus a divisor $B \in |-2(K_Y+C)|$
satisfying the lc singularity condition. In the case when $C=0$, del
Pezzo surfaces of index $\le2$ with log terminal singularities
were classified by Alexeev-Nikulin in
\cite{alexeev1988classification-of-del-pezzo,
  alekseev1989classification-of-del-pezzo, alexeev2006del-pezzo}.
There are 50 main cases which are further subdivided into 73 cases
according to the singularities of $Y$. However, all these surfaces are
smoothable, which follows either by using the theory of K3 surfaces or by
\cite[Prop. 3.1]{hacking2010smoothable-del-pezzo}.
Thus, there are only
10 overall families, with a generic element a smooth del Pezzo surface
of degree $1\le K_Y^2\le 9$ (for $K_Y^2=8$ there are two families, for
$\bF_0$ and $\bF_1$). The dimension of the family of pairs $(Y,B)$,
equivalently of the double covers $(X,\iota)$, is $10+K_Y^2$.  

Del Pezzo surfaces with a half-integral boundary $C$ of index $\le2$
were classified by Nakayama in
\cite{nakayama2007classification-of-log}. An important result of
Nakayama is the Smooth Divisor Theorem
\cite[Cor.3.20]{nakayama2007classification-of-log} generalizing that
of \cite[Thm.1.4.1]{alexeev2006del-pezzo}. It says that for any del
Pezzo surface $(Y,C)$ with boundary of index $\le2$ a general divisor
$B\in |-2(K_Y+C)|$ is smooth and in particular does not pass through
the singularities of~$Y$. Thus, every such surface $(Y,C)$ produces a
family of $(K+D)$-trivial polarized involution pairs $(X,D,\iota)$.

\begin{remark}
  The divisors $C$ and $B$ play a very different role: $C$ is fixed,
  and $B$ varies in a linear system. For this reason, we will refer to
  them differently. We will call $C$ \emph{the boundary} and
  say that $(Y,C)$ is a \emph{surface with boundary} (and sometimes we will
   drop the words ``with boundary''). We will call
  $(Y,C+\frac{1+\epsilon}2 B)$ a \emph{pair}, consisting of a surface
with boundary $(Y,C)$ plus an additional choice of divisor $B$ on it. 
In many cases, surfaces
  with boundary are rigid, but pairs have moduli.
\end{remark}

Let $f\colon \wX\to X$ be the minimal resolution of singularities, and
let $\wD$ be the effective $\bZ$-divisor on $\wX$ defined by the formula
$K_\wX + \wD = f^*(K_X+D) \sim 0$. It follows from the lc condition
that $\wD$ is reduced. 

\begin{lemma}\label{lem:types-of-pairs}
  For the minimal resolution of a $(K+D)$-trivial polarized involution
  pair, 
  one of the following holds:
  \begin{enumerate}
  \item[($\I$)] $D=0$, $\wD=0$, and $X$ is canonical. Then
    $X$ is a K3 surface with \ade  singularities and $\iota$ is an
    non-symplectic involution.
  \item[($\II$)] $(X,D)$ is strictly log canonical and $\wD$ is one or
    two isomorphic smooth elliptic curve(s), 
  \item[($\III$)] $(X,D)$ is strictly log canonical and $\wD$ is a
    cycle of $\bP^1$s.
  \end{enumerate}
\end{lemma}

Accordingly, we will say that the $(K+D)$-trivial polarized involution pair
$(X,D,\iota)$ and the corresponding del Pezzo surface $(Y,C)$ with
boundary have type
$\I$, $\II$, or~$\III$. In type
  $\I$ $(Y,C)$ is klt, and in types $\II$, $\III$ it is not klt.

\begin{proof}
  ($\I$) (Compare \cite[Sec. 2.1]{alexeev2006del-pezzo}) $\wX$ is
  either a K3 surface or an Abelian surface. If $\wX=X$ is an Abelian
  surface then the involution is different from $(-1)$ since
  $R\ne0$. Thus, the induced involution $\iota^*$ on $H^0(\Omega^1_X)$
  is different from $(-1)$ and there exists a nontrivial
  1-differential on $X$ which descends to a minimal resolution $\wY$
  of $Y$. But $Y$ is a del Pezzo surface with log terminal
  singularities, so basic vanishing gives
  $h^0(\Omega^1_\wY)=h^1(\cO_\wY)=h^1(\cO_Y)=0$.  Thus, $\wX$ is a K3
  surface, and we already noted that the involution is non-symplectic.

  ($\II,\III$) Since $\omega_\wD\simeq\cO_\wD$ by adjunction, every
  connected component of $\wD$ is either a smooth elliptic curve or a
  cycle of $\bP^1$s. 
  Since $K_\wX=-\wD$ is not effective, $\wX$ is birationally ruled
  over a curve $E$ and $\wD$ is a bisection. The curve $E$
  has genus 1 or 0 since it is dominated by
  $\wD$.  If one of the connected components of $\wD$ is a cycle of
  $\bP^1$s then $g(E)=0$ and $X$ is rational. In that case from
  $H^1(-\wD) = H^1(K_\wX)=0$ we get $h^0(\cO_\wD)=h^0(\cO_\wX)=1$, so
  $\wD$ is connected. If $g(E)=1$ and $\wD$ has more than one
  connected component then then they all must be horizontal. Thus, there
  must be two of them, each a section of $\wX\to E$, so they are both
  isomorphic to~$E$.
\end{proof}

\section{Definitions of \adeend,  \wade surfaces, pairs, and double covers}
\label{sec:adepairs}

\begin{definition}\label{def:ade-surface}
  The \emph{\ade and \wade surfaces} are certain normal surfaces
  $(Y,C)$ with reduced boundary defined by the explicit constructions
  of this section. They are examples of log del Pezzo surfaces of
  index 2, i.e. each pair $(Y,C)$ has log canonical singularities, and
  the divisor $-2(K_Y+C)$ is Cartier and ample.

  In the sense of Lemma~\ref{lem:types-of-pairs}, the \ade surfaces
  are of type III, and \wade surfaces are of type II.
\end{definition}

\begin{definition}\label{def:ade-pairs}
  Given an \adeend, resp. \wade surface $(Y,C)$, let $L=-2(K_Y+C)$ be
  its polarization, an ample line bundle. If $B\in |L|$ is an
  effective divisor such that \ycb is log canonical 
  for $0<\epsilon\ll1$ 
  then \ycb is
  called an \emph{\adeend, resp. \wade pair.}  The double cover
  $\pi\colon X\to Y$ as in Lemma~\ref{lem:pairs-and-covers} is then
  called \emph{an \adeend, resp. \wade double cover.}
\end{definition}

\begin{remark}\label{rem:BC-transversal}
  By \eqref{lem:pairs-and-covers}(5) the points of intersection
  $B\cap C$ are nonsingular points of $Y$, and the log canonicity
  of \ycb implies that $B$ intersects $C$ transversally. Consequently,
  $R$ intersects $D$ transversally at smooth points of $X$.
\end{remark}

By construction, the \ade and \wade surfaces will admit a
combinatorial classification.  Since the word \emph{type} is overused,
we call the classes \emph{shapes}. To each shape we associate:
\begin{enumerate}
\item a decorated \ade or affine, extended \wade Dynkin diagram,
\item a decorated Dynkin symbol, e.g $\mA^-_5$ or $\wE_8^-$,
\item an ordinary \adeend, resp. affine \wade root lattice, e.g. $A_5$
  or $\wE_8$.
\end{enumerate}
Parts (1) and (2) are equivalent, and (3) may be obtained from them by
deleting the decorations.
The main reason for this association will become apparent later, in
the description of the moduli spaces and their compactifications.  But
in the cases where $Y$ is toric and $C$ is part of its toric boundary, they
also encode some data about the defining polytope.

\smallskip

We divide the shapes into two classes, which we call
\emph{pure} and \emph{primed}. \ade and \wade surfaces of pure
shape are fundamental, we 
define them all explicitly in subsections~\ref{sec:toric-shapes}
and \ref{sec:wA-shapes}.
In type III the pure shapes form 5 infinite families along with 3
exceptional shapes. In type II there are 2 infinite families 
and 4 exceptional shapes. 

The \ade and \wade surfaces of primed shape are secondary and there are many more of them; 
they can all be obtained from surfaces of
pure shape by an operation which we call ``priming'', explained in
subsection~\ref{sec:priming}.

\subsection{Toric pure shapes}
\label{sec:toric-shapes}

The \ade surfaces (type III) of pure shape are all toric, 
as are 3 of the \wade surfaces (type II) of pure shape. 
To construct them we
begin with polarized toric surfaces $(Y,L)$, where $L =
-2(K_Y+C)$. Such toric surfaces correspond in a 
standard way with  lattice polytopes $P$ with vertices in
$M\simeq\bZ^2$. 

\begin{lemma}\label{lem:toric}
  Let $P$ be an integral polytope with a distinguished vertex~$p_*$
  and $(Y,L)$ be the corresponding polarized projective toric
  variety.  Let $C$ be the torus-invariant divisor corresponding to
  the sides passing through~$p_*$.  Suppose that all the other sides
  of $P$ are at lattice distance 2 from $p_*$. Then
  $-2(K_Y+C) \sim L$ is ample and Cartier, and the pair $(Y,C)$ has log
  canonical singularities.
\end{lemma}
\begin{proof}
  Let $C'= \sum C'_i$ be the divisor corresponding to the sides not
  passing through the vertex $p_*$.  The zero divisor of the section
  $e^{p_*} \in H^0(Y,L)$ is $\sum d_i C'_i$ where $d_i$ are the
  lattice distances from $p_*$ to the corresponding sides.  This gives
  $L\sim 2C'$.  Combining it with the identity $K_Y + C+C'\sim 0$
  gives the first statement. It is well known that the pair $(Y,C+C')$
  has log canonical singularities. Thus, the smaller pair $(Y,C)$ also
  has log canonical singularities.
\end{proof}

\begin{definition}
  We now apply this Lemma to define some of our \ade and \wade surfaces 
  $(Y,C)$ of pure shape.  For each
  shape we list its decorated Dynkin symbol and the vertices of its
  defining polytope in Table~\ref{tab:toric-shapes}, and illustrate
  them with pictures in Figures~\ref{fig:A},
  \ref{fig:D},~\ref{fig:E},~\ref{fig:wDwE}.  In these Figures the
  sides of the polytope through $p_*$ are drawn in bold blue; they
  correspond to irreducible components of the divisor $C$.  Within 
  the polytopes we draw 
  the decorated Dynkin diagrams, the rules for doing
  this are explained in Notation~\ref{not:dynkin-diagrams}. Finally,
  we also label some of the lattice points $p_i$, for later use
  in Section~\ref{sec:moduli}. 
  
  The surface $Y$ of shape $\wD_{2n}$ is toric with a torus-invariant
  boundary $C$ only for $2n\ge6$. In the $\wD_4$ shape we formally
  define $(Y,C)$ to be either $\bP^1\times\bP^1$ with a smooth
  diagonal $C\sim s+f$ or, as a degenerate subcase, a quadratic cone
  $\bP(1,1,2)$ with a conic section.
\end{definition}

\begin{table}[htp!]
  \centering
  \renewcommand{\arraystretch}{1.1}
  \begin{displaymath}
    \begin{array}{lc|c|lll}
      \text{shape} & \min(n) & p_*& \text{polytope vertices} \\
      \hline
      A_{2n-1}&1& (0,2)& (0,2), (0,0), (2n,0) \\
      A_{2n-2}^-&1& (0,2)& (0,2), (0,0), (2n-1,0) \\
      \mA_{2n-3}^-&2& (0,2)& (0,2), (1,0), (2n-1,0) \\
      \hline
      D_{2n}& 2& (2,2)& (2,2), (0,2), (0,0), (2n-2,0) \\
      D_{2n-1}^-&3& (2,2) & (2,2), (0,2), (0,0), (2n-3,0)\\
      \hline
      \mE_{6}^-&& (2,2)& (2,2), (0,3), (0,0), (3,0)\\
      \mE_{7}&& (2,2)& (2,2), (0,3), (0,0), (4,0)\\
      \mE_{8}^-&& (2,2)& (2,2), (0,3), (0,0), (5,0)\\
      \hline
      \rule{0pt}{2.8ex}
      \wD_{2n} & 2& (2,2) & (0,2), (0,0), (2n-4,0), (4,2) \\
      \wE_7 && (2,2)& (0,4), (0,0), (4,0)\\
      \wE_8^- && (2,2)& (0,3), (0,0), (6,0)\\
      \hline
    \end{array}
\end{displaymath}
  \caption{Polytopes for the pure shapes}
  \label{tab:toric-shapes}
\end{table}

\begin{figure}[htp!]
  \centering
  \includegraphics{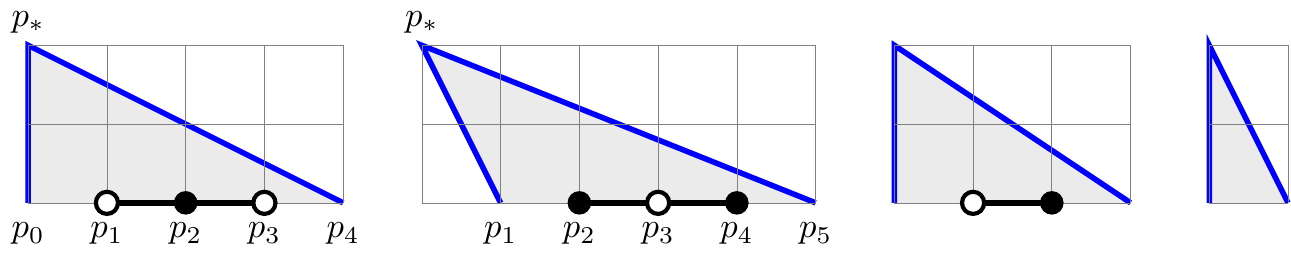}
  \caption{$A$ shapes: $A_3$, $\mA_3^-$, $A_2^-$, $A_0^-$}
  \label{fig:A}
\end{figure}

\begin{figure}[htp!]
  \centering
  \includegraphics{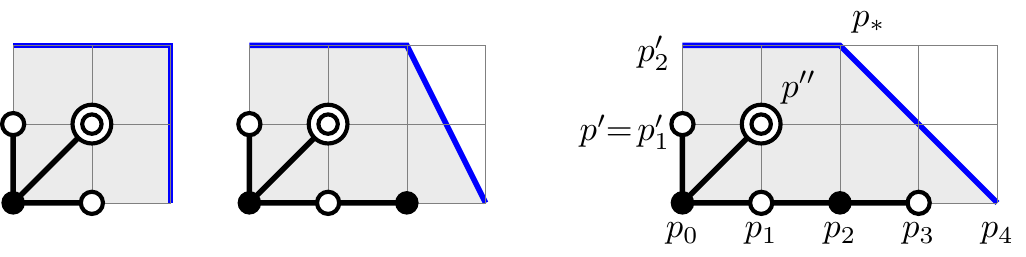}
  \caption{$D$ shapes: $D_4$, $D_5^-$, $D_6$}
  \label{fig:D}
\end{figure}

\begin{figure}[htp!]
  \centering
  \includegraphics{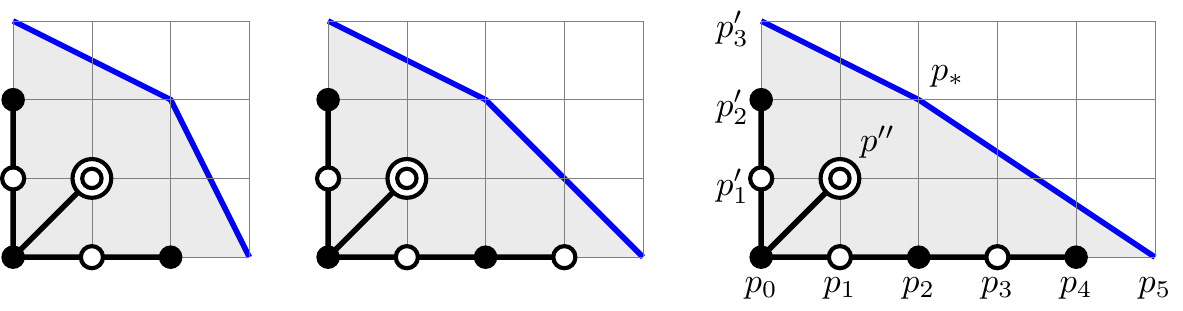}
  \caption{$E$ shapes: $\mE_6^-$, $\mE_7$, $\mE_8^-$}
  \label{fig:E}
\end{figure}

\begin{figure}[htp!]
  \centering
  \includegraphics{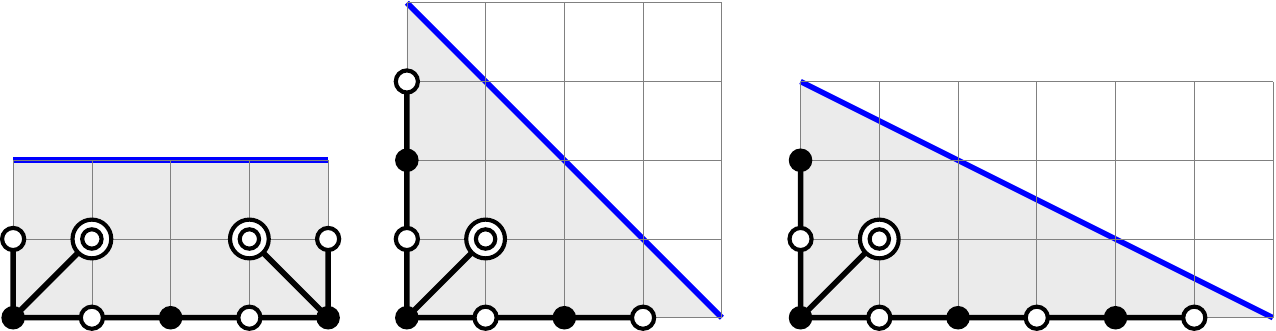}
  \caption{Type II shapes $\wD_8$, $\wE_7$ and $\wE_8^-$}
  \label{fig:wDwE}
\end{figure}

\begin{definition}
Given a surface $(Y,C)$ of pure shape, 
we call the irreducible
components of $C$ \emph{sides}. If $(Y,C)$ is of type III
there are two sides, we call them 
\emph{left} and \emph{right} and decompose 
$C = C_1 + C_2$ correspondingly. If $(Y,C)$ is of type II
there may be one side or no sides.

Let $L = -2(K_Y+C)$. We call a side $C'$ \emph{long} if 
$L.C' = 2$ or $4$, and \emph{short} if $L.C' = 1$ or $3$.
\end{definition}

In the type III cases illustrated in Figures~\ref{fig:A},
\ref{fig:D},~\ref{fig:E}, long sides have lattice length 2 and short
sides have lattice length 1.  In the type II cases illustrated in
Figure~\ref{fig:wDwE}, long sides have lattice length 4 and short
sides have lattice length 3.

Within each polytope in Figures~\ref{fig:A},
\ref{fig:D},~\ref{fig:E},~\ref{fig:wDwE} we draw the corresponding
decorated Dynkin diagram, using the following rule.

\begin{notation}\label{not:dynkin-diagrams}
  Given a surface of pure shape $(Y,C)$ defined
  torically by a polytope $P$ as above, mark a node
  for each lattice point on the boundary of
  of $P$ which is not contained in $C$, and join them
  with edges along the boundary. For any node that 
  lies at a corner of $P$, add an additional \emph{internal}
  node to the diagram and connect it to the corner node. 
  We distinguish such internal nodes by circling them in our diagrams.
  
  This process associates an \ade (resp. \wadeend) diagram to each of
  our torically-defined pure shapes of type III (resp. type II), but
  it does not give a bijective correspondence between diagrams and
  shapes. To fix this we also need to keep track of the parity.  We
  color the nodes of a diagram lying at lattice length 2 from $p_*$
  black, and the nodes lying at lattice length 1 from $p_*$
  white. Internal nodes are always colored white.
\end{notation}  
  
In the type III cases, note that each diagram has a leftmost and rightmost node, 
which sit next  to the left and right sides respectively. The length of the sides may be read off
from the colors of these nodes: white nodes correspond to long sides and
black nodes to short sides. 
  
\begin{notation} For ease of reference, to each decorated Dynkin diagram we
also  associate a decorated Dynkin symbol, in a unique way.
For the pure shapes, this is given by the name of the (undecorated) Dynkin diagram,
with superscript minus signs on the left/right to denote the locations of short sides;
as noted above, this can be read off from the colors of  the nodes at the ends of the diagram.
For instance, as illustrated in Figure~\ref{fig:A}, $A_3$ has two long sides, $\mA_3^-$ has two
short sides, and $A_2^-$ has a long side on the left and a short side on the right.
In type II cases, which have only one side, we place all decorations on the right 
by convention.
\end{notation} 

\begin{remark} \label{rem:left-right}
With this notation, the two shapes $\mA_{2n-2}$ and $A_{2n-2}^-$
are identical up to labeling of the components of $C$. Where this labeling is
unimportant, we will refer to these surfaces by the symbol $A_{2n-2}^-$, with the 
short side on the right. There are, however, some settings in which it will be important
to keep track of the labels, such as when we come to study degenerations.
\end{remark}

\begin{remark}\label{rem:no-wE6}
  Curiously, there is no $\wE_6$ shape. In our ad hoc definition above, 
  the process of adding internal nodes can only produce branches 
   of length 2. This rules out
  Dynkin diagram $\wE_6$, which has three branches of length 3. 
  A deeper reason is that in Arnold's classification of singularities
  \cite{arnold1972normal-forms} the $\wE_7$ and $\wE_8$ singularities
  exist in all dimensions $\ge2$, but $\wE_6$ starts in dimension 3 and so
  cannot appear on a surface.
\end{remark} 
  
\subsection{Nontoric $\wA$ shapes}
\label{sec:wA-shapes}

In addition to the toric surfaces described above, there are also
three nontoric \wade surfaces (type II) of pure shape. These are the
$\wA$ shapes, their decorated Dynkin diagrams and symbols are chosen
to be compatible with moduli and degenerations, although they do not
admit the same nice description in terms of polytopes as the toric
shapes. They are illustrated in Figure~\ref{fig:wA}.

 {\bf (1) $\bm{\wA_{2n-1}}$.} The surface $Y$ is a cone over an elliptic
curve and $C=0$, so there is no boundary.  More precisely, let $\cF$
be a line bundle of degree $n>0$ on an elliptic curve $E$, and let
$\wY$ be the surface $\Proj_E(\cO \oplus \cF)$. Let $s,s_\infty$ be
the zero, resp. infinity sections, and let $f\colon\wY\to Y$ be
the contraction of the zero section. Then $f^*K_Y=K_\wY+s= -s_\infty$,
so $-K_Y$ is ample with $K_Y^2=n$. If $B\in |-2K_Y|$ is a generic
section then $p_a(B)=n+1$ and the map $B\to E$ has $2n$ points of
ramification.  Of course, the surface $Y$ is not toric. The double
cover $X\to Y$ branched in $B$ is unramified at the singular point,
and $X$ has two elliptic singularities. One has $R^2 = 2K_Y^2 = 2n$.

{\bf (2) $\bm{\wA_{1}^*}$.} The surface is the projective plane
$Y=\bP^2$, the boundary $C$ is a smooth conic, and the branch curve
$B$ is a possibly singular conic. If $B$ is smooth then the double cover
$X=\bP^1\times\bP^1$; if $B$ is two lines then $X=\bF_2^0=\bP(1,1,2)$
with $R$ passing through the singular point of $X$. 
We also include here as a
degenerate subcase when $\bP^2$ degenerates to $Y=\bF_4^0$. Then
$X=\bF_2^0$ with $R$ not passing through the singular point.

{\bf (3) $\bm{\wA_{0}^-}$.} The surface is the quadratic cone
$Y=\bP(1,1,2)$ with minimal resolution $\wY=\bF_2$. The strict
preimage of $C$ on $\wY$ is a divisor in the linear system
$|s+3f|$, where $s$ is the $(-2)$-section and $f$ is a fiber.
The curve $C$ passes through the vertex of the cone and is
smooth at that point. The branch curve $B$ is a hyperplane section
disjoint from the vertex.  The double cover is $X=\bP^2$ with an
involution $(x,y,z)\mapsto (x,-y,z)$, and the boundary divisor is a
smooth elliptic curve $y^2z=f_3(x,z)$.

\begin{figure}[htp!]
  \centering
  \includegraphics{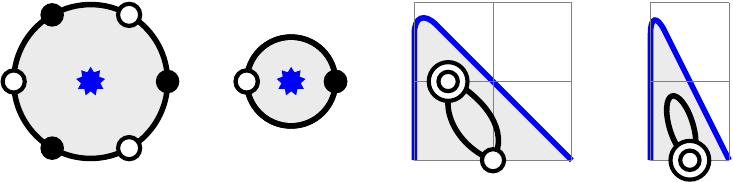}
  \caption{Nontoric type II $\wA$ shapes: $\wA_5$, $\wA_1$,
    $\wA_1^*$, $\wA_0^-$}
  \label{fig:wA}
\end{figure}
The surface $Y$ of shape $\wA_1^*$ is obtained by a ``corner
smoothing'' of a surface of toric shape $A_1$: the union of two lines
$C_1+C_2$ in $\bP^2$ is smoothed to a conic $C$.  Similarly, $\wA_0^-$
is obtained by a ``corner smoothing'' of $A_0^-$. We add the star in
$\wA_1^*$ to distinguish it from the ordinary $\wA_1$ shape, which has
no boundary.

\begin{remark}
  One observes that the $\wA$ shapes cannot be toric because the Dynkin
  diagram is not a tree.
\end{remark} 
  
 \begin{remark}
  With the single exception of $\wA_1^*$, all of our decorated Dynkin graphs are
  bipartite: black and white nodes appear in alternating order.
 \end{remark}
  
 \subsection{Primed shapes}
\label{sec:priming}

Priming is a natural operation producing a new del Pezzo surface $(\oY',\oC')$ of
index 2 from an old one $(Y,C)$.  Let $I_i\simeq (y,x^2)$ be an ideal with
support at a smooth point $P_i\in C$ whose direction is transversal to
$C$. A weighted blowup at $I_i$ is a composition of two ordinary
blowups: at $P_i$ and at the point $P_i'$ corresponding to the direction
of $I_i$, followed by a contraction of an $(-2)$-curve, making an $A_1$ surface
singularity at a point contained in the strict preimage $C'$ of $C$. Weighted blowups 
of this form are the basis of the priming operation.

\begin{definition}\label{def:priming}
  Let $(Y,C)$ be an \ade or \wade surface and let
  $P_1,\dotsc, P_k \in C$ be distinct nonsingular points of $Y$ and
  $C$. Choose ideals $I_i\simeq (y,x^2)$ with supports at $P_i$ and
  directions transversal to $C$ (the closed subschemes
  $\Spec\cO_Y/I_i$ can be thought of as vectors).
  Let $k_s$ denote the number of points
  on side~$C_s$, so  $k=\sum k_s$.
  Define $f\colon Y' \to Y$ to be the weighted blowup at 
  $I=\prod_{i=1}^kI_i$ and let $C'$ be the strict preimage of $C$.
  Let $F=\sum F_i$ be the sum of the exceptional
  divisors and $L'=-2(K_{Y'}+C')$; note that
  $L'$ is a line bundle since an $A_1$ singularity has index~2.
  
  \emph{Assume that $L'$ is big, nef, and semiample.} Then the
  \emph{priming of $(Y,C)$} is defined to be the pair $(\oY',\oC')$
  obtained by composing $f$ with the
   contraction $g\colon Y'\to \oY'$ given by $|NL'|$, $N\gg0$.
   The divisor $\oC'$ is defined to be the strict transform of $C$.
   The resulting pair $(\oY',\oC')$ is an  \ade or \wade surface of primed shape.
\end{definition}

\begin{remark}
  Priming has a very simple geometric meaning for the pairs
  $(Y,C+\frac{1+\epsilon}2 B)$.
  Let $B\in |L|$ be a curve such that \ycb is log
  canonical. By \eqref{rem:BC-transversal} the curve $B$ is
  transversal to $C$. In this case we take the ideals $I_i$ to be
  supported at some of the points $P_i\in B\cap C$, with the
  directions equal to the tangent directions of $B$ at $P_i$. 
  Priming then produces a new pair $(\oY', \oC' + \freps \oB')$ which
  disconnects $B$ from $C$ at the points $P_i$. 
  If a component of $C'$ is completely disconnected from $B'$ then it
  is contracted on $\oY'$.

  But it is on the double cover $\pi\colon (X,D+\epsilon R)\to
  (Y,C+\freps B)$
  where the priming operation becomes
  the most natural and easiest to understand.  The double cover $X'$
  of $Y'$ branched in $B'$ is an ordinary smooth blowup of $X$ at the
  points $Q_i=\pi\inv(P_i)$. So on the cover we simply make
  $k$ ordinary blowups at some points $Q_i\in D\cap R$ in the boundary
  $D$ which are fixed by the involution, then apply the linear system
  $|NR'|$, $N\gg0$, provided that $R'$ is big, nef and semiample, to
  obtain the primed pair $(\oX',\oD'+\epsilon \oR')$.  
  This disconnects $R$ from $D$ at the points $Q_i$.
  If a component of $D'$ is completely disconnected from $R'$ then it
  is contracted on $\oX'$.
\end{remark}

\begin{definition}
  In terms of the pairs, we will call the above operation
  \emph{priming of an \adeend (resp. \wadeend) pair
    \ycbend}, resp.  \emph{priming of an
    \adeend (resp. \wadeend) double cover \xdrend.} The
    result is an \adeend (resp. \wadeend) pair/double cover of
    primed shape.
\end{definition}

We reiterate that a priming only exists if  $L'$ is big, nef, and semiample. 
Below we will give a necessary and sufficient condition for existence
of a priming that is easier to check; before that, however, we need to introduce 
some basic invariants.

\begin{definition}
  The basic numerical invariants of an \ade or \wade surface $(Y,C)$,
  with polarization $L=-2(K_{Y}+C)$, are
  \begin{enumerate}
  \item the volume $v = L^2/2 >0$,
  \item the genus $g = \frac12 (K_Y+L)L -1 \ge0$,
  \item the lengths $LC_s>0$ of the sides.
  \end{enumerate}
\end{definition}
The Hilbert polynomial of $(Y,L)$ is
$\chi(Y, xL) = v x^2 + (v+1-g)x + 1$. The Hilbert polynomials of
$(C_s,L)$ are $2x+1$ for a long side and $x+1$ for a short side.
It is immediate to compute these invariants for the pure shapes. 
We list them in the highlighted rows of Tables~\ref{tab:primed-III} and
\ref{tab:primed-II}.

\begin{lemma}\label{lem:priming-changes} With notation as in Definition \ref{def:priming}:
  \begin{enumerate}
  \item For the main divisors, one has
    \begin{displaymath}
      C'=f^*C-F, \quad K_{Y'}=f^*K_Y+2F, \quad
      L'=f^*L-2F, \quad K_{Y'}+L' = f^*(K_Y+L).
    \end{displaymath}
  \item The basic invariants change as follows:
    \begin{displaymath}
      L'{}^2/2 = L^2/2 - k, \quad g(L')=g(L), \quad L'C_s' = LC_s - k_s.
    \end{displaymath}
  \end{enumerate}
\end{lemma}

\begin{theorem}[Allowed primings] \label{thm:priming}
  Let $(Y,C)$ be an \ade or \wade surface of pure shape, as defined in
  sections~\ref{sec:toric-shapes} and \ref{sec:wA-shapes}, and
  $I_1,\dotsc, I_k$ a collection of ideals as in
  Definition~\ref{def:priming}. 
  Then a necessary and sufficient condition for a priming to exist is: $L'{}^2>0$ and
  $L'C'_s\ge0$ for the sides $C_s$. Under these conditions, $L'$ is
  big, nef, and semiample, and contracts $Y'$ to a normal surface
  $\oY'$ with ample Cartier divisor $-2(K_{\oY'}+\oC')$.
\end{theorem}
\begin{proof}
  The conditions $L'{}^2>0$ and $L'C'_s\ge0$ are necessary since $L'$
  is big and nef. Now assume that they are satisfied. We exclude
  $\wA_{2n-1}$ since its boundary is empty and no primings are
  possible. We can also exclude the shapes of volume 1, which are
  $A_0^-$ and~$\wA_0^-$.  By \eqref{lem:priming-changes} one has
  \begin{displaymath}
    \frac12 L'= (K_{Y'}+L') + C' = f^*(K_Y+L) + C'.    
  \end{displaymath}
  Thus, if $K_Y+L$ is nef then $L'$ is nef. One checks that for all
  the pure shapes except for $A_1$ and $\wA_1^*$ the divisor $K_Y+L$
  is nef. Indeed, the surfaces of $A_{2n-1}$, $A^-_{2n-2}$,
  $\mA^-_{2n-3}$ and $\wE_7$, $\wE_8^-$ shapes have Picard rank 1, so
  $K_Y+L$ is nef iff the genus $g\ge0$, i.e. all except $A_1$ and the
  excluded $A_0^-$. For the $D_4, \wD_4$ shapes one has $K_Y+L=0$. For
  the other $D_{2n}$, $D^-_{2n-1}$, $\wD_{2n}$ shapes $K_Y+L$ gives a
  $\bP^1$-fibration.  Finally, for the $E$ shapes the divisor $K_Y+L$
  is big and nef: for $\mE^-_6$ it is ample, for $\mE_7$ it contracts
  the left side $C_1$ to an $\wE_7$ surface, and for $\mE_8^-$ it
  contracts the right side $C_2$ to an $\wE_8^-$ surface.

  The remaining shapes $A_1$ and $\wA_1^*$ are easy to check
  directly. In both cases $Y=\bP^2$ and $C$ is a conic: two lines for
  $A_1$ and a smooth conic for $\wA_1^*$. The divisor $L'$ is big and
  nef and contracts a $(-1)$-curve $E'$, the strict preimage of a line
  $E$ with the direction of the ideal $I$, to a surface of shape
  $A_0^-$, resp. $\wA_0^-$.

  Since $\frac12 L'$ is of the form $-(K_Y'+C')$, if it is big and nef
  then it is automatically semiample, see
  e.g. \cite[Thm.6.1]{fujino2012minimal-model}. This concludes the
  proof.
\end{proof}

\begin{corollary}
  The shapes $A_{2n-1}$, $D_{2n}$, $\wD_{2n}$, $\wE_7$ can be primed a
  maximum of 4 times, shapes $\wA_{2n-2}^-$, $\wD_{2n-1}^-$, $\wE_8^-$
  3 times, and $\mA^-_{2n-2}$, $\mE_6^-$, $\mE_8^-$ 2 times each.
\end{corollary}

\begin{remark}\label{rem:A1p}
  As we can see from the above proof, the cases $A_1' = A_0^-$ and
  $(\wA_1^*)'= \wA_0^-$ are special. Also, as we will see below, the
  dimension of the moduli space of pairs in these cases drops after priming, 
  but in all other cases it is preserved. For these reasons, and to avoid
  redundancy in our naming scheme, \emph{we do not allow primings of 
  $A_1$ and $\wA_1^*$.}
\end{remark}

We associate decorated Dynkin diagrams and symbols to primed shapes 
by modifying those of the corresponding pure shapes, as follows. Recall
that, in the pure Type III cases, each diagram has a leftmost and rightmost
node, which sit next to the left and right sides, and these nodes are 
colored white/black if and only if the corresponding side is long/short.

\begin{notation}
  For an \ade shape, when priming on a long side once we circle the
  corresponding white node, and when priming a second time we also circle
  the neighboring black node. In the Dynkin symbol we add a prime, resp. double prime on
  the left or right, depending on whether we are priming at a point of
  the left side $C_1$ or the right side $C_2$. When priming on a short
  side, we circle the corresponding black node once and turn the $-$
  superscript into a $+$ superscript (visually $-$ and $'$ gives $+$).

  For an \wade shape, we add up to 4 primes to the Dynkin symbol for a
  long side in $\wD_{2n}$ and $\wE_7$. We also turn $\wE_8^-$ into
  $\wE_8^+$ before adding up to two more primes. In the corresponding 
  decorated diagrams, we circle one node for each prime using the following rule: 
  first circle black nodes at the ends of the diagram,
  then white nodes at the ends of the diagram, then finally black
  nodes connected to circled white ones.
\end{notation}

\begin{remark} \label{rem:ambiguity} We note two pieces of mild
  ambiguity in this notation. The first is that the decorated diagrams
  for the two shapes $\ppA_{3}',\pA_{3}''$ and also for $\ppD'_4, \pD''_4$ are
  the same, so the diagram in these cases does not distinguish left and
  right sides. In practice this won't cause a problem: if we need to
  distinguish sides in these cases we will use the Dynkin symbols
  $\ppA_{3}',\pA_{3}''$, resp. $\ppD'_4,\pD''_4$.

The decorated diagrams for the shapes 
$\pA_2^+$ and $\ppA_2^-$ are also identical. 
In fact, in this case we find that the \ade surfaces
$\pA_2^+$ and $\mA_2''$ are isomorphic, so this is just another instance 
of the diagram not distinguishing left and right sides. These surfaces are 
obtained by priming  $\pA_2^-$ and $\mA_2'$, respectively,  once on the right 
and a surface of $\pA_2^-$ shape is left/right symmetric (in fact it has a
 toric description which makes this symmetry apparent, see 
Lemma~\ref{lemma:primed-toric} and Figure~\ref{fig:pA}). One way to
think of this symmetry is to consider  $\pA_2^- = \mA_2'$ as a
symmetric $\mD_2^-$ shape. 
\end{remark}

\begin{remark}
  If we wish to refer to an \ade or \wade surface with an unspecified
  decoration (i.e. either undecorated or one of $ -, \p, \pp, +$), we
  will use a question mark decoration $\mbox{}^?$.  For example,
  $A_{2n-1}^?$ refers to one of the surfaces $A_{2n-1}$, $A_{2n-1}'$,
  or $A_{2n-1}''$, while $\php A_{2n-2}^?$ refers to one of the
  surfaces $\php A_{2n-2}^-$ or $\php A_{2n-2}^+$.
\end{remark}

Note that circled white nodes can denote either internal nodes or long
sides on which a single priming has taken place. This apparent notational 
ambiguity will be explained in the following subsection.

  \begin{figure}[htp!]
    \centering
    \includegraphics{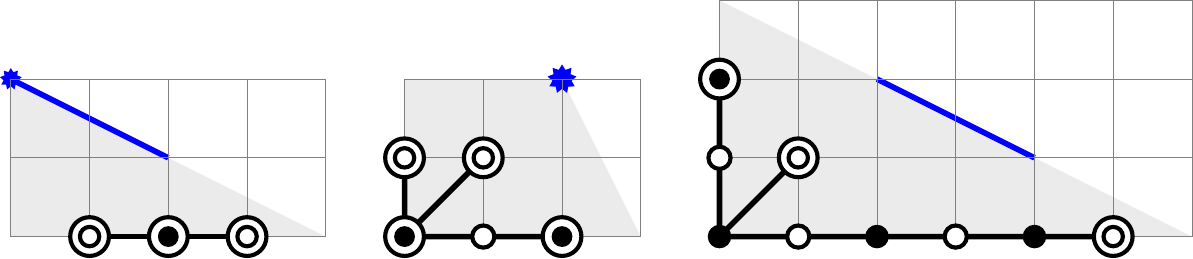}
    \caption{Decorated Dynkin diagrams for shapes $\ppA_3'$,
      $\ppD_5^+$, $\wE_8^+{}'$} 
    \label{fig:circled-diags}
  \end{figure}

\begin{example}
  In Fig.~\ref{fig:circled-diags} we give several examples of such
  diagrams. 
  The surfaces in these cases are not toric. However, we can still
  use pseudo-toric pictures to indicate the lengths $\oL\,\oC'_s$ of
  the sides and the sides $C'_s$ which are contracted by
  $Y'\to \oY'$. The volume of the surface is the volume of the
  polytope minus the number of primes, i.e. additional circles in the
  diagram as compared to a pure shape.
\end{example}

We list all the resulting 43 \ade and 17 \wade shapes and their basic
invariants in Tables~\ref{tab:primed-III} and \ref{tab:primed-II}.
The pure shapes are highlighted. Note that this
 table does not distinguish between left and right sides of $A$-shapes 
(see Remark \ref{rem:left-right}), so e.g.
$A_{2n-2}^-$ and $\mA_{2n-2}$ are listed as the same surface.
The column for the singularities is
explained in section~\ref{sec:singularities}.

\begin{table}[htp!]
  \centering
  \renewcommand{\arraystretch}{1.1}
  \begin{displaymath}
    \begin{array}{ll|llll}
      \text{shape} & \min(n) & \text{volume} & \text{genus} & LC_s 
      &\text{sings in } \nonklt(Y,C)\\
      \hline
      \rowcolor{\graycol}
      A_{2n-1}&1& 2n & n-1 & 2,2 & (n)\\ 
      A'_{2n-1}&2& 2n-1 & n-1 & 2,1 & (n), A_1 \\
      \pA_{2n-1}' &2& 2n-2 & n-1& 1,1 &  A_1, (n), A_1\\
      A_{2n-1}''& 2 & 2n-2 & n-1 & 2 & (n,2;2^2)\\
      \pA_{2n-1}'' & 2 & 2n-3 & n-1 & 1 & A_1, (n,2;2^2)\\
      \ppA_{2n-1}'' & 3 & 2n-4 & n-1 & & (2^2;2,n,2;2^2)\\
      \hline

      \rowcolor{\graycol}
      A_{2n-2}^-&1& 2n-1 & n-1 & 2,1 & (n,2), A_1\\
      \pA_{2n-2}^- & 2 & 2n-2 & n-1 & 1,1 & A_1, (n,2),  A_1\\      
      A^+_{2n-2}& 2 & 2n-2 & n-1 & 2 & (n,2,2;2^2)\\
      \pA_{2n-2}^+ & 2 & 2n-3 & n-1 &  1& A_1, (n,2,2;2^2)\\
      \ppA_{2n-2}^- & 2 & 2n-3 & n-1 & 1 & (2^2;2,n,2),  A_1\\
      \ppA_{2n-2}^+& 3 & 2n-4 & n-1 & & (2^2;2,n,2,2;2^2)\\      
      \hline

      \rowcolor{\graycol}
      \mA_{2n-3}^-&2& 2n-2 & n-1 & 1,1 & A_1, (2,n,2), A_1\\
      \mA_{2n-3}^+& 2 & 2n-3 & n-1 & 1 & A_1, (2,n,2,2;2^2)\\
      \plA_{2n-3}^+& 3 & 2n-4 & n-1 & & (2^2;2,2,n,2,2;2^2)\\
      \hline

      \rowcolor{\graycol}
      D_{2n}& 2& 2n & n-1& 2,2 &  \\
      D_{2n}' &2 & 2n-1 & n-1 & 2,1& A_1\\
      \pD_{2n} &2 & 2n-1 & n-1 & 1,2 & A_1\\
      \pD_{2n}' & 2 & 2n-2 & n-1 & 1,1 & 2 A_1\\
      D_{2n}'' & 2 & 2n-2 & n-1 & 2 & (2;2^2)\\
      \ppD_{2n} & 2 & 2n-2 & n-1 & 2 & (2^2;n)\\
      \pD_{2n}'' & 2 & 2n-3 & n-1 & 1 & A_1,(2;2^2)\\
      \ppD_{2n}' & 2 & 2n-3 & n-1 & 1 & (2^2;n), A_1\\      
      \ppD_{2n}'' & 3 & 2n-4 & n-1 & & (2^2;n,2;2^2)\\
      \hline

      \rowcolor{\graycol}
      D_{2n-1}^-&3 & 2n-1 & n-1 & 2,1 & (2),  A_1\\
      \pD_{2n-1}^- & 3 & 2n-2 & n-1 & 1,1 & A_1, (2), A_1\\
      D_{2n-1}^+ & 3 & 2n-2 & n-1 & 2 & (2,2;2^2)\\
      \pD_{2n-1}^+ & 3 & 2n-3 & n-1 & 1 & A_1, (2,2;2^2)\\
      \ppD_{2n-1}^- & 3 & 2n-3 & n-1 & 1 & (2^2;n,2), A_1\\
      \ppD_{2n-1}^+ & 3 & 2n-4 & n-1 & & (2^2;n,2,2;2^2)\\
      \hline

      \rowcolor{\graycol}
      \mE_{6}^-&& 6 & 3 & 1,1 & A_1, (3), A_1\\
      \mE_{6}^+&& 5 & 3 & 1 & A_1, (3,2;2^2)  \\
      \plE_{6}^+&& 4 & 3 & & (2^2;2,3,2;2^2) \\
      \hline
      \rowcolor{\graycol}
      \mE_{7}&& 7 & 3 & 1,2 & A_1\\
      \mE_{7}' && 6 & 3 & 1,1 & 2  A_1 \\
      \plE_{7}&& 6 & 3 & 2 & (2^2;2) \\
      \plE_{7}' && 5 & 3 & 1 & (2^2;2) , A_1\\
      \mE_{7}''&& 5 & 3 & 1 &  A_1, (2,3,2) \\
      \plE_{7}''&& 4 & 3 &  & (2^2;2,3;2^2) \\
      \hline
      \rowcolor{\graycol}
      \mE_{8}^-&& 8 & 4 & 1,1 & 2  A_1\\
      \mE_{8}^+&& 7 & 4 &  1 & A_1, (3;2^2)\\
      \plE_{8}^-&& 7 & 4 & 1 & (2^2;2) , A_1 \\
      \plE_{8}^+&& 6 & 4 &  & (2^2;2,3;2^2) \\
      \hline
\end{array}
\end{displaymath}
  \caption{All \ade shapes}
  \label{tab:primed-III}
\end{table}

\begin{table}[htp!]
  \centering
  \renewcommand{\arraystretch}{1.1}
  \begin{displaymath}
    \begin{array}{ll|llll}
      \text{shape} & \min(n) & \text{volume} & \text{genus} & LC 
      &\text{sings in } \nonklt(Y,C)\\
      \hline
      \rowcolor{\graycol}
      \rule{0pt}{2.8ex}
      \wA_{2n-1} & 1 & 2n & n+1 && \text{elliptic} \\
      \hline
      \rowcolor{\graycol}
      \rule{0pt}{2.8ex}
      \wA_1^* && 2 & 0 & 4 &  \\
      \hline
      \rowcolor{\graycol}
      \rule{0pt}{2.8ex}
      \wA_0^- && 1& 0 & 3 & A_1\\
      \hline
      \rowcolor{\graycol}
      \rule{0pt}{2.8ex}          
      \wD_{2n} & 2 & 2n & n-1 & 4 &  \\
      \wD_{2n}' &2 & 2n-1 & n-1 & 3 & A_1 \\
      \wD_{2n}'' & 2 & 2n-2 & n-1 & 2 & 2A_1 \\
      \wD_{2n}''' & 2 & 2n-3 & n-1 & 1 &  3A_1\\
      \wD_{2n}'''' & 3 & 2n-4 & n-1 & & (n;2^4)\\
      \hline
      \rowcolor{\graycol}
      \rule{0pt}{2.8ex}
      \wE_7 && 8 & 3 & 4 & \\
      \wE_7' && 7 & 3 & 3 & A_1\\
      \wE_7'' && 6 & 3 & 2 & 2A_1\\
      \wE_7''' && 5 & 3 & 1 & 3A_1\\
      \wE_7'''' && 4 & 3 & & (3;2^4)\\
      \hline
      \rowcolor{\graycol} 
      \rule{0pt}{2.8ex}
      \wE_8^- && 9 & 4 & 3 & A_1 \\                  
      \wE_8^+ && 8 & 4 & 2 & 2A_1 \\                  
      \wE_8^{+\p} && 7 & 4 & 1  & 3A_1 \\                  
      \wE_8^{+\pp} && 6 & 4 & & (3;2^4) \\                  
      \hline           
\end{array}
\end{displaymath}
  \caption{All \wade shapes}
  \label{tab:primed-II}
\end{table}

\subsection{Primed shapes which are toric}
\label{sec:primed-toric-shapes}

We observe that some of the primed shapes also admit toric
descriptions. This provides an explanation for a piece of notational
ambiguity mentioned in the previous subsection: the priming operation
on a long edge (white node) may be interpreted as modifying the
diagram to make that node internal in the toric representation.

\begin{table}[htp!]
  \centering
  \renewcommand{\arraystretch}{1.1}
  \begin{displaymath}
    \begin{array}{lc|c|lll}
      \text{shape} & \min(n) & p_*& \text{polytope vertices} \\
      \hline
      \pA_{2n-1}&2& (2,2) & (2,2), (0,1), (0,0), (2n-2,0) \\
      \pA_{2n-2}^- &2& (2,2) & (2,2), (0,1), (0,0), (2n-3,0) \\
      \pA_{2n-1}' &3& (2,2)& (2,2), (0,1), (0,0), (2n-4,0), (n,1) \\
      D_{2n}' &3 & (2,2) & (2,2), (0,2), (0,0), (2n-4,0), (n,1) \\
      \hline
    \end{array}
\end{displaymath}
  \caption{Polytopes for the toric primed shapes}
  \label{tab:primed-toric-shapes}
\end{table}

\begin{lemma}\label{lemma:primed-toric}
  The shapes $\pA_{2n-1}$, $\pA_{2n-2}^-$, $D'_{2n}$, $\pA_{2n-1}'$
  are toric and can be represented by the polytopes listed in
  Table~\ref{tab:primed-toric-shapes} and illustrated in
  Figs.~\ref{fig:pA} and~\ref{fig:pAp}. 
\end{lemma}
\begin{proof}
  In these cases we can choose the ideals $I_i$ to be torus invariant,
  with $\Supp I_i$ corresponding to vertices of the polytopes of the
  pure shapes $A_{2n-1}$, $A_{2n-2}^-$, $D_{2n}$, and $I_i$ pointing
  in the directions of the respective sides. Then the blown up surface
  $Y'=\oY'$ is also toric, for the polytope obtained from the old
  polytope by cutting corners, as in
  Table~\ref{tab:primed-toric-shapes} and Fig.~\ref{fig:pAp}.
\end{proof}

\begin{figure}[htp!]
  \centering
  \includegraphics{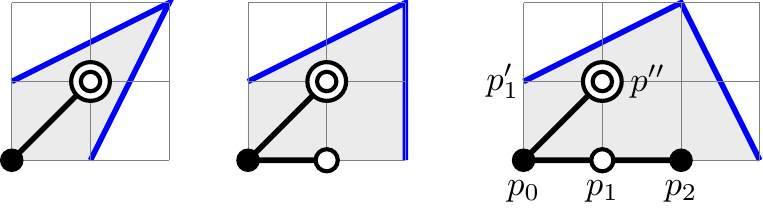} 
  \caption{Toric $\pA$ shapes: $\pA_2^- = \mA_2' = \mD_2^-$, $\pA_3$, $\pA_4^-$}
  \label{fig:pA}
\end{figure}

\begin{figure}[htp!]
  \centering
  \includegraphics{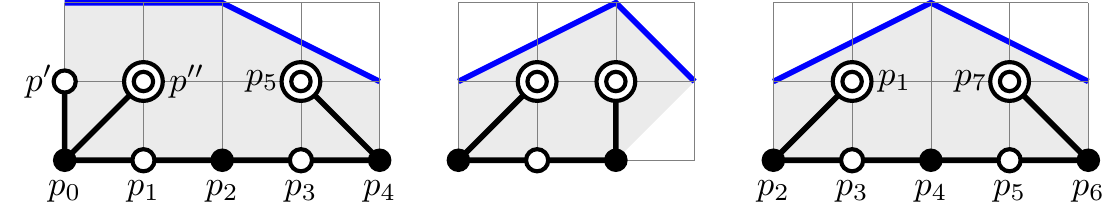}
  \caption{Toric $D'$ and $\pA'$ shapes: $D'_8$, $\pA'_5$, $\pA'_7$}
  \label{fig:pAp}
\end{figure}

\begin{remark}
  For other primed shapes, the surfaces are generally not toric but
  toric surfaces do appear for certain special directions of the ideals being
  blown up.
  \ifshortversion
  \else
  Some of them are shown in Fig.~\ref{fig:toric-degns}.
  \fi
\end{remark}

\ifshortversion
\else
\begin{figure}[htp!]
  \centering
  \includegraphics{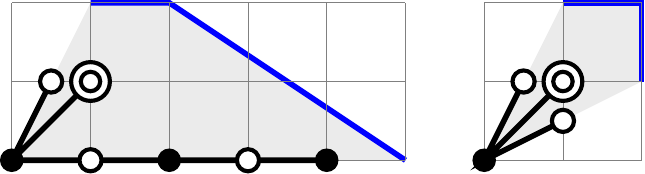}
  \caption{Some special toric surfaces in shapes $\php
    D_{7}^{-}$, $\php D_4^{\p}$ }
  \label{fig:toric-degns}
\end{figure}
\fi

\subsection{Singularities of \ade and \wade surfaces}
\label{sec:singularities}

\begin{theorem}[Singularities] \label{thm:singularities}
  Let $(Y,C)$ be a surface of pure shape $\ne A_1, \wA_1^*$. 
  With notation as in Definition \ref{def:priming},
  when priming $(Y,C)$ to $(\oY',\oC')$ the
  only curves contracted by $g\colon Y'\to \oY'$
 are:
  \begin{enumerate}
  \item The sides $C'_s$ with $L'C'_s=0$. These contract to
    $\nonklt(\oY',\oC')$.
  \item A collection of $(-2)$ curves disjoint from $C'$. These
    contract to Du Val singularities disjoint from $\nonklt(Y,C)$.
  \end{enumerate}
\end{theorem}
\begin{proof}
  Let $E'$ be a curve with $L'E'=0$. 
  As in the proof of Theorem~\ref{thm:priming}, if $K_Y+L$ is nef and
  $E'\ne C'_s$ then $(K_{Y'}+L')E'=0$, so $K_{Y'}E'=0$. Since $E'$ is
  disjoint from the boundary, it lies in the smooth part of $Y'$. We
  have $E'{}^2<0$, and by the genus formula the only possibility is
  $E'\simeq\bP^1$ with $E'{}^2=-2$. 
\end{proof}

\begin{corollary}\label{cor:generic-singularities}
  The singularities of the \ade and \wade surfaces $(Y,C)$ lying in the
  nonklt locus of $(Y,C)$ depend only on the shape and are those
  listed in the last column of Tables~\ref{tab:primed-III} and
  \ref{tab:primed-II}.
\end{corollary}

\begin{notation}
In Tables~\ref{tab:primed-III} and \ref{tab:primed-II}  
we use the following notation for
singularities. We denote simple nodes by the usual
$A_1$. For cyclic quotient singularities, whose resolutions
are a chain of curves, we use the notation $(n_1,n_2,\ldots,n_k)$,
where $-n_i$ is the self-intersection number of the $i$th curve in the
chain; note that $(2,2,\ldots,2)$ corresponds to the Du Val singularity $A_n$.
For more complicated singularities, whose resolution is not
necessarily a chain of curves, we use the following notation:
$(n_1,n_2,\ldots,n_k;2^2)$ denotes a singularity obtained by
contracting a configuration of exceptional curves with the first dual
graph in Fig.~\ref{fig:sings}. Note that this includes Du Val singularities 
of type $D_n$, which are denoted by $(2,2,\ldots,2;2^2)$.

\begin{figure}[htp!]
  \centering
\begin{center}
  \includegraphics{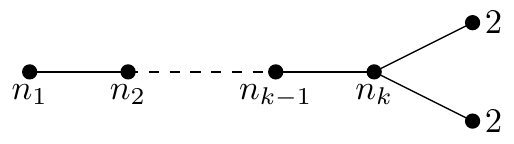}\kern 10pt
  \includegraphics{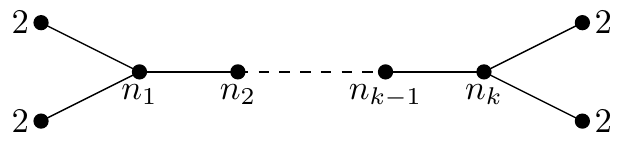}
\end{center}
  \caption{Singularities $(n_1,n_2,\ldots,n_k;2^2)$ and $(2^2;n_1,n_2,\ldots,n_k;2^2)$}
  \label{fig:sings}
\end{figure}
Finally, we will use the expression $(2^2;n_1,n_2,\ldots,n_k;2^2)$ to denote a
singularity obtained by contracting a configuration of exceptional
curves with the second dual graph in Fig.~\ref{fig:sings}.
Two apparently degenerate cases of this notation are $A_1 = (2)$ 
and $(n;2^2) = (2,n,2)$; we nonetheless use both notations,
as it is useful to make a distinction when we discuss double covers. We will also
often use $(n;2^4)$ in place of $(2^2;n;2^2)$.
Separately note that for $n=1$ the ``singularities'' $(n)$ and $(n,2)$
are in fact smooth points.
\smallskip

For completeness, we also note the corresponding singularities on the double covers.
The double cover of a simple node $A_1$ is always a smooth point, and the double
cover of a cyclic quotient singularity $(n_1,n_2,\ldots,n_k)$ is always a pair of 
cyclic quotient singularities with the same resolution; this explains why we 
draw a distinction between $A_1$, which has smooth double cover,
and $(2)$, which has double cover a pair of $(2)$ singularities.

The double cover of a singularity of type $(n_1,n_2,\ldots,n_k;2^2)$ is a
cyclic quotient singularity
$(n_1,n_2,\ldots,n_{k-1},2n_k-2,n_{k-1},\ldots,n_1)$; this explains
the second degenerate piece of notation, as $(2,n,2)$ has double cover
a pair of $(2,n,2)$ singularities, and $(n;2^2)$ has double cover a
single $(2n-2)$ singularity. Finally, the double cover of a
$(2^2;n_1,n_2,\ldots,n_k;2^2)$ singularity, for $k \geq 2$, is a cusp
singularity whose resolution is a cycle of rational curves with the
negatives of 
self-intersections
$(2n_1-2,n_2,\ldots,n_{k-1},2n_k-2,n_{k-1},\ldots,n_2)$
ordered cyclically, and the double cover of an $(n;2^4)$ singularity
is a simple elliptic singularity whose resolution is a smooth elliptic
curve with the minus self-intersection $2n-4$.
\end{notation}

\subsection{Recovering a precursor of pure shape} 
\label{sec:reconsruction}

The aim of this subsection is to explore to what extent the priming operation
is reversible. In other words, given an \ade or \wade surface of primed shape, 
can we uniquely recover the 
surface of pure shape from which it was obtained by priming?

\begin{lemma}[Non-redundancy]\label{lem:non-redundancy}
  When distinguishing the left and right sides, the only redundant
  case in the decorated Dynkin symbol notation for the shapes is 
  $\pA^-_2 = \mA'_2$, for which
  also a symmetric but degenerate notation $\mD^-_2$ may be used.
  (See Remark~\ref{rem:ambiguity}. Recall also that $A_1' = A_0^-$, $\pA_1=\mA_0$,
  and $(\wA_1^*)' = \wA_0^-$; for this reason
  we do not allow primings of $A_1$ and $\wA_1^*$.)

  When not distinguishing the left and right sides, there are also the
  cases coming from the $\bZ_2$ symmetry of $A_{2n-1}$, $\mA_{2n-3}^-$, $\mD_2^-$,
  $D_4$, and $\mE_6^-$: $\pA_5 = A_5'$, $D_4'=\pD_4$, $E_6^-=\mE_6$, etc., including
  $\ppA_2^- = \pA_2^+$. (See Remarks~\ref{rem:left-right} and \ref{rem:ambiguity}.)
\end{lemma}

\begin{proof}
  By Tables~\ref{tab:primed-III} and \ref{tab:primed-II}, most of the
  shapes are already distinguished by the main invariants and
  singularities. The only exception is $D_{2n}'$ and $\pD_{2n}$ for
  $2n\ge6$. However, in these cases the sheaf $K_Y+L$ gives a
  $\bP^1$-fibration. The left side $C_1$ is a bisection of this
  fibration and $C_2$ lies in a fiber, so the two primings are not
  isomorphic.  
\end{proof}

\begin{definition}\label{def:CBperp}
  Let $f\colon\wY\to Y$ be the minimal resolution of an \ade or \wade
  surface $(Y,C)$. Let $\wC_s$ be the strict transforms on $\wY$ 
  of the components $C_s$ of $C$, and let
  $F_i$ be the $f$-exceptional curves. Let
  $\cperp := \la C_s, F_i\ra^\perp \subset \Pic \wY$ and let 
  $\Lambda_0 := C^\perp \cap B^\perp$.  Denote by $\Delta^{(2)}_0$ the set of
  $(-2)$ vectors in $\Lambda_0$, by $\Lambda_0^{(2)}$ the root system
  generated by them, and by $W_0 = W(\Delta^{(2)}_0)$ the
  corresponding Weyl group.  Since $B^2>0$, the lattice $\Lambda_0$ is
  negative definite, and $\Lambda_0^{(2)}$ and $W_0$ are of \ade type.
\end{definition}

\begin{theorem}\label{thm:reconstruction}
  For a surface $(\oY',\oC')$ of a primed shape, its pure shape
  precursor $(Y,C)$, from which it comes by priming, is defined up
  to the action of $W_0$. The group $W_0$ is trivial
  except for the following shapes:
  \begin{enumerate}
  \item For $2n\ge6$, for $D_{2n}$ and $D_{2n-1}^-$ with
    $k$ primes on the left and any number of primes on the right, 
    and for $\wD_{2n}$ with $k$ primes one has 
    $W_0 = W(A_1^k) = S_2^k$. 
  \item the following exceptional shapes of genus 1:
  \begin{displaymath}
    \begin{array}[h]{l|llllllllll}
      \text{shape} &\pA'_3 &\pA_3'' &D'_4 &D''_4 &\pD'_4 &\pD''_4
      &\wD_4 &\wD_4' &\wD_4'' &\wD_4''' \\
      \hline
      \Lambda_0^{(2)} &A_1 &A_1^2 &A_1 &A_1^2 &A_2 &A_3 &A_1 &A_2 &A_3 &D_4
    \end{array}
  \end{displaymath}

  \end{enumerate}
  
  For the \ade shapes for a generic surface of the given shape the Weyl
  group $W_0$ acts freely on the choices of a precursor, and for the
  $\wD$ shapes it acts with a degree 2 stabilizer.
  For a generic surface of the given shape there are no singularities
  outside the set $\nonklt(\oY',\oC')$. For special surfaces there may exist
  additional Du Val singularities for all the \ade root sublattices of
  $\Lambda_0^{(2)}$, and all of these appear.

  In addition, for the exceptional case $\ppA_2^- = \plA'_2$ of
  Lemma~\ref{lem:non-redundancy} one has $W_0=0$, and there are two
  choices for the $\mA_2$ precursors, and only one choice for $A_2^-$.
\end{theorem}

\begin{example}
  For $\ppD_6$ one has $W(A_1^2)=S_2^2$, and generically there are 4
  choices for a precursor of shape $D_6$. For special choices of the
  directions of priming ideals $I_i$ the surfaces may have additional
  singularities of types $2A_1$ or $A_1$.

  For $\wD_4'''$ one has $|W(D_4)|=192$, and generically there are 96
  choices for a precursor of shape $\wD_4$. For special choices of
  the directions of priming ideals $I_i$ the surfaces may have
  additional singularities of types $D_4$, $A_3$, $3A_1$, $A_2$,
  $2A_1$, $A_1$.
\end{example}

\begin{proof}[Proof of Thm.~\ref{thm:reconstruction}]
  We computed the lattice $\Lambda_0$ for every shape in
  Tables~\ref{tab:primed-III}, \ref{tab:primed-II} by a lengthy but
  straightforward computation. The root systems $\Lambda_0^{(2)}$ are
  the ones stated in (1), (2). For example, for $\ppD''_{2n}$ one has
  $\Lambda_0 = A_1^2 \oplus\la-4\ra$, and the root system is
  $A_1^2$. We skip the details.

  We find the precursors and singularities separately but then confirm
  that the answer is the same as above. 
  Let $f\colon Y'\to Y$ be the first step in the priming, before the
  contraction $Y'\to\oY'$ (see Definition~\ref{def:priming}).
  Let $E'\ne C'_s$ be a curve with $L'E'=0$ and $E$ its image on $Y$.
  As in the proofs of Theorem~\ref{thm:priming},
  \ref{thm:singularities}, one must have $(K_Y+L)E=0$, and 
  such a curve may only exist in 
  \begin{enumerate}
  \item $D_{2n}$, $D_{2n-1}^-$ shapes for $2n\ge6$, where $K_Y+L$
    gives a $\bP^1$-fibration over $\bP^1$,
  \item the shapes of genus 1, where $K_Y+L=0$.
  \end{enumerate}
  Let us consider the case (1).  The only possibilities for $E$ are
  the fibers of the $\bP^1$ fibration. Let $P_i\in C_1$, and $I_i$
  with $\Supp I_i=P_i$ be an ideal appearing in the priming. Let $E$
  be a fiber of the $\bP^1$ fibration passing through $P_i$. If the
  direction of $I_i$ is generic, namely it is not the direction of $E$
  then on the blowup $Y'$ the preimage $f\inv(E)$ consists of two
  curves: the strict preimage $E' = f_*\inv(E)$ and the exceptional
  divisor $F$. Both of them are $\bP^1$, and one has
  $(E')^2=F^2 = -\frac12$ and $E'C = FC=1$. Contracting either $F$ or $E'$
  gives a pure shape precursor, so we get two choices.
  On the other hand, if $I_i$ has the direction $E$, i.e. $I_i\supset
  I(E)$ then $f\inv(E) = E'\cup F$, $E'$ lies in the smooth part of
  $Y'$, and one has $(E')^2=-2$ and $E'C=0$. The linear system $|L'|$
  contracts $E'$ to an $A_1$ singularity. Thus, in this case there is
  one precursor and $\oY'\setminus\oC'$ has an extra $A_1$
  singularity.

  In the case (2) for any curve $E\subset Y$ one has
  $E'\cdot f^*(K_Y+L)=0$.  The shapes of genus 1 are $A_3$, $A_2^-$,
  $\mA_1^-$, $D_4$, $\wD_4$ and those obtained from these by priming.
  For all of them the minimal resolution $\wY$ is a weak del Pezzo
  (i.e. with big and nef $-K_\wY$) of degree 2, 4, 6, or 8.
  To analyze both possible precursors and singularities
  we computed the graphs of $(-1)$ and $(-2)$ curves on the minimal
  resolution of singularities $\wY$. These graphs are classically known, see 
  e.g. \cite[Ch.8]{dolgachev2012classical-algebraic}. 
  The answers are the same as given in the statement of the Theorem.

  The exceptional case $\ppA_2^- = \plA'_2$ of genus~1 is treated in
  the same way. 
\end{proof}

\section{Classification of nonklt log del Pezzo surfaces of index 2}
\label{sec:nakayama}

\ifshortversion
The purpose of this Section is to prove Theorem~\ref{thm:logdP=ade}
from the Introduction.
\else
The purpose of this Section is to prove:
\begin{customthm}{A}
  The log canonical non-klt del Pezzo surfaces $(Y,C)$ with $2(K_X + C)$ Cartier 
  and $C$ reduced (or possibly empty) are exactly the same as the \ade and \wade
  surfaces $(Y,C)$, pure and primed.
\end{customthm}
\fi

Log del Pezzo surfaces with boundary $(Y,C)$ such that $-2(K_Y+C)$ is
ample and Cartier were classified by Nakayama in
\cite{nakayama2007classification-of-log}, over fields of arbitrary
characteristic. Some work is still required to extract
Theorem~\ref{thm:logdP=ade} from his classification.  First, in
\cite{nakayama2007classification-of-log} the divisor $C$ is
half-integral, and in our case it should be integral.

Secondly, the case of genus $g=1$ in
\cite{nakayama2007classification-of-log} is reduced to classifying log
canonical pairs $(Y,C)$ such that $Y$ is a Gorenstein del Pezzo
surface and $C$ is an effective Weil divisor with $-K_Y\sim 2C$.  The
classification of such pairs is not provided.  Rather than trying to
perform such a classification, we adapt the arguments from other parts of
\cite{nakayama2007classification-of-log} to deal with this case.

\smallskip
{For ease of the use of \cite{nakayama2007classification-of-log},
  for this section only, we adopt the notation of the latter paper.}
The basic setup is as follows.  The log del Pezzo surface with boundary is denoted
$(S,B)$, versus our $(Y,C)$.
At the outset, let us mention an important general result
\cite[Cor.3.20]{nakayama2007classification-of-log} generalizing that
of \cite[Thm.1.4.1]{alexeev2006del-pezzo}:
\begin{theorem}[Smooth Divisor Theorem]\label{thm:smooth-divisor}
  Let $(S,B)$ be a log del Pezzo surface with boundary of index $\le2$. Then a general
  element of the linear system $|-2(K_S+B)|$ is smooth.
\end{theorem}

By \cite[3.16, 3.10]{nakayama2007classification-of-log}, the only pairs with
irrational $S$ and integral $B$ are cones over elliptic curves which
we call  $\wA_{2n-1}$.  So below we assume that $S$ is rational.
The minimal resolution of
singularities of $S$ is denoted by $\alpha\colon M\to S$. One defines:
\begin{enumerate}
\item An effective $\bZ$-divisor $E_M$ on $M$ by the formula
  $K_M = \alpha^*(K_S+B) -\frac12 E_M$. Since we assume the pair
  $(S,B)$ to be lc, $E_M$ has multiplicities 1 and 2. If $B=0$ and $S$
  is log terminal then $E_M$ is reduced. Otherwise, there is at least
  one component of multiplicity $2$. 
\item A big and nef line bundle $L_M= \alpha^*(-2(K_S+B))$. Thus, one
  has $L_M = -2K_M-E_M$.
\item The \emph{genus} $g(S,B) = \frac12 (K_M+L_M)L_M + 1$. This is the
  genus of a general element of $|-2(K_S+B)|$.
\end{enumerate}

This is the standard notation used in
\cite{nakayama2007classification-of-log}: 
\begin{itemize}
\item On $\bP^2$, a line is denoted by $\ell$.
\item On $\bF_n$, a zero section is $\sigma$, an infinite section
  $\sigma_\infty$, and a fiber $\ell$.
\item On $\bP(1,1,n)$, $\bar\ell$ is the image of a fiber from
  $\bF_n$, i.e. a line through the $\frac{1}{n}(1,1)$ singular point $(0,0,1)$. $\bar
\sigma_\infty$ is the image of $\sigma_{\infty}$ on
$\bP(1,1,n)$; note that $\bar \sigma_\infty \sim n\bar\ell$.
\end{itemize}

The classification of log del Pezzo surfaces with boundary is divided into three cases:
\begin{enumerate}
\item $K_M+L_M$ is not nef.
\item $K_M+L_M$ is nef and $g\ge2$.
\item $K_M+L_M$ is nef and $g=1$. 
\end{enumerate}

\subsection{ The case $K_M+L_M$ is not nef }
\label{sec:KMLM-notnef}

By \cite[3.11]{nakayama2007classification-of-log}, the only cases for
us are:
\begin{enumerate}
\item $S=\bP^2$, $\deg B=2$ $\implies$ $B$ is a smooth conic (our
   $\wA_1^\exo$) or two lines ($A_1$). 
\setcounter{enumi}{2}
\item $S=\bP(1,1,n)$, $n\ge2$ and $B\in |(\frac{n}2+2)\bar\ell|$, in
  particular $n$ is even.
\end{enumerate}
In the latter case, note that the smallest divisor not passing through the  singular point $(0,0,1)$ is $\bar
\sigma_\infty \sim n\bar\ell$. We consider the subcases:

\begin{enumerate}
\renewcommand{\theenumi}{\alph{enumi}}
\item $B\not\ni (0,0,1)$. We need $\frac{n}2+2\ge n \implies n=2,4$. If
  $n=2$ then $B\in |3\bar\ell|$ is not Cartier, a contradiction. If $n=4$
  then $B\in |4\bar\ell|=|\cO(1)|$, $L_M=\cO(1)$. This is a degenerate
  subcase of  $\wA_1^\exo$,  when $\bP^2$ degenerates to
  $\bF_4^0=\bP(1,1,4)$  (see Subsection~\ref{sec:wA-shapes}(2)). 
\item $B\ni (0,0,1)$  and is smooth there. The strict preimage of $B$ is
  then $\wB\sim \ell + k\sigma_\infty$ for some $k\ge0$. Then $B \sim
  (1+kn)\bar\ell$ $\implies \frac{n}2+2=1+kn$. It follows that $n=2$ and
  $k=1$. If $B$ is irreducible then this is our  $\wA_0^-$  case; if
  $B=\bar\ell+\bar\sigma_\infty$ then this is  $A_0^-$. 
\item $B\ni (0,0,1)$ and has two branches there. Then $\wB\sim
  2\ell+k\sigma_\infty$ and $B\sim (2+kn)\bar\ell \sim
  (\frac{n}2+2)\bar\ell$. This is impossible. 
\end{enumerate}

\subsection{ $K_M+L_M$ is nef and $g\ge2$ }
\label{sec:g-ge2}

Nakayama defines a \emph{basic pair} to be a projective surface $X$
and a \emph{nonzero} effective $\bZ$-divisor $E$ so that, for
$L=-2K_X-E$ one has:
\begin{enumerate}
\item[($\cC1$)] $K_X+L$ is nef, 
\item[($\cC2$)] $(K_X+L)L=2g-2>0$, 
\item[($\cC3$)] $LE_i\ge0$ for any irreducible component $E_i$ of $E$.
\end{enumerate}
So, the minimal resolution of a log del Pezzo surface with boundary of index $\le2$ is
a basic pair, unless $B=0$ and $S$ has Du Val singularities (because
then $E=0$). Vice versa, by
\cite[3.19]{nakayama2007classification-of-log}, any basic pair is the
minimal resolution of a log del Pezzo surface with boundary of index $\le2$, with the
semiample line bundle $NL$, $N\gg0$, providing the contraction.

The next step is to run MMP for the divisor $K_X+\frac12 L$. Namely,
if for some $(-1)$-curve $\gamma$ one has $(2K_X+L)\gamma=-E\gamma<0$
then $L\gamma=E\gamma=1$, the curve $\gamma$ can be contracted
$\tau\colon X\to Z$ to obtain a new basic pair $(Z, E_Z)$, and one has
$K_X+L=\tau^*(K_Z+L_Z)$, $K_X+E=\tau^*(K_Z+E_Z)$.  Here,
$E_Z = \tau_*(E)$ and it is again nonzero.

The minimal basic pairs, without the $(-1)$-curves as above are
$\bP^2$ and $\bF_n$, and it is easy to list the possibilities for $E$
on them. Nakayama proves that the morphism $\phi\colon M\to X$ to a
minimal basic pair is a sequence of blowups of the simplest type which
can be conveniently locally encoded by a zero-dimensional subscheme
$\Delta$ of a smooth curve, i.e. a subscheme given by an ideal
$I=(y, x^k)$ for some local parameters $x,y$ and $k>0$. If
$\mu\colon Y\to X$ is a simple blowup then
$I_Y=\mu^*I = (y,x^{k-1}) \otimes \cO_Y(-\Gamma)$, where $\Gamma$ is
the exceptional $(-1)$-curve of $\mu$. Then one continues to eliminate
$I_Y$ by induction, making $k$ blowups in total. Equivalently, one can
blow up the ideal $I$ and then take the minimal resolution.

In this way, we obtain a triple $(X,E,\Delta)$ satisfying
\begin{enumerate}
\item[($\cF1$)] $(X,E)$ is a minimal basic pair, $L=-2K_X-E$.
\item[($\cF2$)] $\Delta$ is empty or a zero-dimensional subscheme of
  $X$ which is locally a subscheme of a smooth curve,
\item[($\cF3$)] $\Delta$ is a subscheme of $E$ considered as a
  subscheme of $X$ (recall that $E$ is an effective Cartier divisor
  with multiplicities 1 or 2) such that for every reduced
  irreducible component $E_i$ of $E$ one has $LE_i\ge \deg(\Delta\cap
  E_i)$. 
\end{enumerate}
Nakayama calls these \emph{quasi fundamental triplets}. Vice versa, by
\cite[4.2]{nakayama2007classification-of-log} for any quasi
fundamental triplet $(X,E,\Delta)$ the pair $(M,E_M)$ obtained by
eliminating $\Delta$ is a basic pair, that is the minimal resolution
of singularities of a log del Pezzo surface with boundary. Thus, one is reduced to
enumerating quasi fundamental triplets.

For a given basic pair $(M,E_M)$, the sequence of blowdowns of
$(-1)$-curves and thus the resulting quasi fundamental triplet
$(X,E,\Delta)$ are not unique. To cure this, Nakayama defines a
\emph{fundamental triplet} that satisfies additional normalizing
conditions \cite[Def. 4.3]{nakayama2007classification-of-log}. He then
proves in \cite[4.9]{nakayama2007classification-of-log} that the
fundamental triplet exists and is unique in most cases, including all
the cases when $(S,B)$ is strictly log canonical -- the case that we
operate in.  For this case, the possible fundamental triplets are
listed in \cite[4.7(2)]{nakayama2007classification-of-log}.

It remains to consider these fundamental triplets and the resulting
minimal resolutions $M$. But first, we can narrow down the
possibilities for $\Delta$ since our situation is restricted by the
condition that $B$ is integral and not half-integral as in
\cite{nakayama2007classification-of-log}.

\begin{definition}\label{def:subschemes}
  We introduce the following simple subschemes $\Delta\subset E$.
  \begin{center}
    \begin{tabular}[htp!]{l|llcc}
      &$E$ &$\Delta$ &$\deg(\Delta)$ &$\mult_P(\Delta\cap E_i)$ \\
      \hline
      ($\ \cdot\ $)\ & $(y)$ &$(y,x)$ &1 &$1$ \\
      ($-$)$_1$ & $(y)$ & $(y,x^2)$  &2 &$2$ \\
      \hline
      ($-$)\  & $(y^2)$ & $(y,x^2)$  &2 &$2$ \\
      ($\ \p\ $)\  & $(y^2)$ & $(y^2,x)$ &2 &$1$ \\
      ($+$)\  & $(y^2)$ & $(y^2,y+\epsilon x^2)$, $\epsilon\ne0$ &4 &$2$ \\
    \end{tabular}
  \end{center}

  An alternative description for the last subscheme is
  $(y+\epsilon x^2, x^4)$.
\end{definition}

The subschemes appearing in this definition are given suggestive
names, which reflect the notation used for priming in
Section~\ref{sec:priming}. The reason for this will become clear in
the proof of Theorem~\ref{thm:pairs-large-genus}. 

\begin{lemma}
  The effect of eliminating the subschemes of \eqref{def:subschemes}
  is as follows. 
  \begin{enumerate}
  \item[($\ \cdot\ $)\ ] $E_M=1E_i + 0\Gamma_1$, $\Gamma_1^2=-1$,
    $E_i\Gamma_1=1$, $L_M\Gamma_1=1$.
  \item[($-$)$_1$] $E_M=1E_i + 0\Gamma_1 +0\Gamma_2$, $\Gamma_1^2=-1$,
    $\Gamma_2^2=-2$, $E_i\Gamma_1=\Gamma_1\Gamma_2=1$, $L_M\Gamma_1=1$.
  \item[($-$)\ ] $E_M=2E_i + 2\Gamma_1 +1\Gamma_2$, $\Gamma_1^2=-1$,
    $\Gamma_2^2=-2$, $E_i\Gamma_1=\Gamma_1\Gamma_2=1$, $L_M\Gamma_1=1$.
  \item[($\ \otherbar\ $)\ ] $E_M=2E_i + 1\Gamma_1 +0\Gamma_2$, $\Gamma_1^2=-2$,
    $\Gamma_2^2=-1$, $E_i\Gamma_1=\Gamma_1\Gamma_2=1$, $L_M\Gamma_2=1$.
  \item[($+$)\ ] $E_M=2E_i + 2\Gamma_1 +1\Gamma_2
    +1\Gamma_3 + 0\Gamma_4$, $\Gamma_1^2=\Gamma_2^2=\Gamma_3^2=-2$,
    $\Gamma_4^2=-1$, $E_i\Gamma_1= \Gamma_1\Gamma_2= \Gamma_1\Gamma_3=
    \Gamma_3\Gamma_4=1$, $L_M\Gamma_4=1$.
  \end{enumerate}
  It is pictured in Fig.~\ref{fig:subschemes}.
\begin{figure}[htp!]
  \centering
  \includegraphics[width=22pc]{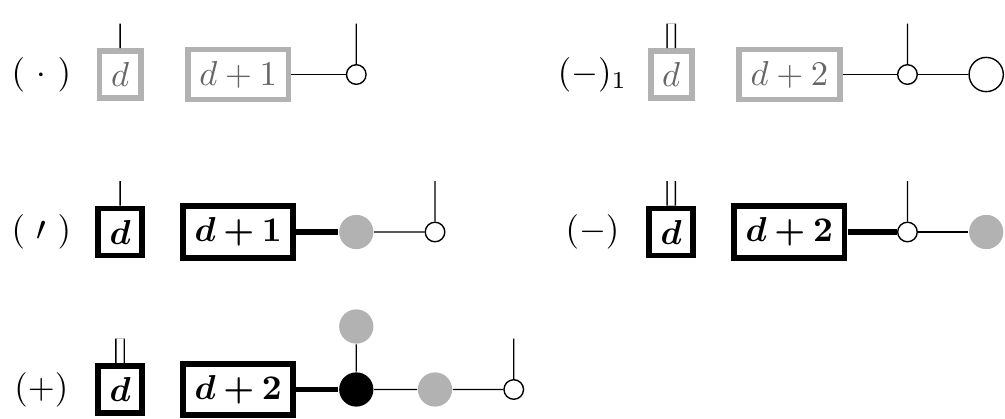}
  \caption{Effect of eliminating simple subschemes}
  \label{fig:subschemes}
\end{figure}
\end{lemma}
\begin{proof}
  This is direct computation, following
  \cite[Sec.2]{nakayama2007classification-of-log}. 
\end{proof}

\begin{notation}
  In Fig.~\ref{fig:subschemes}, the rectangle with label ``$d$''
  denotes an irreducible component $E_i$ of $E$ with $E_i^2=-d$. The
  small nodes are $\bP^1$'s of square $(-1)$, the large ones of
  square $(-2)$. Rectangles and nodes are shown in bold black, resp. gray
  or white, if they appear in $E_M$ with multiplicity $2$, resp. $1$
  or $0$. The half-edges denote $\mult_P(\Delta\cap E_i)$,
  which are 2 (double line) or 1 (single line). When
  we are working with a geometric triple
  $(X, B + \frac{1+\epsilon}2 D)$, where $D\in |-2K_S - E|$ is a
  section, these half edges are the local intersection numbers $DE_i$ at a
  point $P\in D\cap E_i$. The double edge means that $D$ is tangent
  to $E_i$ at $P$.
\end{notation}

The following lemma is a direct consequence of a proof from \cite{nakayama2007classification-of-log}.

\begin{lemma}
  The pair $(S,B)$ is log canonical iff for every irreducible component $E_i$ of $E$ in the fundamental triplet  $(X,E,\Delta)$,  one has
  $\mult_E(E_i)\le2$, $\Delta$ is disjoint from the nodes of the
  double part $\llcorner E\lrcorner$ of $E$, and $\mult_P(\Delta\cap E_i)\le2 $
  for every irreducible component $E_i$ with $\mult_E(E_i)=~2$ and
  all $P\in\Delta$.
\end{lemma}
\begin{proof} Follows immediately from the proof of \cite[Cor. 4.7]{nakayama2007classification-of-log}.\end{proof}

\begin{theorem}\label{thm:only-five-subschemes}
  Let $(M,E_M)$ be a basic pair with $M$ the minimal resolution of
  singularities of a strictly log canonical log del Pezzo surface with boundary $(S,B)$
  of index $\le2$ with integral $B$, and let $\phi\colon M\to X$ be a
  contraction to a minimal basic pair so that $(M,E_M)$ is obtained
  from a quasi fundamental triplet $(X,E,\Delta)$ by eliminating the
  0-dimensional scheme $\Delta$. Then 
  \begin{enumerate}
  \item If a component $E_i$ of $E$ has multiplicity 1 then its strict
    preimage on $M$ must be isomorphic to $\bP^1$ and have $E_i^2\le -2$. 
  \item Additionally, assume that $\Delta$ is disjoint from the singular
    part of $E_{\rm red}$ and that for every irreducible component
    $E_i$ of $E$ with $\mult_E(E_i)=1$, one has $\mult_P(\Delta\cap
    E_i)\le 2$. 
    Then the only connected components of
    $\Delta$ are the five subschemes of Def.~\ref{def:subschemes}.
  \end{enumerate}
\end{theorem}
\begin{remark}
  Concerning the additional assumptions of (2), we note that they are
  satisfied for the strictly log canonical fundamental triplets by
  \cite[4.6]{nakayama2007classification-of-log}.
  So we can ignore them in the case $g(S,B)\ge2$.
\end{remark}

\begin{proof}
  (1)   Our condition for the integrality of $B$ means that all components
  of $E_M$ of multiplicity 1 must be contracted by
  $\alpha\colon M\to S$. They are all $\bP^1$'s with $E_i^2\le -2$.

  (2) We then go through the short list of subschemes with
  $\mult_P(\Delta\cap E_i)\le2$, eliminating those that lead to
  $(-1)$-curves $\Gamma$ with $\mult_{E_M}(\Gamma)=1$. For example,
  the case $\Delta=(x,y) \subset E=(y^2)$ is eliminated.
\end{proof}

Nakayama defined fundamental triplets $(X,E,\Delta)$ (without
``quasi'') in order to obtain uniqueness for them, in most cases. 
We pick a different normalization: we pick $(X,E)$ to
correspond to one of the pure shapes and all connected components of
$\Delta$ to be of type $(\otherbar)$.

\begin{theorem}\label{thm:pairs-large-genus}
  Let $(S,B)$ be a log del Pezzo surface with boundary $(S,B)$ of index $\le2$ of genus
  $g(S,B)\ge2$. Then it is one of the following shapes or is obtained
  from them by any allowable primings as in
  Theorem~\ref{thm:priming}. 
  \begin{enumerate}
  \item  $\wD_{2n}$, $D_{2n}$, $D_{2n-1}^-$, $A_{2n-1}$,
    $A_{2n-2}^{-}$, $\phmi A_{2n-3}^-$ for $2n\ge 6$.  
  \item  $\wE_7$, $\phmi E_7$, $\phmi E_6^-$. 
  \item  $\wE_8^-$, $\phmi E_8^-$. 
  \end{enumerate}
\end{theorem}
\begin{proof}
  We go through the complete list
  \cite[4.7(2)]{nakayama2007classification-of-log} of fundamental
  triplets and see that they are as above. 

  \emph{Case $[n; 2, e]_2$ for $n\ge0$, $e\le \max(4, n+1)$ with
    $\mult_\ell F\le 2$ for any $\ell \le F$.} This means that
  $X=\bF_n$ and $E=2\sigma + F$, where $F\sim e\ell$ is a sum of
  several fibers, each with multiplicity $\le2$, and
  $\Delta\cap\sigma = \emptyset$.
  We have $L\sim 2\sigma_\infty + (4-e)\ell$ and $L\sigma=4-e$. 

  If $e=0$ then $\Delta=\emptyset$. This is $\wD_{2n+8}$, so we obtain
   $\wD_{2m}$ for $2m\ge8$. 

  If $e=1$ then we must have $\Delta=(\cdot\cdot)$, that is two disjoint copies
  of $(\cdot)$ contained in a fiber $F$, or $(-)_1$ which is a
  degeneration of it.  Let us use the extended notation
  $[n; 2,0; \cdot\cdot]$, resp. $[n; 2,0; -_1]$ by writing $\Delta$ at the end.
  Note that we must apply $(\cdot)$ twice, otherwise $\wF$ is a
  $(-1)$-curve in $E_M$ with multiplicity~1, which is not allowed
  by Theorem~\ref{thm:only-five-subschemes}.

  Contracting one of the $(-1)$-curves back and then $F_i$, we can
  view this as the quasi fundamental triplet $[n-1; 2,0; \p]$, which
  is $\wD_{2n+6}^\p$. Thus, we get  $\wD_{2m}^\p$ for
  $2m\ge6$.  

  In the degenerate case $\Delta=(-)_1$ of $(\cdot\cdot)$, the direction of
  the ``prime'' coincides with the direction of the fiber $\ell$ on
  $\bF_{n-1}$ for the triplet $[n-1; 2,0; \p]$.
  In that case the 
  strict preimage of this fiber gives an extra
  $(-2)$-curve, and the surface $Y$ acquires an extra $A_1$
  singularity outside of $B$.

  If $e=2$ and $F=\ell_1+\ell_2$ then we get $[n-2;2,0;\pp]$ this way,
  which is $\wD_{2n+4}^{\pp}$. Since $n\ge1$, we get 
  $\wD_{2m}^{\pp}$ for $2m\ge6$.  Similarly when $e=3,4$
  and $F$ is the sum of $e$ distinct fibers, we get
   $\wD_{2m}^{\p\pp}$ and $\wD_{2m}^{\pp\pp}$ for
  $2m\ge6$.  
  Similar to the above, for every priming the preimage of the
  corresponding fiber $\ell$ gives an $(-2)$-curve which gives an
  additional singularity of $Y$. 
  
  Now consider the case when $e=2$ and $F=2\ell$ is a double fiber. If
  $\Delta=\emptyset$ then this is $D_{2n+4}$, i.e.  $D_{2m}$ for
  $2m\ge6$.  
  For $\Delta = -, \p, \pp, +$ we get  $D_{2m-1}^{-}$,
  $D_{2m}^{\p}$, $D_{2m}^{\pp}$, $D_{2m-1}^{+}$ for $2m\ge6$. 
  Adding single fibers to $F$, i.e. $F=2\ell+\ell_1$ or
  $2\ell+\ell_1+\ell_2$, gives priming on the left side, which produces all the
  cases  $\php D^{?}$ and $\phpp D^{?}$ for $2m\ge6$. 

  Finally, $e=4$, $F=2\ell_1+2\ell_2$ and $\Delta=\emptyset$ gives
   $A_{2n-1}$.  Adding $\Delta = -, \p, \pp, +$ adds corresponding
  decorations in the $A$ case, with each $-,+$ decreasing
  the index by~1.

  \smallskip
  \emph{Case $[1;2,2]_{2\infty}$:} $Y=\bF_1$, $E=2\sigma_\infty$, and
  $\Delta=\emptyset$. This is  $\wD_6$. 

  \smallskip
  \emph{Case $[2]_2$ with $\mult_P(\Delta\cap\ell)\le2$ for any
    $P\in\ell$:} $Y=\bP^2$, $E=2\ell$ and $L=\cO(4)$. 
  For $\Delta=\emptyset$, this is  $\wE_7$.  For $\Delta=(-)$,
  resp. $(--)$, this is  $\phmi E_7$, $\phmi E_6^-$.  
  Considering various other possibilities for
  $\Delta$ leads to all the allowable primings of $\wE_7$, $\phmi E_7$,
  $\phmi E_6^-$. 

  \smallskip
  \emph{Case $[2;1,2]_{2+}$ with $\mult_P(\Delta\cap\ell)\le2$ for any
    $P\in\ell$:} $Y=\bF_2$, $E=\sigma+2\ell$, $\deg(\Delta\cap\ell)\le
  3$ and $\Delta\cap\sigma=\emptyset$.
  For $\Delta=\emptyset$ this is  $\wE_8^-$,  and for $\Delta=(-)$
  this is  $\phmi E_8^-$.  Considering various other possibilities for
  $\Delta$ leads to all the allowable primings of $\wE_8^-$ and
  $\phmi E_8^-$.

  \smallskip
  \emph{Case $[0;2,1]_0$.} This is a typo, this is a klt case so it does
  not appear.  
\end{proof}

\subsection{ $K_M+L_M$ is nef and $g=1$ }
\label{sec:g=1}

In this case the main result of
\cite{nakayama2007classification-of-log} is (3.12) which says that $S$
must be a Gorenstein log del Pezzo surface and $2B\sim -K_S$. To apply
it in our case, we would have to find all Gorenstein del Pezzo
surfaces with Du Val singularities and $K_S$ divisible by~2 as a Weil
divisor -- of which there are many -- and then consider all the
possibilities for $B$. 

Instead, we adopt a different strategy. Let us define a \emph{weak
  basic pair} with the same definition as a basic pair but dropping
the condition $(\cC2)$ that $2g-2>0$. Similarly, we define a
\emph{weak quasi fundamental triplet} $(X,E,\Delta)$ by asking that
$X$ in $(\cF1)$ is merely a \emph{weak} minimal basic pair. Then:

\begin{enumerate}
\item It is still true that $K_M+L_M$ is nef for any weak basic pair
  obtained by eliminating a 0-dimensional scheme of a weak fundamental
  triple $(X,E,\Delta)$: the corresponding proofs in \cite[4.2, 3.14
  nefness]{nakayama2007classification-of-log} go through.
\item We have additional conditions $K_M+L_M=K_M+E_M=0$ by
  \cite[3.12]{nakayama2007classification-of-log}.
\item Our Theorem~\ref{thm:only-five-subschemes} still holds.
\item We have to check separately that $L_M$ is big, this condition is
  no longer automatic. However, this is easy to do: $L^2/2$ drops by
  $\deg(\Delta)/2$, i.e. by 1 under the operations $(\p)$, $(-)_1$,
  $(-)$, and by 2 under $(\pp)$.
\end{enumerate}

\begin{lemma}\label{lem:weak-fund-triplets}
  The weak fundamental triplets for strictly lc pairs $(S,B)$ are:
  \begin{enumerate}
  \item $X=\bP^2$, $E=2\ell_1+\ell_2$.
  \item $X=\bF_0$, and 
    (a) $E=2\sigma+2\ell$, 
    (b) $E=2\sigma+\ell_1+\ell_2$, 
    (c) $E=2D$, $D\sim\sigma+\ell$.
  \item $X=\bF_1$, and 
    (a) $E=2\sigma+2\ell_1+\ell_2$, 
    (b) $E=2\sigma+\ell_1+\ell_2+\ell_3$, \newline
    (c) $E=\sigma+\sigma_\infty+2\ell$,
    (d) $E=2\sigma_\infty + \ell$.
  \item $X=\bF_2$, and 
    (a) $E=2\sigma+2\ell_1+2\ell_2$,
    (b) $E=2\sigma+2\ell_1+\ell_2+\ell_3$, \newline
    (c) $E=2\sigma+\ell_1+\ell_2+\ell_3+\ell_4$,
    (d) $E=\sigma+2\ell+\sigma_\infty$,
    (e) $E=2\sigma_\infty$. 
  \end{enumerate}
\end{lemma}
\begin{proof}
  Immediate: $X=\bP^2$ or $\bF_n$, $L=-K_X$ must be nef, and
  $E=-2K_X-L$ must have at least one component of multiplicity 2. We
  simply list the possibilities.
\end{proof}

\begin{theorem}\label{thm:pairs-genus-one}
  Let $(S,B)$ be a log del Pezzo surface with boundary $(S,B)$ of index $\le2$ of genus
  $g(S,B)=1$. Then it is one of one of the shapes  $\wD_4$,
  $D_4$, $A_3$, $A_2^-$, $\phmi A_1^{-}$,  or is obtained from
  one of them by any allowable primings as in
  Theorem~\ref{thm:priming}.
\end{theorem}
\begin{proof}
  The pairs of Lemma~\ref{lem:weak-fund-triplets} in which all
  components of $E$ have multiplicity 2 already appear in our
  classification: (2a)  $D_4$, (2c) $\wD_4$, (4a) $A_3$, (4e)
  degenerate case of $\wD_4$.  Our first step is to reduce
  all other cases to them.

  Let us begin with case (1). The line $\ell_2$ must be blown up at
  least once by Theorem~\ref{thm:only-five-subschemes}(1). Thus, we are
  reduced to case (2). 

  Now consider for example case (2b). The fiber $\ell_1$ must be blown
  up at least once, again by~\eqref{thm:only-five-subschemes}(1).  Let
  $\tau\colon X'\to X$ be the first blowup at a point $P\in E$ and let
  $E_0$ be the exceptional $(-1)$-curve. We have
  $K_{X'}+E' = \tau^*(K_X+E) = 0$. If $P=\ell_1\cap\sigma$ then $E_0$
  appears in $E'$ with coefficient 2, otherwise it appears with
  coefficient 0; either way it is even. Let $X'\to X''$ be the
  contraction of the strict preimage of $\ell_1$, which is a
  $(-1)$-curve on $X'$. We obtain another minimal model $M\to X''$
  for $M$ which has fewer components of multiplicity 1 in $E$.

  This way, we reduce all cases to the purely even cases above except
  cases (3c) and (4d). Consider now (3c). The curve
  $\sigma_\infty$ has to be blown up at least once. Blowing up and
  contracting the strict preimage of a fiber reduces to the case (3a)
  which was already considered. The case (4d) reduces to (3c) and then to
  (3a). 

  So now we are reduced to the pairs of shapes $D_4$, $A_3$, $\wD_4$
  and the pairs obtained from them by eliminating 0-dimensional
  subschemes $\Delta$. The conditions of
  Theorem~\ref{thm:only-five-subschemes}(2) hold, so the connected
  components of $\Delta$ have types $(\p)$, $(-)$, $(+)$. In the
  cases $D_4$, $A_3$ we also have $\deg(\Delta\cap E_i)\le 2$ for
  $i=1,2$. In all three cases, $\deg(\Delta) \le 6$ by the condition
  $L_M^2>0$.

  So let us now begin with $D_4$ and consider different possibilities
  for $\Delta$. If one or two components of $\Delta$ are $(-)$ then we
  get respectively  $D_3^-=A_3^{\p}$ and
  $\phmi D_2^{-}=\phmi A_2'$.  If the components are $(+)$ then
  we get respectively  $D_3^+=A_3^{\pp}$ and
  $\phmi D_2^{+}=\phpp A_3^{-}$.  When the components of $\Delta$
  are $(\p)$, we get the usual primings.

  For $\wD_4$, $\Delta=(-)$ gives  $D_4^{\p}$  and $\Delta=(--)$ gives
   $\php A_3^{\p}$,  with other combinations of $(-)$, $(+)$, $(\p)$
  giving primings of those. 
  For $A_3$, it is easier: $\Delta=(-)$, $(-,-)$, $(+)$ etc. gives the
  usual  $A_2^-$, $\phmi A_1^{-}$, $A_2^+$  and adding
  $(\p)$'s gives the usual primings.
\end{proof}

{This completes the proof of
  Theorem~\ref{thm:logdP=ade}. We now switch back from the
  notation of \cite{nakayama2007classification-of-log} to our
  notation $\pi\colon (X,D+\epsilon R) \to (Y, C+\frac{1+\epsilon}2 B)$.}

\section{Moduli of \ade pairs}
\label{sec:moduli}

\subsection{Two-dimensional projections of \ade lattices}
\label{sec:2dim-models}

Here, we fix the notations from representation theory and prove a
number of basic results that will be used in the remainder of the
paper.

\begin{notation}
  $\Lambda$ will denote one of the root lattices $A_n$, $D_n$, $E_n$,
  and $\Lambda^*\supset\Lambda$ its dual, the weight lattice.  One has
  $\Lambda^* = \oplus_{i=1}^n \bZ \alpha_i$ and
  $\Lambda^* = \oplus_{i=1}^n \bZ \vpi_i$, where $\alpha_i$ are the
  simple roots and $\vpi_i$ the fundamental weights (same as
  fundamental coweights). One has
  $\la\alpha_i, \vpi_j\ra = \sigma_{ij}$. 
\end{notation}

\begin{notation}\label{not:roots-weights}
  We label the nodes of the Dynkin diagrams as in Figs.~\ref{fig:A},
  \ref{fig:D}, \ref{fig:E}. For example, for the $E_8$ diagram we
  denote the nodes by $p''$, $p'_1$, $p'_2$, $p_0, \dotsc, p_4$.  For
  the $D_n$ diagram they are $p''$, $p'_1=p'$, $p_0, \dotsc, p_{n-3}$.
  We use the same notation to denote the roots and fundamental
  weights, i.e. we call them $\alpha''$, $\alpha'_1=\alpha'$, etc.
    
  In addition, for each of the polytopes $P$ in Figs.~\ref{fig:A},
  \ref{fig:D}, \ref{fig:E} we have the special vertex $p_*$ and two
  vertices $p_\ell, p_r$ which are the end points of the left and right
  sides. For example, for the $E_8$ diagram one has $p_\ell=p'_3$ and
  $p_r=p_5$, and for $D_n$ they are $p_\ell=p'_2$ and $p_r=p_{n-2}$. 
\end{notation}

\begin{definition}
  We define the extended weight lattice as $\Lambda^*\oplus\bZ^2$, and
  we denote the basis of $\bZ^2$ by $\{\vpi_l, \vpi_r\}$. 
\end{definition}

\begin{lemma}\label{lem:2dim-models}
  For pure \ade shapes,
  the rule $\vpi_i \mapsto p_i-p_*$ defines a homomorphism
  \begin{displaymath}
    \phi\colon \Lambda^*\oplus \bZ\vpi_\ell\oplus\bZ\vpi_r
    \xrightarrow{\phi} \bZ^2.
  \end{displaymath}
  The projection
  $\pi_1\colon\Lambda^*\oplus \bZ\vpi_\ell\oplus\bZ\vpi_r
  \to\Lambda^*$ identifies $\ker\phi$ with $\Lambda$.  The
  homomorphism $\phi$ is surjective for $D,E$ shapes, and one has
  $\coker\phi=\bZ_2$ for $A$ shapes.
\end{lemma}
\begin{proof}
  Any root $\alpha$ can be expressed as
  $\alpha = \sum_{\vpi} \la\alpha,\vpi\ra$ with the sum going over
  the $n$ fundamental weights $\vpi$. In particular, 
  if $p_{i-1}$, $p_i$, $p_{i+1}$ are three consecutive nodes in a chain then
  \begin{equation}
    \label{eq:root-middle}
    \alpha_i = 2\vpi_i -  \vpi_{i-1} - \vpi_{i+1}
    \xrightarrow{\phi} 2p_i -  p_{i-1} - p_{i+1} = 0
  \end{equation}
  For an end node $p_i$ next to $p_r$ one has
  \begin{equation}
    \label{eq:root-end}
    \alpha_i - \vpi_r = 2\vpi_i -  \vpi_{i-1} - \vpi_{r}
      \xrightarrow{\phi} 2p_i -  p_{i-1} - p_{r} = 0
    \end{equation}
    and similarly for the node next to $p_\ell$. For a  node
    $p_0$ occurring at a corner of the polytope, one has
  \begin{equation}
    \label{eq:root-corner}
    \alpha_0 = 2\vpi_0 - \vpi_1 - \vpi_1' - \vpi''
    \xrightarrow{\phi} 2p_0 - p_1 - p_1' - p'' + p_* = 0
  \end{equation}
  Thus, $\Lambda = \la\alpha\ra \subset \ker\phi$, and it is easy to
  see that the equality holds. 
\end{proof}

Recall that the finite group $\Lambda^*/\Lambda$ is $\bZ_{n+1}$ for
$A_n$, $\bZ_2^2$ for $D_{2n}$, $\bZ_4$ for $D_{2n-1}$, and
$\bZ_3$, $\bZ_2$,~$0$ for $E_6$, $E_7$, $E_8$ respectively.

\begin{corollary}\label{cor:H-grading}
  $\bZ^2 / \la p_\ell-p_*, p_r-p_* \ra$ is equal to $\Lambda^*/\Lambda$
  for the pure $D$ and $E$ shapes, and
  $(\Lambda^*/\Lambda)\oplus\bZ_2$ for the pure $A$ shapes.
\end{corollary}

\begin{lemma}\label{lem:2dim-models-primed}
  For primed \ade shapes which admit a toric description (see 
  Subsection~\ref{sec:primed-toric-shapes})
  the rule $\vpi \mapsto p-p_*$ defines a homomorphism
  \begin{math}
    \phi\colon \Lambda^*\oplus \bZ\vpi_\ell\oplus\bZ\vpi_r
    \xrightarrow{\phi} \bZ^2.
  \end{math}
  The projection
  $\pi_1\colon\Lambda^*\oplus \bZ\vpi_\ell\oplus\bZ\vpi_r \to\Lambda^*$
  identifies $\ker\phi$ with $\Lambda\subset\Lambda'\subset\Lambda^*$
  given below
  \begin{displaymath}
    \begin{array}[h]{lll}
      \text{shape} & \Lambda'/\Lambda & \text{generators} \\
      \hline
      \pA_{2n-1}, \pA_{2n-2}^-& 0 & \\
      \pA_{2n-1}' & \bZ_2 & \vpi_n\\
      D_{2n}' \text{ for } n \text{ even, resp. odd}
                   & \bZ_2 & \vpi' \text{ resp. }
                        \vpi'' 
  \end{array}
  \end{displaymath}
\end{lemma}
\begin{proof}
  For the corner node $p_0$ in $\pA_{2n-1}, \pA^-_{2n-2}$ one uses
  the corner relation~\eqref{eq:root-corner} with $\vpi'$ replaced
  by $\vpi_\ell$, and similarly for $D'_{2n}$.  Additionally: for
  $\pA'_{2n-1}$ one has
  $\vpi_n - \vpi_\ell - \vpi_r\xrightarrow{\phi} 0$, and for $D_{2n}$
  one has
  $\phi(\vpi_r + \lfloor \frac{n}2 \rfloor\vpi_\ell) = \phi(\vpi')$,
  resp. $=\phi(\vpi'')$.  See Fig.~\ref{fig:pAp} for the node
  notations.
\end{proof}

\subsection{Moduli of \ade pairs of pure shapes}
\label{sec:moduli-ade-pure}

In this subsection we prove the first part of Theorem~\ref{thm:moduli-summary}.
Recall that in Section~\ref{sec:adepairs} we associated to each \ade
pair \ycb an \ade root lattice. We use the notation introduced in
Section~\ref{sec:2dim-models}. 

\begin{definition}
  We define the tori $T_\Lambda = \Hom(\Lambda,\bC^*)$ and
  $T_{\Lambda^*} = \Hom(\Lambda^*,\bC^*)$ both isomorphic to
  $(\bC^*)^n$. We also define a finite multiplicative group
  $\mu_\Lambda = \Hom(\Lambda^*/\Lambda,\bC^*)$. Thus,
  $\mu_\Lambda = \mu_{n+1}$ for $A_n$, $\mu_2^2$ for $D_{2n}$, $\mu_4$
  for $D_{2n-1}$, and it is $\mu_3$, $\mu_2$, $1$ for $E_6$, $E_7$, $E_8$
  respectively.
\end{definition}

\begin{warning}
The theorem below is for pairs in which \emph{we distinguish the two sides
$C_1$ and $C_2$.} The moduli stack for the pairs with a single $C$ is
the $\bZ_2$-quotient for the shapes with the left / right symmetry,
and is the same for the nonsymmetric shapes.
\end{warning}

\begin{theorem}\label{thm:moduli-ade-pure}
  The moduli stack of \ade pairs 
  of a fixed pure \ade shape is
  \begin{eqnarray*}
    &[\Hom(\Lambda^*, \bC) : \mu_\Lambda\times\mu_2] =
      [T_\Lambda : W_\Lambda \times \mu_2]    
    &\text{for $A$ shapes} \\
    &[\Hom(\Lambda^*, \bC) : \mu_\Lambda] = [T_\Lambda : W_\Lambda]    
    &\text{for $D$ and $E$ shapes.} 
  \end{eqnarray*}
\end{theorem}

\begin{remark}\label{rem:stabilizers}
  The first presentation is convenient for finding automorphism
  groups. In particular, the maximal automorphism group that a pair
  can have is $\mu_\Lambda\times\mu_2$ for $A$ shapes
  and $\mu_\Lambda$ for $D$ and $E$ shapes.  The
  second form is convenient for compactifications, which in
  Section~\ref{sec:compactifications} are shown to be quotients of toric varieties
  by Weyl groups.
\end{remark}

\begin{proof}
  We first note that the pair \ycb is log canonical near the boundary
  $C$ iff the divisor $B$ intersects $C$ transversally. Vice versa,
  with this condition satisfied the pair \ycb for $0<\epsilon\ll 1$ is
  automatically log canonical. Otherwise, the pair $(Y, C+\frac12 B)$
  is not log canonical.  But by
  \cite[6.9]{shokurov1992three-dimensional} the non-klt locus must be
  connected, with a single exception when it may have two components,
  both of them simple, i.e. on a resolution each should give a unique
  curve with discrepancy $-1$. For an \ade surface the curve
  $C=C_1+C_2$ is connected with two irreducible components, so they
  are not simple. (We thank V.V. Shokurov for this argument.)

  Each of the \ade shapes is toric, and the polarized toric variety
  $(Y,L)$ corresponds to a lattice polytope $P$ as in Figs.~\ref{fig:A},
  \ref{fig:D}, \ref{fig:E}. However, $C=C_1+C_2$ gives
  only part of the toric boundary. Fixing the torus structure is
  equivalent to making a choice for the remainder of the torus
  boundary: one curve for the $A$ shapes and two curves for the $D,E$
  shapes. 
  With this choice made $(Y,L)$ is a polarized toric surface, and
  the equation of $B$ is
  \begin{math}
    f\in H^0(Y,L) = \oplus_{m\in \bZ^2\cap P}\, \bC e^m,
    \ \text{where } e^{(k,l)} = x^k y^l.
  \end{math}

  For the $A$ shapes the remaining toric boundary has the equation
  $y^2 \in H^0(Y,L)$. All the other choices for the toric boundary
  differ by the transformation $y\mapsto y + a(x)$. Completing the
  square we can make the coefficients of the monomials $yx^i$ in $f$ all
  zero.  By rescaling $x\mapsto \alpha x$, $y\mapsto\beta y$ we can
  put the equation $f$ in the form given in Table~\ref{tab:normal-forms}.
  In this table, $A_n^?$
   denotes either $A_n$ or $A_n^-$ depending on
  the parity of $n$, and similarly for $D,E$.

  \begin{table}[htp!]
    \centering
    \begin{displaymath}
      \renewcommand{\arraystretch}{1.2}
      \begin{array}{cll}
        \text{shape} & f_\bdry & f_\dyn \\
        \hline
        A_n^?& y^2 + 1 + x^{n+1} & c_1x + \dotsc + c_nx^n  \\
        \mA_n^?& y^2 + x(1 + x^{n+1}) & x(c_2x + \dotsc + c_{n+1}x^n)  \\
        D_{n}^?& x^2y^2 + y^2 + x^{n-2}
                               & c''xy + c'_1 y + c_0 + c_1x + \dotsc + c_{n-3} x^{n-3}\\
        \mE_n^?& x^2y^2 + y^3 + x^{n-3}
                               & c''xy + c'_1 y + c'_2y^2 + c_0 + c_1x + \dotsc + c_{n-4} x^{n-4}\\
        \hline
      \end{array}
    \end{displaymath}
    \caption{Normal forms for the equation $f=f_\bdry + f_\dyn$ of
      divisor $B$}
    \label{tab:normal-forms}
  \end{table}

  For the $D$ and $E$ shapes the remaining toric boundary has the
  equation $(xy)^2 \in H^0(Y,L)$. All other choices for the toric
  boundary differ by the transformations $x\mapsto x+a$,
  $y\mapsto y + b(x)$, with $\deg b(x) \le \frac12(n-3)$ for $D_n$ and
  $\deg b(x)\le \frac12(n-4)$ for $E_n$;
  and then rescaling $x$ and $y$. Using such transformations,
  one can put the equation $f$ in the form given in
  Table~\ref{tab:normal-forms} in an essentially
  unique way .
  
  The only remaining choice is the normalization of $f_\bdry$, which is
  unique up to the action of
  $\Hom( \bZ^2 / \la p_\ell-p_*, p_r-p_*\ra, \bC^*)$, equal to
  $\mu_\Lambda$ by Corollary~\ref{cor:H-grading}.  The end result is a
  normal form, given in Table~\ref{tab:normal-forms}, which is unique
  up to $\mu_\Lambda$. This
  gives the stack $[\bA^n : \mu_\Lambda]$. Finally, in the $A$ shapes
  every pair has an additional $\mu_2$ automorphism $y\mapsto
  -y$. This gives the first presentation of the moduli stack, as a
  $\mu_\Lambda\times\mu_2$, resp. $\mu_\Lambda$ quotient of $\bA^n$.

  It is a well known and easy to prove fact that the
  ring of invariants $\bC[\Lambda^*]^{W_\Lambda}$ is the polynomial
  ring $\bC[\chi_1,\dotsc, \chi_n]$, where $\chi_i=\chi(\vpi_i)$ are
  the characters of the fundamental weights (\cite[Ch.8, \S7, Thm.2]
  {bourbaki2005lie-groups}). In other words,
  $T_{\Lambda^*} / W_\Lambda = \bA^n$, with the coordinates $\chi_i$.
  The $\mu_\Lambda$-actions on $T_{\Lambda^*}$ and $\bA^n$ are given
  by the compatible $(\Lambda^*/\Lambda)$-gradings; thus they commute
  with the $W$-action. The $\mu_\Lambda$ action on $T_{\Lambda^*}$ is
  free, and $T_{\Lambda^*} / \mu_\Lambda = T_{\Lambda}$. Thus
  \begin{displaymath}
    [\bA^n : \mu_\Lambda] = [(T_{\Lambda^*} : W) : \mu_\Lambda] =
    [(T_{\Lambda^*} : \mu_\Lambda) : W] =
    [T_\Lambda : W], 
  \end{displaymath}
  giving the second presentation. For the $A$ shapes the additional
  $\mu_2$ action commutes with both $\mu_\Lambda$ and $W$.
\end{proof}

\subsection{Moduli of \ade pairs of toric primed shapes}
\label{sec:moduli-ade-toric-primed}

We state the theorem analogous to Theorem~\ref{thm:moduli-ade-pure}
for the primed \ade shapes which admit a toric description
 (see Section~\ref{sec:primed-toric-shapes}). It can be proved
 analogously to the theorem above, using
Lemma~\ref{lem:2dim-models-primed}, or can be seen as an
immediate consequence of
Theorem~\ref{thm:moduli-ade-primed}. 

\begin{theorem}\label{thm:moduli-ade-toric-primed}
  The moduli stack of \ade pairs
  of a fixed toric primed shape is
  \begin{displaymath}
    [\Hom(\Lambda^*, \bC) : \mu_{\Lambda'}] = [T_{\Lambda'} :
    W_\Lambda],
    \quad\text{where } T_{\Lambda'} = \Hom(\Lambda',\bC^*),
  \end{displaymath}
    $\mu_{\Lambda'} = \Hom(\Lambda^*/\Lambda', \bC^*)$,
  and the lattice $\Lambda\subset\Lambda'\subset\Lambda^*$
  is given in Lemma~\ref{lem:2dim-models-primed}. 
\end{theorem}

\subsection{Moduli of \ade pairs of all primed shapes}
\label{sec:moduli-ade-primed}

In this subsection we find the moduli stack for all primed shapes, including
those which do not admit a toric description, and in doing so
complete the proof of Theorem~\ref{thm:moduli-summary}.
\emph{We still mark the sides as left and right,} even if some or all
of the boundary curves are contracted. 

\begin{theorem}\label{thm:moduli-ade-primed}
  The moduli stack of pairs
  of a fixed primed shape is 
  \begin{displaymath}
    [\Hom(\Lambda^*, \bC) : \mu_{\Lambda'} \times W_0] = [T_{\Lambda'} :
    W_\Lambda \rtimes W_0],
  \end{displaymath}
  where $\mu_{\Lambda'} = \Hom(\Lambda^*/\Lambda', \bC^*)$ and the lattice
  $\Lambda\subset\Lambda'\subset\Lambda^*$ is as follows:
  \begin{displaymath}
    \begin{array}[h]{lll}
      \text{shape} & \Lambda'/\Lambda & \text{generators} \\
      \hline
      \pA_{2n-1}, \pA_{2n-2}^-& 0 & \\
      \pA_{2n-1}' & \bZ_2 & \vpi_n\\
      D_{2n}' \text{ for } n \text{ even, resp. odd}
                   & \bZ_2 & \vpi' \text{ resp. }
                        \vpi'' \\
      \pD_{2n},
      \text{ resp. } \pD_{2n-1}^- 
                   & \bZ_2 & \vpi_{2n-3}, \text{ resp. } \vpi_{2n-4} \\
      \pD_{2n}' \text{ for } n \text{ even, resp. odd}&
      \bZ_2\times\bZ_2 & \vpi_{2n-3}, \vpi', \text{ resp. }
                \vpi_{2n-3}, \vpi'' \\
      \mE'_7 & \bZ_2 & \vpi_3
  \end{array}
  \end{displaymath}
  For shapes $\ppS$ and $S''$ the lattices $\Lambda'$ are the 
  same as for the unprimed shape $S$, and similarly for 
  $\plS$ resp. $S^+$ and the unprimed shapes $\mS$ resp. $S^-$. 
  The additional Weyl group $W_0$ is the one given in
  Theorem~\ref{thm:reconstruction}, and its action is described in
  Theorem~\ref{thm:W0-action}.
\end{theorem}

\begin{proof}
  The pair $(\oY', \oC'_1 + \oC'_2 + \freps \oB')$ is obtained from a
  pair $(Y, C_1 + C_2 + \freps B)$ of pure shape by blowing up several points
  $P_i\in B\cap C$ at the ideals $I_i$ with directions equal to the tangent
  directions of $B$, and
  then contracting by the semiample line bundle $L'$. This
  construction works for the entire family over $\bA^n =
  \Hom(\Lambda^*, \bC)$: we blow up sections and it is easy to see
  that the sheaf $\cL'$ in the family is relatively semiample.

  When priming on a short side, or priming twice on a long side, there
  are no choices for $\prod I_i$. The only 2:1 choice is when there is
  a long side $C_s$ and we prime only at one of the two points in
  $B\cap C_s$. 
  Secondly, as stated in Theorem~\ref{thm:reconstruction}, for some
  shapes of genus 1 there is more than one precursor. These choices
  define an additional quotient by $W_0$.
\end{proof}

\subsection{Definitions of the naive \ade families}
\label{sec:naive-families}

For the toric \ade shapes $A$, $\pA$, $D$ and $E$ we define explicit
modular families of \ade pairs over the torus $T_{\Lambda^*}$.  We call these
the \emph{naive families}. Blowing up the sections corresponding to
the points in $C\cap B$, we obtain the naive families for all the
primed \ade shapes.

\smallskip
 
For the $A_n^?$-shapes, where $A_n^?$ is either $A_n$ or $A_n^-$
depending upon the parity of $n$, we take
the equation of Table~\ref{tab:normal-forms} with
$c_i = \chi_i = \chi(\vpi_i)$, the characters of the fundamental
weights, and with $y^2$ rescaled to $-(\frac{y}2)^2$, which will be
convenient when we come to discuss degenerations.

We recall that the $A_n$ root lattice is
$\la e_i-e_j \ra \subset \bZ^{n+1}$ and the dual weight lattice is
$A_n^* = \la f_i \ra$, where $f_i = e_i - p$,
$p = \frac1{n+1}\sum e_i$, so that $\sum_{i=1}^{n+1} f_i = 0$.  Thus,
$\bC[\Lambda^*] = \bC[t_1^\pm, \dots, t_{n+1}^\pm] / (\prod t_k - 1)$
and $\bC[\Lambda] = \bC[t_i/t_j]$, with $t_i = e^{f_i}$.
The first torus is
$T_{\Lambda^*} = \{\prod t_i = 1 \} \subset (\bC^*)^{n+1}$, and the
second one is $T_\Lambda = (\bC^*)^{n+1}/\diag\bC^*$. One has
$T_\Lambda = T_{\Lambda^*} / \mu_{n+1}$.

The Weyl group is $S_{n+1}$, and the characters of the fundamental
weights are the symmetric polynomials $\chi_i =
\sigma_i(t_k)$. Therefore, the defining equation of the naive family
is
\begin{equation}\label{eq:A-family}
  A_n^?:\ f = -\left(\frac{y}2\right)^2 + \prod_{i=1}^{n+1} (x+t_i) =
       -\left(\frac{y}2\right)^2 + 1 + \chi_1 x + \dotsc \chi_n x^n + x^{n+1}.
\end{equation}

For $\mA_n^?$ shapes we number the nodes $2,\dotsc,n+1$
(cf. Fig.~\ref{fig:A}) and the equation is as follows, where
$\chi_k = \sigma_{i-1}(t_i)$:
\begin{equation}\label{eq:mA-family}
  \mA_n^?: f = -\left(\frac{y}2\right)^2 + x\prod_{i=1}^{n+1} (x+t_i) =
 -\left(\frac{y}2\right)^2 + x\left(1 + \chi_2 x + \dotsc \chi_{n+1} x^n + x^{n+1}\right).
\end{equation}

\smallskip

For the toric shapes with one corner, i.e. $D_n^?$, $\mE_n^?$ and
$\pA_n^?$ (here again the $?$ is either no decoration or a $-$, depending
upon the parity), we make the following change of coordinates. We begin with
the affine equation of a double cover $X\to Y$ of the form
\begin{displaymath}
  F(x,y,z) = - xyz + z^2 + c''z + p(x) + q(y) = 0.
\end{displaymath}
Introducing the variable $w = z-\frac12(xy-c'')$, the equation becomes 
\begin{displaymath}
  w^2+f(x,y) = 0, \qquad
  f(x,y)= -\left( \frac{xy-c''}2 \right)^2 + p(x) + q(y)
\end{displaymath}
with the same $p(x)$, $q(y)$.  Thus, the affine equation of the branch
curve $B$ is $f(x,y)$, which we accept as our main
equation. Explicitly, the families are:
\begin{eqnarray}
  \label{eq:A'-families}
  &\pA_n^?: & f = -\left( \frac{xy-c''}2 \right)^2 +
  y + c_0 + c_1x + \dotsb+ c_{n-2}x^{n-2} + x^{n-1}\\
  \label{eq:D-families}
  &D_n^?: & f = -\left( \frac{xy-c''}2 \right)^2 +
  y^2 + c'_1y + c_0 + c_1x + \dotsb+ c_{n-3}x^{n-3} + x^{n-2}\\
  \label{eq:E-families} 
  &\mE_n^?: & f = -\left( \frac{xy-c''}2 \right)^2 +
  y^3 + c'_2y^2 + c'_1y + c_0 + \dotsb+ c_{n-4}x^{n-4} + x^{n-3}
\end{eqnarray}
In all of these families we take the coefficients to be
$c=\chi(\vpi)$, the fundamental characters, i.e. the characters of
the fundamental weights corresponding to the $n$ nodes of the Dynkin
diagram, using our Notation~\ref{not:roots-weights}.

\subsection{Action of the extra Weyl group $W_0$}
\label{sec:W0-action}

When a pure shaped precursor is not uniquely determined, as in
Theorem~\ref{thm:reconstruction}, there is an additional Weyl group
$W_0$ acting on the pure shape moduli torus $T_{\Lambda'}$. We divide
by it in Theorem~\ref{thm:moduli-ade-primed}.

\begin{theorem}\label{thm:W0-action}
  The Weyl group $W_0$ of Theorem~\ref{thm:reconstruction} acts on
  $T_{\Lambda'}$ as follows:
  \begin{enumerate}
  \item \emph{Genus $>1$.} For $\pD_{2n}^?$ and $\pD_{2n-1}^?$
    shapes, $W_0=W(A_1)=S_2$ acts by an automorphism of the
    $D$-lattice switching the two short legs $p'$ and $p''$.
    For $\ppD_{2n}^?$ and $\ppD_{2n-1}^?$ shapes, one has
    $W_0=W(A_1^2)=S_2^2$. The first $S_2$ acts by switching
    the two short legs $p'$ and $p''$. The second $S_2$ gives an
    additional $S_2$ automorphism of the pair \ycbend.
  \item \emph{Genus 1.} For the following shapes the action is as in
    (1) under the identifications: $\pA'_3=\pD^-_3$,
    $\pA''_3=\ppD^-_3$, $D'_4=\pD_4$, $D''_4=\ppD_4$. For $\pD'_4$
    the group $W_0=W(A_2) = S_3$ acts by permuting the three legs of
    the $D_4$ diagram. For $\pD''_4$, one has $W_0=W(A_3)=S_4 = S_3
    \ltimes S_2^2$. Here, $S_3$ acts by permuting the legs and
    $S_2^2$ gives an extra automorphism group of the pair \ycbend.
\end{enumerate}
\end{theorem}
\begin{proof}
  (1) From the equation~\eqref{eq:D-families} of the $D$-family we see
  that the side $C_1$ is defined by $(y_0:y_1)=(0:1)$, where
  $y=\frac{y_1}{y_0}$. There are two points $x=\pm 2$ on $C_1$ at which one can
  prime. For $x=2$, consider the map
  $\varphi_+\colon y\mapsto \frac{c''+c'}{x-2}-y$, $x\mapsto x$. 
  It is easy to check
  that the equation~\eqref{eq:D-families} maps to the same equation
  but with $c'$ and $c''$ switched. The map $\varphi_+$ is a rational
  map for a surface of $D^?$ shape but it becomes regular on the
  blowup, a surface of $\pD^?$ shape. Similarly, for priming at $x=-2$
  the map $\varphi_-\colon y\mapsto \frac{c''-c'}{x+2}+y$ works the
  same way.  The composition
  $\varphi_-\circ\varphi_+\colon y\mapsto \frac{c''-c'}{x+2} +
  \frac{c''+c'}{x-2}- y$, $x\mapsto x$ exchanges the two branches of
  the curve $B$, a two-section of the $\bP^1$-fibration. For
  surfaces $Y$ of $D^?$ and $\pD^?$ shapes this is a rational
  involution. It becomes a regular involution of a surface of $\ppD^?$
  shape, where $B$ is disconnected from $C_1$.  Case (2) is checked
  similarly.
\end{proof}

\begin{definition}\label{def:W00}
  Let $W_{00}\subset W_0$ be the subgroup which acts trivially on the
  the points of $T_{\Lambda'}$, giving extra automorphisms of the pairs.
\end{definition}

\begin{corollary}
  The group $W_0/W_{00}$ acts by diagram automorphisms of the
  decorated Dynkin diagram, permuting the short legs, all of them
  white circled vertices: for
  $\pD^?_{2n}$, $\pD^?_{2n-1}$, $\ppD^?_{2n}$, $\ppD^?_{2n-1}$,
  $\pA'_3$, $\pA''_3$ it is two legs, and for $\pD_4'$, $\pD_4''$ three legs,
  cf. Fig.~\ref{fig:circled-diags}.
\end{corollary}

\section{Compactifications of moduli of \ade pairs}
\label{sec:compactifications}

\subsection{Stable pairs in general and stable \ade pairs}
\label{sec:stable-pairs-general}

We recall some standard definitions from the theory of moduli of
stable pairs. We note in particular a close relationship between the
contents of this subsection and work of Hacking
\cite{hacking2004compact-moduli,
  hacking2004a-compactification-of-the-space}, who studied similar
ideas in the context of moduli of plane curves.

\begin{definition}\label{def:slc}
  A pair $(X,B = \sum b_iB_i)$ consisting of a reduced variety and a
  $\bQ$-divisor is \emph{semi log canonical (slc)} if $X$ is $S_2$, has at
  worst double crossings in codimension 1, and for the normalization
  $\nu\colon X^\nu \to X$ writing
  \begin{displaymath}
    \nu^*(K_X+B) = K_{X^\nu} + B^\nu,
  \end{displaymath}
  the pair $(X^\nu, B^\nu)$ is log canonical. Here $B^\nu = D + \sum b_i
  \nu\inv(B_i)$ and $D$ is the double locus.   
\end{definition}

\begin{definition}\label{def:stable-pair}
  A pair $(X,B)$ consisting of a connected projective
  variety $X$ and a $\bQ$-divisor $B$ is \emph{stable} if
  \begin{enumerate}
  \item $(X,B)$ has slc singularities, in particular $K_X+B$ is
    $\bQ$-Cartier.
  \item The $\bQ$-divisor $K_X+B$ is ample. 
  \end{enumerate}
\end{definition}

Next we introduce the objects that we are interested in here:
We could work equivalently with the pairs \ycb or with their double
covers $(X,D+\epsilon R)$. We choose the former.

\begin{definition}
  For a fixed degree $e\in\bN$ a fixed rational number
  $0<\epsilon\le 1$, a \emph{stable del Pezzo pair} of type $(e,\epsilon)$ is
  a pair \ycb such that
  \begin{enumerate}
  \item $2(K_X + C) + B \sim 0$
  \item The divisor $B$ is an ample Cartier divisor of degree $B^2=e$.
  \item \ycb is stable in the sense of
    Definition~\ref{def:stable-pair}.
  \end{enumerate}
\end{definition}

\begin{definition}\label{def:family-stable-pairs}
  A family of stable del Pezzo pairs of type $(e,\epsilon)$ is a flat
  morphism $f\colon (\cY,\cC + \freps\cB)\to S$ such that
  $(\omega^{\otimes 2}_{\cY/S}(\cC)^{**}\simeq \cO_Y$ locally on $S$,
  the divisor $\cB$ is a 
  relative Cartier divisor, such that every fiber is a stable del
  Pezzo pair of type $(e,\epsilon)$. We will denote by
  $\cM\slc_{\rm dp}(e,\epsilon)$ its moduli stack.
\end{definition}

\begin{proposition}\label{prop:dp-moduli}
  For a fixed degree $e$ there exists an $\epsilon_0(e)>0$ such that
  for any $0<\epsilon\le \epsilon_0$ the moduli stacks
  $\cM^{\rm slc}(e,\epsilon_0)$ and $\cM^{\rm slc}(e,\epsilon)$ coincide.
  The stack $\cM\slc(e,\epsilon_0)$ is a Deligne-Mumford stack of
  finite type with a coarse moduli space $M\slc(e,\epsilon_0)$ which
  is a separated algebraic space. 
\end{proposition}
\begin{proof}
  For a fixed surface $Y$, there exists an $0<\epsilon_0\ll 1$ such
  that the pair $(Y,C + \frac{1+\epsilon}2 B)$ is slc iff $B$ does not contain any
  centers of log canonical singularities
  of $(Y,C + \frac12 B)$: images of the divisors with
  codiscrepancy $b_i=1$ on a log resolution of singularities
  $Z\to Y^\nu\to Y$.
  There are finitely many of such centers.  Then for any
  $\epsilon<\epsilon_0$, the pair $(Y,C + \frac{1+\epsilon_0}2 B)$ is slc iff
  \ycb is. 
  Now since $B$ is ample Cartier of a fixed degree,
  the family of the pairs is bounded, and the number
  $\epsilon_0$ with this property can be chosen universally.

  We refer to \cite{kollar1988threefolds-and-deformations,
    kollar2015book-on-moduli},
  \cite{alexeev2006higher-dimensional-analogues} for the existence and
  projectivity of the moduli space of stable pairs $(X,\sum b_iB_i)$.
  There are complications arising in the construction when some
  coefficients $b_i\le\frac12$ and when the divisor $B$ is not
  $\bQ$-Cartier, all of which are not present in this situation.
\end{proof}

\begin{definition}\label{def:moduli-stable-ade}
  For a fixed \ade shape, we denote by $M\slc_{ADE}$ the closure of
  the moduli space of \ade pairs of this shape in $M\slc_{\rm
    dP}(e,\epsilon_0)$ for $e=B^2$, with the reduced scheme structure.
\end{definition}

In this Section will show that $M\slc_{ADE}$ is proper and that in
fact the stable limits of \ade pairs are of a very special kind: they
are \emph{stable \ade pairs}. We will also show that the normalization
of $M\slc_{ADE}$ is an explicit projective toric variety for a
generalized Coxeter fan.

\begin{definition}\label{def:stable-ade-pair}
  A \emph{stable \ade pair} is a stable del Pezzo pair \ycb such that
  its normalization is a union of \ade pairs
  $\sqcup (Y^\nu_k, C^\nu_k + \freps B^\nu_k)$.
\end{definition}

\begin{theorem}\label{thm:normalize-stable-ade-pair}
  For a stable \ade pair the irreducible components are of two kinds:
  \begin{enumerate}
  \item \emph{normal}, i.e. $\nu\colon Y^\nu_k\isoto Y_k$, or
  \item \emph{folded:} the morphism $\nu\colon Y^\nu_k\to Y_k$ is an
    isomorphism outside of $C_k$, and is a double cover
    $\bP^1\to\bP^1$ on one or two sides $C^\nu_{k,s} \to C_{k,s}$,
    $s=1,2$. In this case, the side $C^\nu_{k,s}$ is necessarily a long
    side of the \ade pair.
  \end{enumerate}
\end{theorem}
\begin{proof}
  The normalization of a stable pair is an isomorphism outside of the
  double locus and is 2:1 on the double locus, so these are the two
  possibilities. The side must be long because $\nu^* B_k \cdot
  C^\nu_{k,s} = 2 B_k \cdot C_{k,s}$ is even and $\ge2$. 
\end{proof}

\begin{definition}\label{def:folded-shape}
  We will call the surfaces of type (2) in the above theorem the
  \emph{folded shapes}.  We denote a fold by adding the $f$
  superscript to the corresponding long side, e.g.  $A_{2n-1}^f$,
  $\foA_{2n-1}^f$, $\mA_{2n}^f$, $\pA_{2n-1}^f$.  We define the
  decorated Dynkin diagrams for these shapes by double circling the
  corresponding end (unfilled) node.  We do not draw any pictures for
  these here.
\end{definition}

Next, we extend the naive families of \ade pairs, defined
in section~\ref{sec:naive-families}, to families of stable pairs over a
projective toric variety corresponding to the Coxeter fan. We start
with the $A_n$ case.

\subsection{Compactifications of the naive families for the $A$ shapes}
\label{sec:comps-families-A}

Recall that $T_{\Lambda^*} = \Spec\bC[\Lambda^*]$. We define the
following elements of the homogeneous ring $\bC[\Lambda^*][x,y][\xi]$,
with the grading defined by $\deg\xi = 1$.

\begin{definition}
  In the $A_n^?$ shape, where the $?$ denotes either
  no decoration or a $-$ depending upon the parity of $n$, 
  for each node $p_1,\ldots,p_n$ of the Dynkin
  diagram we introduce a degree~2 element
  $u_i = e^{\vpi_i} x^i \cdot\xi^2$, where
  $e^{\vpi_i} \in \bC[\Lambda^*]$ is the monomial corresponding to
  the fundamental weight $\vpi_i\in\Lambda^*$. In addition, we
  introduce the degree~2 elements $u_0=1\cdot\xi^2$ and $u_{n+1} =
  x^{n+1}\cdot \xi^2$, corresponding to the left and right nodes 
  $p_l = p_0$ and $p_r = p_{n+1}$,  and $u_* = y^2 \cdot \xi^2$
  corresponding to the vertex $p_*$.
  Similarly, in the $\mA_n^?$ shape we define the elements $u_1,\dotsc,
  u_{n+2}$ and $u_*$. 
\end{definition}

Because even the simplest \ade surface of $A^-_0$-shape is a
weighted projective space $\bP(1,1,2)$, it is convenient to introduce
some square roots.

\begin{definition}
  For the even nodes $p_{2i}$ we introduce the degree~1 elements of the
  ring $R[\xi]$:   $v_{2i} = e^{\vpi_{2i}/2} x^i\cdot \xi$ and $v_* =
  y \cdot \xi$. 
  Thus, $v_{2i}^2 = u_{2i}$ and $v_*^2 = u_*$. 
\end{definition}

We recall that in the naive families \eqref{eq:A-family},
\eqref{eq:mA-family} we take the coefficients $c_i = \chi_i$, the fundamental
characters.  As in Section~\ref{sec:2dim-models}, let $\alpha_i$ be
the simple roots.

\begin{definition}
  Set $a_i = e^{-\alpha_i}$ for all $i$, and for odd indices
  set $b_{2i+1} = e^{-\alpha_{2i+1}/2}$.  Finally,
  define normalized coefficients $\hc_i = e^{-\vpi_i} c_i$.
\end{definition}

It is well known that for any dominant weight $\lambda\in\Lambda^*$
the character $\chi(\lambda)\in\bC[\Lambda^*]$ is a
$W_\Lambda$-invariant Laurent polynomial whose highest weight is
$\lambda$ and the other weights are of the form
$\mu=\lambda - \sum n_i\alpha_i$ for some $n_i\ge0$. Thus, $\hc_k$ are
polynomials in $a_i$'s, and $\hc_k = 1 +
(\text{higher terms in } a_i)$. 

With these notations, we consider the equation $f$ of the naive family
\eqref{eq:A-family} to be the following homogeneous degree~2 element
in $\bC[\Lambda^*][x,y][\xi]$ (similarly for $\mA_n^?$):
\begin{equation}\label{eq:A-families-uv}
  f = -\left(\frac{v_*}2\right)^2 
  + u_0 + \hc_1 u_1 + \dotsc + \hc_n u_n + u_{n+1}
  \in \bC[\Lambda^*][x,y] \cdot\xi^2
\end{equation}

For the construction of the family one might as well work with the
ring $\bC[\frac12\Lambda^*]$ but we will use the minimal choice for
clarity.

\begin{definition}\label{def:Mlattice-A}
  Let $M$ be the lattice obtained by adjoining to $\Lambda^*$ the
  vectors $\vpi_{2i}/2$ and $\alpha_{2i+1}/2$ for all $i$.  Let
  $M^+ = M\cap \sum \bR_{\ge0} (-\alpha_i)$ and $R = \bC[M^+]$. Thus,
  $\Spec R$ is a normal affine toric variety which is a
  $\mu_2^N$-cover of $\bA^n=\Spec\bC[a_i]$ for some $N$.
\end{definition}

\begin{definition}[Compactified naive families for the $A_n^?$,
  $\mA_n^?$ shapes] 
  \label{def:compactified-family-A}
  Let $S$ be the graded subring of $R[x,y][\xi]$ generated by
  $v_{2i}$, $u_{2i+1}$, and $v_*$. The \emph{compactified naive
    family} is $\cY:= \Proj S \to \Spec R$ with a relative Cartier
  divisor $\cB = (f)$, $f\in H^0(\cO(2))$.  We note that since the
  subring $S^{(2)}$ is generated in degree 1, the sheaf
  $\cO_{\Proj S}(2)$ is invertible and ample.
\end{definition}

\begin{example}
  For the $A_1$ shape, the $A_1$ root lattice has
  $\bC[\Lambda^*] = \bC[t_1^{\pm},t_2^{\pm}]/(t_1t_2 - 1) \cong \bC[t^{\pm}]$, 
  with $t = e^{\alpha_1/2}$. 
  The family is 
  $\Proj S\to \bA^1 = \Spec \bC[b_1]$, where
  $S=R[v_*, v_0,u_1,v_1] / (v_0v_2 - b_1 u_1)$.  One has
  $\chi_1 = t + t\inv = t(1+b_1^2)$, and the equation of the divisor
  $\cB$ is
  \begin{displaymath}
    f= -\left(\frac{v_*}2\right)^2 + v_0^2 + (1+b_1^2)u_1 + v_2^2.  
  \end{displaymath}
  Setting $b_1=0$ gives the degenerate fiber
  $\bP(1,1,2) \cup \bP(1,1,2)$ with the coordinates $v_*,v_0,u_1$,
  resp. $v_*,v_2,u_1$, glued along a $\bP^1$ with the coordinate
  $u_1$.  The restriction of $f$ to $\bP(v_*,v_0,u_1)$ is
  $v_*^2 + v_0^2 + u_1$, and for $\bP(v_*,v_2,u_1)$ it is
  $v_*^2 + v_2^2 + u_1$. Thus, the degenerate fiber is a union of two
  \ade pairs $A_0^-\mA_0$ glued along a short side.

  For the $\mA^-_1$ shape the family is
  $\Proj S \to \bA^1 = \Spec \bC[a_1]$, where
  $S=R[u_1,v_2,u_3] / (u_1u_3 - a_1 v_2^4)$, and the equation of the
  divisor is
  \begin{displaymath}
    f= -\left(\frac{v_*}2\right)^2 + u_1 + (1+a_2)v_2^2 + u_2.
  \end{displaymath}
  Setting $a_1=0$ gives the degenerate fiber
  $\bP(1,1,2) \cup \bP(1,1,2)$ with the coordinates $v_*,v_2,u_1$,
  resp. $v_*,v_2,u_3$, which is the union of two \ade pairs
  $\mA_0 A_0^-$ glued along a long side, a $\bP^1$ with the
  coordinate $v_2$. 
\end{example}

The general case is essentially a generalization of this simple
example.  The degenerations of pairs for the slightly more complicated
$A^-_2$ shape are illustrated in Fig.~\ref{fig:A2degens}.

\begin{figure}[htp!]
  \centering
  \includegraphics{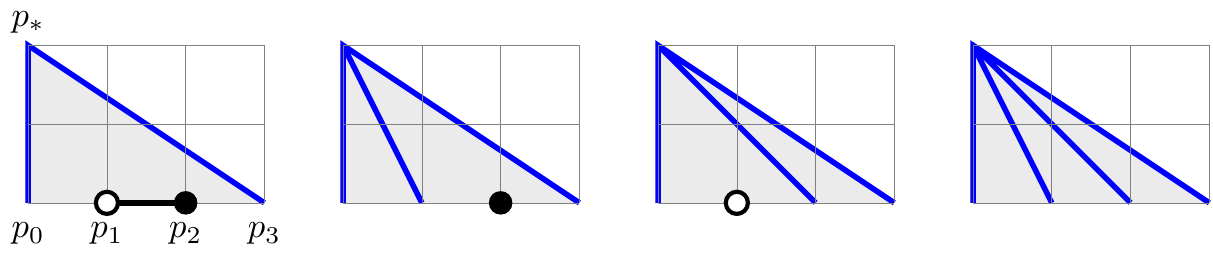}
  \caption{$A_2^-$ and its degenerations: $A_0^-\mA_1^-$, $A_1A_0^-$, 
    and $A_0^-\mA_0 A_0^-$}
  \label{fig:A2degens}
\end{figure}

\begin{definition}
  The Coxeter fan for a root lattice $\Lambda$ is the fan on
  $\Lambda_\bR = \Lambda^*_\bR$ obtained by cutting this vector space
  by the mirrors $\alpha^\perp$ to the roots $\alpha$. Its maximal
  cones are chambers, the translates of the positive chamber under the
  action of the Weyl group $W_\Lambda$.  We denote by $V\cox_M$ the
  torus embedding of $T_M=\Hom(M,\bC^*)$ for the Coxeter fan of $A_n$.
\end{definition}

\begin{lemma}\label{lem:A-relations}
  The following relations hold:
  \begin{enumerate}
  \item (Primary) $v_{2i}v_{2i+2} = b_{2i+1}u_{2i+1}$ and 
    $u_{2i-1}u_{2i+1} = a_{2i}v^4_{2i}$.
  \item (Secondary) 
    $u_{2i-1}u_{2j+1} = v_{2i}^2v_{2j}^2 \cdot A$, 
    $v_{2i-2} u_{2j+1} = v_{2i} v_{2j}^2 \cdot b_{2i-1} A$, and 
    $v_{2i-2}v_{2j+2} = v_{2i}v_{2j} \cdot b_{2i-1}b_{2j+1} A$, where
    $A=\prod_{k=2i}^{2j} a_k$. 
  \end{enumerate}
\end{lemma}
\begin{proof}
  An easy direct check using equations~\eqref{eq:root-middle},
  \eqref{eq:root-end}. 
\end{proof}

\begin{theorem}\label{thm:compactified-family-A}
  The compactified family $\cY = \Proj S \to \Spec R$ of shape $A_n^?$
  or $\mA_n^?$ is flat.  The degenerate fibers are over the subsets
  given by setting some $a_i$'s to zero.  Every fiber of this family
  is a stable \ade pair which is a union of \ade pairs of shapes
  obtained by subdividing the $A_n^?$, resp. $\mA_n^?$, polytope into
  integral subpolytopes of smaller $A$ shapes by intervals from the
  vertex $p_*$ to the points $p_i$ for which one has $a_i=0$. 

  The $W_\Lambda$-translates of this family glue into a flat
  $W_\Lambda$-invariant family $(\cY, \cC + \freps\cB) \to V\cox_M$ of
  stable \ade pairs.
\end{theorem}

\begin{proof}
  Let $t\in \Spec R$ be a closed point and $\cY_t$ be a fiber over
  $t$. Suppose that some $a_k(t)=0$ or $b_k(t)=0$.  The relations of
  Lemma~\ref{lem:A-relations} imply that any two ($u$ or $v$) variables
  with indices $i<k$ and $j>k$ multiply to give zero. On the other
  hand, the product of two variables with indices $i,j$ for which
  the coordinates with $i<k<j$ satisfy $a_k(t),b_k(t)\ne0$, is a nonzero
  monomial.

  Let $P$ be the polytope corresponding to the shape $A_n^?$, resp.
  $\mA_n^?$.  The above equations define a stable toric variety
  $Z=\cup Z_s$ for the polyhedral decomposition $P=\cup P_s$ obtained
  by cutting $P$ by the intervals from the vertex $p_*$ to the points
  $p_k$ for each $k$ with $a_k=0$ or $b_k=0$,
  cf. \cite{alexeev2002complete-moduli}.  In other words, $Z$ is a
  reduced seminormal variety which is a union of projective toric
  varieties, glued along torus orbits.

  The fiber $\cY_t$ is a closed subscheme of $Z$. But the Hilbert
  polynomial of $Z$ with respect to $\cO(2)$ is the same as for a general
  fiber, a projective toric variety for the polytope~$P$. By the
  semicontinuity of Hilbert polynomials in families, 
  $\cY_t=Z$. Since the base $T_{\Lambda^*}$ is reduced, the
  constancy of the Hilbert polynomial implies that the family is
  flat.
  
  The equation $f$ restricts on each irreducible component to the naive
  equation of an \ade pair for a smaller $A$ shape by
  Lemma~\ref{lem:restrict-chi}.

  The $W_\Lambda$-translates of this family automatically glue into a 
  $W_\Lambda$-invariant family over a torus embedding of $T_M$ for
  the Coxeter fan of $A_n$ because the $u,v$ variables map to the
  corresponding variables for a different choice of positive roots,
  and the equation $f$ is $W$-invariant. Flatness is a local
  condition, so it holds.
\end{proof}

\begin{lemma}\label{lem:restrict-chi}
  Let $\Lambda$ be a an irreducible \ade root lattice with Dynkin
  diagram~$\Delta$ and Weyl group
  $W=\la w_\alpha \mid \alpha\in\Delta \ra$.  Let $\beta\in \Delta$ be
  a simple root, and $\Lambda'$ be the lattice (not necessarily
  irreducible) corresponding to $\Delta' = \Delta\setminus\beta$, with
  Weyl group $W' = \la w_\alpha \mid \alpha\ne\beta \ra$.  Let $r$
  be the natural restriction homomorphism
  \begin{displaymath}
    r\colon k[e^{-\alpha},\alpha\in\Delta]\to
    k[e^{-\alpha},\alpha\in\Delta'], \qquad 
    e^{-\beta}\mapsto 0, \quad
    e^{-\alpha} \mapsto e^{-\alpha} \text{ for } \alpha\ne\beta.
  \end{displaymath}
  Then for the normalized fundamental character
  $\hchi_\alpha = e^{-\vpi_\alpha} \chi_\alpha$ corresponding to a
  simple root $\alpha$ one has
  \begin{displaymath}
    r( \hchi_\alpha ) = 
    \begin{cases}
      1 & \text{for } \alpha = \beta\\
      \hchi_\alpha & \text{for } \alpha\ne\beta
    \end{cases}
  \end{displaymath}
\end{lemma}
\begin{proof}
  Consider a dominant weight $\mu\in\Lambda^*$. We first make an
  elementary observation about the weight diagram of the highest
  weight representation $V(\mu)$. The weight diagram is obtained by
  starting with the highest weight $\mu = \sum m_k\vpi_k$ and
  subtracting simple roots $\alpha_s$ if the corresponding coordinate
  $m_s$ of $\mu$ is positive.
  Thus, for $\mu= \vpi_\beta$ the first and only move down is to
  the weight $\mu-\beta$. This says that $\wh{\chi}(\vpi_\beta) = 1 +
  e^{-\beta}(\dotsc)$. Therefore, 
  $r(\wh{\chi}(\vpi_\beta))=1$. 

  For $\alpha\ne\beta$, the moves down in the weight diagram of
  $V(\vpi_\alpha)$ not involving $\beta$ are the same as the moves in
  the Dynkin diagram $\Delta\setminus\beta$.  So the monomials
  appearing in $r\big( \wh{\chi}(\vpi_\alpha) \big)$ for the Dynkin
  diagram $\Delta$ and the monomials appearing in
  $\wh{\chi}(\vpi_\alpha)$ for the Dynkin diagram
  $\Delta\setminus\beta$ are the same.

  We have to show that the coefficients of these monomials are also
  the same. This follows from the Weyl character formula
  \begin{displaymath}
    \chi(\lambda) = \frac
      { \sum_{w\in W} \varepsilon(w) e^{w(\lambda + \rho)} }
      { \sum_{w\in W} \varepsilon(w) e^{w(\rho)} }
      , \quad \text{where }
      \rho = \sum_{\vpi\in\Pi} \vpi. 
    \end{displaymath}
    Isolating the terms $e^\mu$ on the top and the bottom where the
    linear function $(\beta,\mu)$ takes the maximum, and setting other
    terms to zero gives the same Weyl Character formula expression for
    the Weyl group $W'$. This concludes the proof.
  \end{proof}

  \begin{remark}
    The construction of the family of curve pairs over the Losev-Manin
    space for $A_n$ follows from this by an easy simplification: the
    two-dimensional polytope is replaced by $[0,n+1]$ and there are
    only $u_i$ variables, all of degree~1.
  \end{remark}
  
\subsection{Compactifications of the naive families for the $\pA$,
  $D$, $\mE$ shapes}
\label{sec:comps-families-pADE}

Before stating the general result, we begin with an elementary example.

\begin{figure}[htp!]
  \centering
  \includegraphics{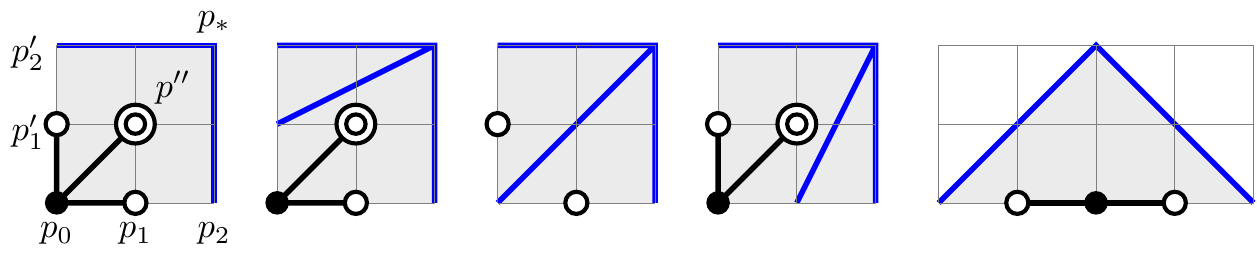}
  \caption{$D_4$ and its degenerations $A_0^-\pA_3$, $A_1A_1$,
    $A_3'\mA_0$, $A_3$}
  \label{fig:D4-degen}
\end{figure}

\begin{example}\label{ex:D4-degn}
  An \ade surface $(Y,C=C_1 + C_2)$ of shape $D_4$ is $Y=\bP^1\times\bP^1$, with
  $C_1=s$, $C_2=f$ a section and a fiber. In an \ade pair 
  $(Y,C + \frac{1+\epsilon}2 B)$, the divisor $B$ is
  in the linear system $|2s+2f|$. There are three obvious
  toric degenerations corresponding to removing the nodes $p'_1$, $p_0$, $p_1$
  in the Dynkin diagram, shown in the middle three pictures of Figure~\ref{fig:D4-degen}. 
  In the degeneration corresponding to $p'_1$
  we get a 3-dimensional family of stable \ade pairs with two
  components corresponding to $A_0^-\pA_3$. By symmetry, we get
  $A_3'\mA_0$ surfaces for the node $p_1$.

  The toric degeneration for the node $p_0$ is already somewhat
  unusual. Here $\bP^1\times\bP^1$ degenerates to $\bP^2\cup\bP^2$,
  and the stable \ade pairs of shape $A_1A_1$ form only a
  2-dimensional family, so some moduli are lost.

  Additionally, there is an obvious nontoric degeneration of
  $\bP^1\times\bP^1$ to a quadratic cone $\bP(1,1,2)$, with the limits
  of $C_1,C_2$ passing through the vertex, and $B$ a double
  section. These are pairs of shape $A_3$ forming a 3-dimensional
  family. 
\end{example}

\begin{definition}
  In the $\pA_n^?$, $D_n^?$, $\mE_n^?$ shapes, 
  where $?$ denotes either no decoration or a $-$ depending upon
  the parity of $n$,
  we introduce the following elements of the homogeneous ring
  $\bC[\Lambda^*][x,y][\xi]$:
  \begin{displaymath}
    u_i = e^{\vpi_i} x^i \cdot\xi^2, \quad
    v_{2i} = e^{\vpi_{2i}/2} x^i \cdot\xi, \quad
    u'_i = e^{\vpi'_i} y^i \cdot\xi^2, \quad
    v'_{2i} = e^{\vpi'_{2i}/2} y^i \cdot\xi.
  \end{displaymath}
  We also have the $u$ variables for the left and right sides
  $p_\ell$, $p_r$ and, when these are even, their square roots, the
  $v$ variables. 
  Additionally, we define a special non-monomial variable
  $v_* = (xy - c'')\cdot \xi$ of degree~1.
\end{definition}

As before, the coefficients $c=\chi(\vpi)$ are the characters of the
fundamental weights, and we define the normalized characters by
$\hc = e^{-\vpi}c$. With these notations, the naive
families~\eqref{eq:A'-families}, \eqref{eq:D-families},
\eqref{eq:E-families} become
\begin{eqnarray}
  \label{eq:A'-families-uv}
  &\pA_n^?: & f = -\left( \frac{v_*}2 \right)^2 +
  u'_1 + \hc_0 u_0 + \hc_1 u_1 + \dotsb+ \hc_{n-2} u_{n-2} + u_{n-1}\\
  \label{eq:D-families-uv}
  &D_n^?: & f = -\left( \frac{v_*}2 \right)^2 +
  u'_2 + \hc'_1 u'_1 + \hc_0 u_0 + \hc_1 u_1 + \dotsb+ \hc_{n-3}u_{n-3} + u_{n-2}\\
  \label{eq:E-families-uv} 
  &\mE_n^?: & f = -\left( \frac{v_*}2 \right)^2 +
  u'_3 + \hc'_2u'_2 + \hc'_1u'_1 + \hc_0u_0 + \dotsb+ \hc_{n-4}u_{n-4} + u_{n-3}
\end{eqnarray}

\begin{definition}\label{def:Mlattice-A'DE}
  Let $M$ be the lattice obtained by adjoining to $\Lambda^*$ the
  vectors $\vpi_{2i}/2$, $\vpi'_{2i}/2$, $\alpha_{2i+1}/2$ (for all $i$), and
  $\alpha''/2$.  Let $M^+ = M\cap \sum_\alpha \bR_{\ge0} (-\alpha)$ and
  $R = \bC[M^+]$. 

  We define $a_i = e^{-\alpha_i}$, resp. $a'_i = e^{-\alpha'_i}$, for
  each node $p_i$ in the Dynkin diagram, and also $a'' = e^{-\alpha''}$. For
  the odd nodes $p_{2i+1}$ we define
  $b_{2i+1} = e^{-\alpha_{2i+1}/2}$, resp.
  $b'_{2i+1} = e^{-\alpha'_{2i+1}/2}$, and also
  $b'' = e^{-\alpha''/2}$.
\end{definition}

\begin{definition}[Compactified naive families for the $\pA, D, E$ shapes]
  \label{def:compactified-family-A'DE}
  Let $S$ be the graded subring of $R[x,y][\xi]$ generated by
  $v_{2i}$, $u_{2i+1}$, $v'_{2i}$, $u'_{2i+1}$, and $v_*$. The
  \emph{compactified naive family} is $\cY:= \Proj S \to \Spec R$ with
  a relative Cartier divisor $\cB = (f)$, $f\in H^0(\cO(2))$.  We note
  that since the subring $S^{(2)}$ is generated in degree 1, the sheaf
  $\cO_{\Proj S}(2)$ is invertible and ample.
\end{definition}

\begin{lemma}\label{lem:A'DE-relations}
  The following relations hold:
  \begin{enumerate}
  \item The same monomial relations as in Lemma~\ref{lem:A-relations}
    for the variables $u_i,v_i$, with $i\ge0$,
    and for the variables $u'_i$, $v'_i$, $v_0$.
  \item A non-monomial \emph{``corner''} relation
    $u_1u'_1 = a_0 v_0^3 (\hc'' v_0 + b'' v_*)$.
  \item For each $u'_i,v'_i$ variable and each $u_i,v_i$ variable lying on
    the different sides of $v_0$, the same equations as in
    Lemma~\ref{lem:A-relations}, but with
    $A = \frac{\hc'' v_0 + b'' v_*}{v_0}\prod_{k = 2i}^{2j}a_k$.
\end{enumerate}
\end{lemma}
\begin{proof}
  We check the non-monomial relation. The LHS is
  $e^{\vpi_1+\vpi'_1}xy \cdot \xi^4$. The RHS:
  \begin{displaymath}
    e^{-\alpha_0+\frac32\vpi_0} \big( e^{-\vpi''+\frac12 \vpi_0}c'' + 
    e^{-\frac12\alpha''} xy -     e^{-\frac12\alpha''} c''
    \big) \cdot \xi^4
  \end{displaymath}
  The equality now follows from
  $-\vpi''+\frac12\vpi'' = -\frac12\alpha''$ and
  $-\alpha_0+\frac32\vpi_0 - \frac12\alpha'' = \vpi_1 + \vpi'_1$,
  which hold because $\alpha'' = 2\vpi''-\vpi_0$ and
  $\alpha_0 = 2\vpi_0 - \vpi_1 - \vpi'_1 - \vpi''$. The proof of part
  (3) is formally the same as for the secondary monomial relations,
  with each term $\hc'' v_0 + b'' v_*$ contributing an extra $xy$.
\end{proof}

\begin{theorem}\label{thm:compactified-family-A'DE}
  The compactified families $\cY = \Proj S \to \Spec R$ of $\pA_n^?$,
  $D_n^?$, $\mE_n^?$ shapes are flat. The degenerate fibers are over
  the subsets given by setting some $a$'s to zero. 
  Every fiber is a 
  stable \ade pair which is a union of \ade pairs of shapes obtained
  as follows:
  \begin{enumerate}
  \item For the degenerations $a_i=0$ and $a'_i=0$: by subdividing the
    corresponding polytope into integral subpolytopes by intervals
    from the vertex $p_*$ to the point $p_i$, resp. $p'_i$.
  \item For the degeneration $a''=0$: by ``straightening the corner'',
    i.e. to the shape obtained by removing the node $p''$ from the
    Dynkin diagram. 
  \end{enumerate}

  The $W$-translates of these families glue into flat $W_\Lambda$-invariant
  families $(\cY, \cC + \freps\cB) \to V\cox_M$ of stable \ade pairs over
  a torus embedding of $T_M=\Hom(M,\bC^*)$ for the Coxeter fan of $A_n$,
  resp. $D_n$, resp. $E_n$.
\end{theorem}
\begin{proof}
  The proof for the toric degenerations is the same as in
  Theorem~\ref{thm:compactified-family-A}.  Gluing the family over
  $\Spec R$ to a $W$-invariant family over a projective toric variety
  for the Coxeter fan is also the same.  We do not repeat these parts.
  Instead, we concentrate on the degenerations involving the corner
  relation (2) of Lemma~\ref{lem:A'DE-relations}.

  When $a_0=0$ we get the toric relation $u_1u'_1=0$ which as before
  gives a stable \ade pair for the subdivision of our polytope into
  two polytopes obtained by cutting it from $p_*$ to $p_0$ into the
  shapes $^?A_{m'}A^?_m$. The only observation here is that $m+m' =
  n-2$, not $n-1$, so the moduli count drops by two, not one.

  When $b''=0$, we also get $\hc''=1$ by Lemma~\ref{lem:restrict-chi},
  and the corner relation becomes $u_1u'_1 = a_0v_0^4$. Thus, the new
  set of relations is equivalent to those for the $\quA_{n-1}^?$
  shape. The equations \eqref{eq:A'-families-uv},
  \eqref{eq:D-families-uv}, \eqref{eq:E-families-uv} reduce to the
  equation \eqref{eq:A-families-uv}.
\end{proof}

\subsection{Compactifications of the naive families for all primed shapes}
\label{sec:comps-families-primed}

Theorems~\ref{thm:compactified-family-A},
\ref{thm:compactified-family-A'DE} describe all stable \ade pairs that
appear as degenerations of \ade pairs of pure shapes. In particular,
irreducible components of degenerate pairs \ycb are normal, and they
are \ade pairs for smaller shapes.  For some degenerations of pairs of
primed shapes the folded shapes of Definition~\ref{def:folded-shape}
appear.

\begin{definition}\label{def:priming-rules}
  The \emph{Priming Rules} are $A'_1 \to A_0^-$, $\pA_1 \to \mA_0$;
  $\mA'_0 \to 0$, $\pA^-_0 \to 0$; and $A_0^+\to f$, $\plA_0\to f$ are
  folds applied to the neighboring surface. 
\end{definition}

\begin{theorem}\label{thm:compactified-family-primed}
  For each primed shape, there exists a flat family of stable \ade
  pairs $(\cY, \cC+\freps\cB) \to V\cox_M$ over the a torus embedding of
  $\Hom(M,\bC^*)$ for the Coxeter fan of $A_n$, resp. $D_n$,
  resp. $E_n$, where $M$ is the lattice defined in
  \eqref{def:Mlattice-A} for the $A$ shapes and
  \eqref{def:Mlattice-A'DE} for $D,E$ shapes. The fibers over the
  toric strata of $V\cox_M$ are computed by starting with the fibers of the
  family for the pure shape and then applying the Priming Rules
  \emph{one prime at a time}.
\end{theorem}

\begin{example}\label{ex:primed-degns}
  We list the degenerate fibers in the compactified families for the
  pure shape $A_2^-$ (see Fig.~\ref{fig:A2degens}) and the
  corresponding fibers in the families for the primed shapes
  $\pA_2^-$, $\ppA_2^-$, $A_2^+$.
  \begin{displaymath}
    \begin{array}{l|lll}
      \text{shape} & \multicolumn{3}{c}{\text{shapes of degenerations}}\\
      \hline
      A_2^- &A_0^-\mA_1^-  &A_1A_0^-  &A_0^-\mA_0 A_0^- \\
      \pA_2^- &\pA_0^-\mA_1^- \to \mA_1^- &\pA_1A_0^- \to \mA_0A_0^-
                                      &\pA_0^-\mA_0 A_0^- \to \mA_0
                                        A_0^-\\
      \ppA_2^- &\plA_1^- &\plA_0A_0^- \to \foA_0^-
                                      &\plA_0 A_0^- \to \foA_0^-\\
      A_2^+ &A_0^-\mA_1^+ &A_1A_0^+ \to A_1^f 
                                      &A_0^-\mA_0 A_0^+ \to A_0^-\mA_0^f
    \end{array}
  \end{displaymath}
\end{example}

Before proving the theorem, we explain the meaning of the Priming
Rules.

\begin{lemma}\label{lem:priming-rules}
  One has the following:
  \begin{enumerate}
  \item Priming a surface of shape $A_1$ gives a surface of shape $A^-_0$.
  \item Priming a surface $Y$ of shape $\mA_0$ on the long side $C_2$
    gives a surface $Y'$ and a nef line bundle $L'$ such that
    $(L')^2=0$ and $|L'|$ contracts $Y'$ to $\bP^1$, with the other
    side $C_1$ mapping isomorphically to $\bP^1$.
  \item Priming a surface $Y$ of shape $\mA_0$ on the short side $C_1$
    gives a surface $Y'$ and a nef line bundle $L'$ such that
    $(L')^2=0$ and $|L'|$ contracts $Y'$ to $\bP^1$, with the other
    side $C_2$ folding 2:1 to $\bP^1$.
  \end{enumerate}
\end{lemma}
\begin{proof}
  We proved in (1) in Theorem~\ref{thm:priming} already, see also
  Remark~\ref{rem:A1p}. Parts (2,3) are easy computations.
\end{proof}

\begin{proof}[Proof of Thm.~\ref{thm:compactified-family-primed}]
  Let $f\colon (\cY, \cC + \freps \cB)\to V\cox_M$ be a family for a pure
  shape. It comes with canonical sections: one for a short side of the
  shape, and two disjoint sections for a long side. Now make a
  weighted blow up one of the sections to obtain a family
  $f'\colon (\cY', \cC' + \freps \cB')\to V\cox_M$. Then the sheaf
  $\cL' = \cO_{\cY'}(-2(K_{\cY'/Z}+C')$ is invertible and relatively
  nef. As in proof of Theorem~\ref{thm:priming}, this sheaf is
  relatively semiample and gives a contraction
  $\cY'\to \overline{\cY'}$ to a family
  $\bar f'\colon (\overline{\cY'}, \overline{\cC'} + \freps
  \overline{\cB'})\to V\cox_M$ over the same base. For a reducible fiber
  $Y'=\cup Y'_k$ of the family $f'$, the sheaf $\cL'$ is ample on all
  components $Y'_k$ except possibly on the blown up surface on the
  end. For this surface the resulting surface $\oY'_k$ is given by
  Lemma~\ref{lem:priming-rules}. The other sections of $\cY\to V\cox_M$ map
  to disjoint sections of $\overline{\cY'}\to V\cox_M$.  We then repeat the
  process for the second prime, etc.
\end{proof}

\begin{remark}\label{rem:degns-folded}
  Theorem~\ref{thm:compactified-family-primed} extends to the
  degenerations of surfaces of shapes with folds, e.g. $A_{2n-1}^f$
  as follows: the degenerations are the same as for the shape with a
  long side, but that long side is folded.
\end{remark}

\subsection{A generalized Coxeter fan}
\label{sec:generalized-coxeter}

As Examples~\ref{ex:D4-degn} and \ref{ex:primed-degns} show, the
families $(\cY, \cC+\freps \cB) \to V\cox_M$ over the projective toric
variety for the Coxeter fan have repeating fibers over certain
boundary strata. Here we define
a coarser generalized Coxeter fan and a birational contraction
$\rho\colon V\cox_M\to V\semi_M$ such that the families are constant
on the fibers of $\rho$ and such that the correspondence between the
isomorphism classes of the pairs \ycb and the points of $V\semi_M$ is
finite to one.

\smallskip

The Coxeter fan $\tau\cox$ on the vector space
$N_\bR=\Lambda_\bR=\Lambda^*_\bR$ is obtained by cutting it with the
mirror hyperplanes $\alpha^\perp$ for the roots
$\alpha\in\Lambda$. Another definition is: it is the normal fan to the
\emph{permutahedron}, $\Conv(W_\Lambda.p)$, the convex hull of the
$W_\Lambda$-orbit of a generic point $p$ in the interior of the
positive chamber $C^+=\{\alpha\ge0\}$, where $\Delta=\{\alpha\}$ are the
simple roots. In particular, the maximal cones of the Coxeter fan are
in a bijection with the vertices of $\Conv(W_\Lambda.p)$ and with the
elements of the Weyl group~$W = W_\Lambda$.

\begin{definition}
  For a proper subset $\Delta^0 \subset \Delta$ of
  the simple roots, let $p\in C^+$ be a point such that
  $\alpha\cdot p=0$ for $\alpha\in\Delta^0$ and $\alpha\cdot p>0$ for
  $\alpha\in \Delta\setminus \Delta^0$. A \emph{generalized
    permutahedron} is the convex hull $\Conv(W_\Lambda.p)$ and a
  \emph{generalized Coxeter fan} $\tau\semi$ is defined to be its
  normal fan.
\end{definition}

\begin{definition}
  We will call $\Delta^0$ the \emph{irrelevant subset}.  For
  $S\subset\Delta$ we define its \emph{relevant content} $S\rel$ to be
  the union of the connected components not lying in $\Delta^0$.  We
  will call a connected component $S'$ of $S\subset\Delta$
  \emph{irrelevant} if $S'\subset\Delta^0$.
\end{definition}

The proof of the following lemma is straightforward.

\begin{lemma}
  \begin{enumerate}
  \item The $W$-orbits of cones of $\tau\cox$ are in bijection with
    the subsets $S\subset\Delta$ via: $S\mapsto \Conv(W_S.p)$, where
    $W_S\subset W_\Lambda$ is the Weyl subgroup generated by the
    simple roots $\alpha\in S$. 
  \item The $W$-orbits of cones of $\tau\cox$ are in bijection with
    the subsets without irrelevant connected components. 
  \item The fan $\tau\semi$ is a coarsening of the fan $\tau\cox$ and
    the morphism $\rho\colon V\cox\to V\semi$ of the corresponding
    projective toric varieties is proper and birational.
  \item The image of a torus orbit $O_S\subset V\cox$ is
    $O_{S\rel}\subset V\semi$. One has $\dim O_S=|S|$ and $\dim
    O_{S\rel} = |S\rel|$. If $S$ has no irrelevant components the morphism
    $O_S~\isoto~O_{S\rel}$ is an isomorphism. 
  \end{enumerate}
\end{lemma}

\begin{definition}\label{def:irrelevant-subdiag}
  For a decorated Dynkin diagram of a (possibly primed) shape, we
  define the irrelevant subset $\Delta^0\subset\Delta$ to be the
  set of circled white (i.e. unfilled) nodes.
\end{definition}

\begin{example}
  In the pure $D,\mE$ shapes, and also in the toric $\pA$ shape, the
  interior circled white node is irrelevant, see
  Figs.~\ref{fig:D}, \ref{fig:E}, \ref{fig:pA}. In the toric shapes
  $D'$ and $\pA'$ the irrelevant subset consists of two nodes, see
  Fig.~\ref{fig:pAp}. In the primed shapes there may be more
  irrelevant nodes, cf. Fig~\ref{fig:circled-diags}.
\end{example}

\begin{theorem}\label{thm:fibers-of-semi}
  The pairs in the family $(\cY, \cC+\freps \cB) \to V\cox_M$ are
  isomorphic on each fiber of the contraction
  $\rho\colon V\cox_M\to V\semi_M$.  The correspondence between the
  points of $V\semi_M$ and the isomorphism classes of the pairs \ycb
  is finite to one.
\end{theorem}
\begin{proof}
  Consider a codimension~1 orbit of $V\cox_M$ corresponding to
  setting $a=e^{-\alpha}$ to zero for a single node of the Dynkin
  diagram $p$. By Theorems~\ref{thm:compactified-family-A'DE} and
  \ref{thm:compactified-family-primed} the dimension of the family
  over the boundary stratum drops by 2 instead of the expected 1
  exactly when one of the following happens:
  \begin{enumerate}
  \item In the $\pA$, $D$, $E$ shapes, we remove the corner node
    $p_0$, leaving the circled white node $p''$ isolated.
  \item A single left-most or right-most white node which in our
    shape is primed or doubly primed (so white and circled)
    becomes an isolated $A_1'$ or $\pA_1$ after a node next to it is
    removed. 
  \end{enumerate}
  In both cases this happens precisely when the subdiagram $S=\Delta -
  p$ corresponding to the codimension~1 orbit of $V\cox_M$ has an
  irrelevant component, a single node. 

  We now observe that for any shape the irrelevant subset consists of
  several \emph{disjoint} isolated nodes. There is a drop in the
  moduli count by one for each of them. On the other hand, for the orbits
  $O_S$ for $S$ without irrelevant components, the restriction of the
  family to $O_S$ is the naive family for the Dynkin diagram $S$. The
  set of the isomorphism classes in the latter family is $O_S$ modulo
  a finite Weyl group $W_S\rtimes W_0$ and a finite multiplicative
  group $\mu_S$. This proves the statement.
\end{proof}

\subsection{Description of the compactified moduli space of \ade pairs}
\label{sec:compact-ade-moduli}

We now prove Theorem~\ref{thm:compact-ade-moduli-custom}.
In fact we will prove the following slightly stronger result, 
which contains more information about the
toric primed shapes.

\begin{theorem}\label{thm:compact-ade-moduli}
  For each \ade shape the moduli space $M\slc_{ADE}$ is proper and the
  stable limit of \ade pairs are stable \ade pairs. 
  \begin{enumerate}
  \item For the pure \ade shapes, the normalization $(M\slc_{ADE})^{\nu}$
    is $V\semi_\Lambda/W_\Lambda$, a $W_\Lambda$-quotient of the
    projective toric variety for the generalized Coxeter fan.
  \item For the toric primed shapes $\pA_{2n-1}$, $\pA_{2n-2}^-$,
    $\pA_{2n-1}'$, $D_{2n}'$ with $n\ge3$, the normalization
    $(M\slc_{ADE})^\nu$ is $V\semi_{\Lambda'}/W_\Lambda$ with the lattice
    $\Lambda'$ described in \eqref{thm:moduli-ade-toric-primed}.
  \item For an arbitrary primed shape, the normalization
    $(M\slc_{ADE})^\nu$ is $V\semi_{\Lambda'}/W_\Lambda\rtimes W_0$,
    for a lattice extension $\Lambda'\supset\Lambda$. The lattice
    $\Lambda'$ and the Weyl group $W_0$ are described in
    \eqref{thm:moduli-ade-primed}.
  \end{enumerate}
\end{theorem}
\begin{proof}
  (1) By Theorems~\ref{thm:compactified-family-A},
  \ref{thm:compactified-family-A'DE},
  \ref{thm:compactified-family-primed}, every one-parameter family of
  \ade pairs has a limit which is a stable \ade pair, since it has a
  limit (after a finite base change) in the family over $V\cox_M$.  By
  Theorem~\ref{thm:fibers-of-semi} the classifying morphism
  $\phi\colon V\semi_M/W \to M\slc_{ADE}$ is finite-to-one.  By
  Theorem~\ref{thm:moduli-ade-pure} on a dense open set it equals
  $\Hom(M, \bC^*)/W \to \Hom(\Lambda,\bC^*)/W$, and it factors through
  the homomorphism $\Hom(M, \bC^*) \to \Hom(\Lambda,\bC^*)$, the
  quotient by the finite multiplicative group
  $\mu:=\Hom(M/\Lambda, \bC^*)$. Thus, $\phi$ factors through
  $V\semi_\Lambda = V\semi_M / \mu$, and the morphism
  $V\semi_M\to M\slc_{ADE}$ is finite to one and an isomorphism over
  an open dense subset. Since $V\semi_M$ is normal, it is the
  normalization of $M\slc_{ADE}$.

  Parts (2) and (3) are proved the same way.
\end{proof}

\begin{remark}\label{rem:compact-moduli-folded}
  Theorem~\ref{thm:compact-ade-moduli} extends to surfaces of shapes
  with folds, e.g. $A_{2n-1}^f$, cf. Remark~\ref{rem:degns-folded}.
\end{remark}

\section{Canonical families and their compactifications}
\label{sec:canonical-families}

In the previous Section we compactified the stack of \ade pairs --
which for the pure shapes is $[\bA^n: \mu_\Lambda]$ -- and extended
the \emph{naive family} over it to a family of stable pairs.  However,
$[\bA^n: \mu_\Lambda]$ has many automorphisms, and consequently in the equation
$f$ of the divisor $B$ we have a lot of freedom for the coefficients
$c_i = c_i(\chi_1, \dotsc, \chi_n)$ as polynomials in the fundamental
characters. Many of these choices extend to the compactification.

\begin{example}
  For the $A_1$ shape the moduli stack is $[\bA^1:\mu_2]$, and we
  write $\bA^1$ as the quotient of the torus $\bC^*_t$ by the Weyl
    group $W_{\Lambda}=\bZ_2$, $t\to t\inv$. The compactification is
  $[\bP^1:\mu_2]$. The equation of $B$ is $f=1+c_1x+x^2$, where in the naive
  family we have $c_1=\chi_1 = t + t\inv$.  We can apply to $\bA^1$ an
  automorphism $c_1 \mapsto ac_1 + b$, with $a,b \in \bC$ and $a\ne0$, then pull the
  family back to $\bC^*_t$. This automorphism extends to the
  compactification $\bP^1$ of $\bA^1$, but the coordinate
  change is not compatible with the $\mu_\Lambda$-action (i.e. with the
  $(\Lambda^*/\Lambda)$-grading) unless $b=0$, since $(-1)\in \mu_2$
  acts by $ac_1+b\mapsto -ac_1+b$, so it is not an automorphism of
  $[\bA^1:\mu_2]$. In this case the naive family is 
  effectively unique.

  However, for the root systems $D_n$ ($n\ge5$) and $E_n$ ($n=6,7,8$)
  there exist dominant weights $\lambda<\vpi_i$ lying below the
  fundamental weights and with $\lambda \equiv \vpi_i$ in
  $\Lambda^*/\Lambda$, and we can modify the coefficients $c_i=\chi_i$
  by adding linear combinations of their characters
  $\chi(\lambda)$. For example, for $E_8$ there are 23 dominant
  weights $\lambda$ below $\vpi_0$, and $\Lambda^*/\Lambda=0$.
  Counting all fundamental weights $\vpi$ and their lower terms, there
  is a $\bC^{51}$ worth of choices for
  $c = \chi(\vpi)+ \sum_{\lambda<\vpi} c_\lambda \chi(\lambda)$, all
  extending to automorphisms of our moduli compactification.
\end{example}

In this Section we show that the naive family can be deformed in an
essentially unique way so that the new family, which we call \emph{the
  canonical family}, has the following wonderful property: the
discriminant locus in $T_\Lambda$, over which the divisor $B$ in our 
\ade pairs become singular, is a union of root hypertori
$\{e^\alpha=1\}$, with $\alpha$ going over the roots of the lattice
$\Lambda$.

We then show that the canonical family extends to the compactification
and that on the boundary strata it restricts to the canonical families
for smaller Dynkin diagrams.

\subsection{Two notions of the discriminant}
\label{sec:two-discriminants}

Let $f(x,y)$ be one of the polynomials in the equations
\eqref{eq:A-family}--\eqref{eq:E-families}, which we related to the
root lattices $\Lambda = A_n$, $D_n$, $E_n$.  There are two
different notions of the discriminant in this situation:
\begin{enumerate}
\item The discriminant $\Discr(f)$ of a polynomial $f(x,y)$. This is a
  polynomial in the coefficients $c_i$ of $f$ for which the zero set
  of $f$ on $Y\setminus C$ is singular. If $c_i = c_i(\chi_j)$ are
  polynomials in the fundamental characters of the lattice $\Lambda$,
  then $\Discr(f)$ becomes a polynomial in $\chi_j$.
\item The discriminant $\Discr(\Lambda)$ of the lattice, the square of the
  expression
  \begin{displaymath}
    \prod_{\alpha\in \Phi^+} (e^{\alpha/2} - e^{-\alpha/2}) =
    \sum_{w\in W_{\Lambda}} \varepsilon(w) e^{w.\rho},
    \quad \text{where }
    \rho = \sum_{\vpi\in\Pi} \vpi = \frac12 \sum_{\alpha\in\Phi^+} \alpha.
  \end{displaymath}
  appearing in the Weyl character formula.  $\Discr(\Lambda)$ is
  $W_\Lambda$-invariant, so it is also a polynomial in the
  fundamental characters.  The zero set of $\Discr(\Lambda)$ is
  obviously the union of the root hypertori $e^\alpha=1$.
\end{enumerate}

The following theorem forms the first part of Theorem~\ref{thm:canonical-families-summary}.
We prove it separately for the $A,D,E$ shapes in Theorems~\ref{thm:A-discr},
 \ref{thm:D-discr}, \ref{thm:E-discr}, respectively.

\begin{theorem}\label{thm:two-discrs}
  For each \ade pair of pure shape, there exists a unique deformation of the form
  $c = \chi(\varpi) + \text{(lower terms)}$ of the naive equation 
  such that $\Discr(f) = \Discr(\Lambda)$.
\end{theorem}

\subsection{Canonical families}
\label{subsec:canonical-families}

\begin{theorem}[$A$ shapes]
  \label{thm:A-discr}
  For the pure shapes $A_n^?$, resp. $\mA_n^?$, in Theorem~\ref{thm:two-discrs} one
  has $c_i = \chi_i$, resp. $c_i = \chi_{i-1}$.
\end{theorem}
\begin{proof}
  For $A_n^?$ the curve $ -\frac14 y^2 + c(x)$ is singular iff $c(x)$
  has a double root. If
  \begin{displaymath}
    c(x) = \prod_{i=1}^{n+1} (x+t_i) =
    1 + \chi_1x + \dotsc + \chi_nx^n + x^{n+1} 
  \end{displaymath}
  then this happens iff $e^{e_i-e_j} = t_i/t_j = 1$. Here, $e_i-e_j$
  are precisely the roots of $A_n$. Thus, the statement holds for
  $c_i = \chi_i$. For $A_n$ there are no lower weights below the
  fundamental weights, so the solution is unique.
  The open sets $Y\setminus C$ for the shapes
  $A_n^?$ and $\mA_n^?$ are the same, so this argument applies to the
  $\mA_n^?$ shapes as well.
\end{proof}

Recall from section~\ref{sec:naive-families} that for the $D$ and $E$
shapes there are two equivalent forms of the equation: $F(x,y,z)$ and
$f(x,y)$, and the latter is obtained from the former by
completing the square in $z$.  For $D_n$ one has the following root
lattice, weight lattice, Weyl group, fundamental roots $\alpha_i$, and
fundamental weights $\vpi_i$:
\begin{eqnarray*}
  &&\Lambda = \left\{ (a_i)\in\bZ^n=\oplus\bZ e_i \mid \sum a_i \text{ is even} \right\},
     \quad \Lambda^* = \bZ^n + \frac12\sum e_i.\\
  &&W_{\Lambda}=\bZ_2^{n-1}\rtimes S_n,\\
  &&\alpha_{n-2-i} = e_i - e_{i+1} \text{ for } i\le n-2, \ 
     \alpha_1' = e_{n-1}-e_n,\ 
     \alpha''=e_{n-1}+e_n.\\
  &&\vpi_{n-2-i} = \sum_{k=1}^i e_k \text{ for } i\le n-2, \
     \vpi'_1 = \frac12\left( -e_n + \sum_{i=1}^{n-1}e_i \right), \ 
     \vpi'' = \frac12\left( \sum_{i=1}^{n} e_i\right).
\end{eqnarray*}

Denoting by $\sigma_i$ the $i$-th symmetric polynomial, the
fundamental characters are
\begin{eqnarray*}
  \chi_i = \sigma_{n-2-i}(t_k^\pm) \text{ for } i\le n-2, \ 
  \chi'_1 = \frac{\sum_{s\ge0} \sigma_{2s+1}(t_k)
  }{\sqrt{\prod t_k}}, \ 
  \chi'' = \frac{\sum_{s\ge0} \sigma_{2s}(t_k) }{\sqrt{\prod t_k}},
\end{eqnarray*}
where $t_k^\pm$ are $t_1,t_1\inv, \dotsc, t_n,t_n\inv$.

\begin{definition}
Let $f_k(x)$ be the polynomials defined recursively by $f_0=1$,
$f_1=x$, and $f_{k+2} = xf_{k+1}-f_k$. These are the \emph{Fibonacci polynomials},
except for the signs and a shift in degrees by 1.
One has $f_2=x^2-1$, $f_3=x^3-2x$, etc.
\end{definition}

\begin{theorem}[$D$ shapes]\label{thm:D-discr}
  For the $D_n^?$ shapes, in Theorem~\ref{thm:two-discrs}
  one has $c'_1=\chi'_1$, $c''=\chi''$, and 
  the expression for $c(x)$ can be obtained
  from the generating function
  \begin{displaymath}
    c(x,\chi) = \sum_{i,j\ge0} c_{ij}x^i\chi^j = \sum_{i\ge0} c_i(\chi)x^i
    = \sum_{j\ge0} p_j(x) \chi^j
  \end{displaymath}
  by substituting $\chi_j$ for $\chi^j$ and setting $\chi_{n-2}=1$ and
  $\chi_j=0$ for $j>n-2$. One has
  \begin{enumerate}
  \item $\displaystyle c(x,\chi) = \frac1{ (1-\chi^2)(1-x\chi+\chi^2)}$ and
    $\displaystyle c_i(\chi) = \frac{\chi^i}{(1-\chi^4)(1+\chi^2)^{i}}$.
  \item $p_{2k}(x) = f_k^2$ and $p_{2k+1} = f_kf_{k+1}$.
  \item $c_{ij}=0$ if $j-i$ is odd or $i>j$. Otherwise,
    \begin{displaymath}
      c_{i,i+2k} = \sum_{p\ge0}^{k} (-1)^p {i+p\choose i}
      = (-1)^{k} \sum_{q\ge0} {i+k - 1 - 2q \choose i-1}.
    \end{displaymath}
  \end{enumerate}
  The central fiber has a $D_n$ singularity at the point
  $(x,y,z)=(-2,-2^{n-3},-2^{n-3})$. 
\end{theorem}

\begin{example}\label{ex:D7}
For $D_7$ we obtain for the following expressions for
$c(x)$:
\begin{eqnarray*}
  && \chi_0 + \chi_1 x + \chi_2 x^2 + \chi_3 x(x^2-1) + \chi_4
           (x^2-1)^2 + (x^2-1)(x^3-2x) = \\
  && (\chi_0 + \chi_4) + ( \chi_1-\chi_3+2)x + (\chi_2-2\chi_4)x^2 +
      (\chi_3-3)x^3 + \chi_4 x^4 + x^5
\end{eqnarray*}
and for any lower $D_n$ the formulas can be obtained from these by
truncation. 
\end{example}

\begin{proof}[Proof of Thm.~\ref{thm:D-discr}]
  We start with the polynomial $f(x,y)$ in equation
  \eqref{eq:D-families}. As a quadratic polynomial in $y$, it
  represents a curve which is a double cover of $\bA^1$. This curve is
  singular when the following polynomial in $x$ 
  \begin{displaymath}
    \Discr_y(f) = (x^2-4)c(x) + c'_1c'' x + c'_1{}^2 + c''{}^2
  \end{displaymath}
  has a double root. On the other hand, the polynomial
  $p(x) = \prod_{i=1}^n (x+t_i+t_i\inv)$ has a double root iff some
  $t_i+t_i\inv = t_j+t_j\inv$, i.e. $t_it_j\inv=1$ or
  $t_it_j=1$. These are exactly the root hypertori for the root lattice
  $D_n$.

  Thus, $\Discr(f) = \Discr(\Lambda)$ iff $\Discr_y(f) = p(x)$.  The
  coefficients of $p(x)$ are $\sigma_i(t_k+t_k\inv)$, and they are
  invariant under the $W(D_n)$-action, so they are polynomials in the
  fundamental characters $\chi_i$ listed above. The rest of the proof
  is a  combinatorial manipulation to get the exact formula.
   From this procedure we see that the
  solution is unique.
\end{proof}

\begin{theorem}[$E$ shapes]\label{thm:E-discr}
  For the $\mE_n^?$ shapes, in Theorem~\ref{thm:two-discrs} one has
  \begin{enumerate}
  \item [$\bm{E_6}$:]
    \begin{align*}
      & c'' = \chi'' -6 \qquad
	c'_2= \chi'_2 \qquad
      c'_1 = \chi'_1 -\chi_2\\       
      &c_0 = \chi_0 -3\chi'' +9\qquad
        c_1 = \chi_1 -\chi'_2\qquad
       c_2 = \chi_2 \qquad
    \end{align*}

  \item [$\bm{E_7}$:]
    \begin{align*}
  c''  &=  \chi'' -6 \chi_3 \qquad
  c'_2 = \chi'_2 - 25 \qquad
  c'_1 = \chi'_1 - \chi_2 -16 \chi'_2  +206 \\
  c_0  &=  \chi_0 - 3 \chi(\vpi''+\vpi_3) + \chi(2\vpi'_2) -
  12 \chi'_1            
          + 9 \chi(2 \vpi_3)
  + 16 \chi_2  + 69 \chi'_2 - 548 \\
  c_1  &=  \chi_1 - \chi(\vpi'_2+\vpi_3) - 6 \chi'' 
    +28 \chi_3\qquad
  c_2  =  \chi_2 - 2\chi'_2 +23 \qquad
  c_3  =  \chi_3
    \end{align*}
    
  \item [$\bm{E_8}$:]
    \begin{align*}
  c''  &=  \chi'' -6 \chi_3 -35 \chi'_2 +920 \chi_4 - 57505\qquad
  c'_2 = \chi'_2 - 25 \chi_4 + 2325 \\
  c'_1 &= \chi'_1 - \chi_2 -16 \chi(\vpi'_2+\vpi_4) -44 \chi'' +206
         \chi(2\vpi_4) + \\
  &+360\chi_3 +2196\chi'_2 - 51246 \chi_4 +2401900 \\
  c_0  &=  \chi_0 - 3 \chi(\vpi''+\vpi_3) + \chi(2\vpi'_2+\vpi_4) - 12 \chi(\vpi'_1+\vpi_4) 
          -28 \chi(\vpi'_2+\vpi'') + \\
          &+ 9 \chi(2 \vpi_3)
  + 16 \chi(\vpi_2+\vpi_4) - 68 \chi_1 + 69 \chi(\vpi'_2+2\vpi_4) 
     +212\chi(\vpi'_2+\vpi_3) +\\
     &+ 1024 \chi(\vpi''+\vpi_4) 
  + 236 \chi(2\vpi'_2) + 2453 \chi'_1 - 548 \chi(3\vpi_4) -\\
     &- 5228 \chi(\vpi_3+\vpi_4) 
     - 1507\chi_2
  - 42656 \chi(\vpi'_2+\vpi_4) - 107636 \chi'' +\\
     &+ 488553 \chi(2\vpi_4) + 640064 \chi_3
     + 2988404 \chi'_2 - 52027360 \chi_4 + 1484779780\\
  c_1  &=  \chi_1 - \chi(\vpi'_2+\vpi_3) - 6 \chi(\vpi''+\vpi_4) +2
         \chi(2\vpi'_2) -17 \chi'_1 +\\
  &  +28 \chi(\vpi_3+\vpi_4) - 79 \chi_2 +383 \chi(\vpi'_2+\vpi_4)
    +1429 \chi''-\\ 
  &  - 4414 \chi(2\vpi_4) +84 \chi_3 -49768 \chi'_2 +271934 \chi_4 +4528192\\
  c_2  &=  \chi_2 - 2\chi(\vpi'_2+\vpi_4) - 9\chi'' +23 \chi(2\vpi_4)
    - 114 \chi_3
     + 601 \chi'_2 + 7673 \chi_4 - 955955\\
  c_3  &=  \chi_3 - 3\chi'_2 - 170 \chi_4 + 23405 \qquad
  c_4  = \chi_4 - 248 
\end{align*}
  \end{enumerate}

The central fiber with an $E_6$, resp. $E_7$, resp. $E_8$
singularity is
  \begin{enumerate}
  \item [$\bm{E_6}$:]
    \begin{math}
      xyz = z^2 + 72z + y^3 + 27y^2 + 324 y + 2700 + 324x + 27x^2+ x^3
    \end{math}
at $(x,y,z)=(-6,-6,-18)$.
  \item [$\bm{E_7}$:]
    \begin{math}
      xyz = z^2+576 z
      +y^3 +108 y^2 +5184 y
      +193536 +17280 x +1296 x^2 +56 x^3 + x^4
    \end{math}
    at $(x,y,z)=(-12,-24,-144)$.    
  \item [$\bm{E_8}$:] $xyz=z^2+y^3+x^5$ at $(0,0,0)$.
  \end{enumerate}

\end{theorem}

The formulas for the polynomials $F(x,y,z)$ were found by Etingof,
Oblomkov, Rains in \cite{etingof2007generalized-double} in a
completely different context, as relations for the centers of certain
non-commutative algebras associated to affine star-shaped Dynkin
diagrams $\wD_4,\wE_6,\wE_7,\wE_8$. We found them independently, using
Tjurina's construction as explained below. The answer given above is
in terms of the additive basis of characters of dominant weights,
which is needed for computing the degenerations in
Theorem~\ref{thm:canonical-family-compactified}. Once we recomputed
our answer in the polynomial basis of the fundamental characters
$\chi_i$ and did a web search for the largest coefficient, a
single mathematical match came up, to 
\cite[Sec. 9]{etingof2007generalized-double}.

\medskip

Before proving the theorem, we begin with preliminary observations and
lemmas.

\smallskip

In \cite{tjurina1970resolution-of-singularities} Tjurina constructed a
versal deformation of an $E_8$ singularity as a family over $\bA^8$ =
the parameter space for 8 smooth points on a cuspidal cubic $C$ (note
that one has $C\setminus{\rm cusp}\simeq\bA^1$). See also
\cite[p.190]{demazure1980seminaire-sur-les-singularites}.  The
discriminant locus of this family is a union of affine hyperplanes
$e^\alpha=0$ for the roots $\alpha\in E_8$.  Our observation is that replacing
the cuspidal cubic by a nodal cubic~$C$ (so that
$C\setminus{\rm node}\simeq\bC^*$) gives a multiplicative version of
Tjurina's family over $(\bC^*)^8$ that we are after.

The lattice $E_8$ can be realized as an intermediate sublattice of
index 3 in $A_8\subset E_8\subset A_8^*$. The lattice $A_8^*$ is
generated by $e_i-p$, where $1 \leq i \leq 9$ and $p=\frac19\sum_{i=1}^9 e_i$. 
The lattice
$A_8$ is generated by $e_i-e_j$, and the intermediate lattice $E_8$ is
obtained by adding $\ell - e_1-e_2-e_3$, where $\ell=3p$.

Now let $C$ be an irreducible curve of genus 1, so $C$ is either
smooth, or has a node, or a cusp. Let $G=\Pic^0 C$, so either an
elliptic curve (with a choice of 0), or $\bC^*\ni 1$, or $\bG_a\ni
0$. The nonsingular locus $C^0$ is a $G$-torsor.

\begin{lemma}\label{lem:root-tensor-group}
  Let $A_n, E_8$ be the standard root lattices, and $A_n^*$ the dual
  lattice. Then:
  \begin{enumerate}
  \item $\Hom(A_n,G) = A_n^*\otimes G = G^{n+1}/\diag G = (C^0)^{n+1}/ G$
    parametrizes $(n+1)$ nonsingular points $P_i$ on $C$ modulo
    translations by $G$.
  \item
    $\Hom(A_n^*, G) = A_n\otimes G = \{(g_1,\dotsc, g_{n+1}) \mid \sum
    g_i = 0\}$ parametrizes the choice of an origin $P_0\in C^0$ plus
    $(n+1)$ nonsingular points $P_i\in C^0$ such that
    $(n+1)P_0 \sim \sum P_i$.
  \item $\Hom(E_8,G) = E_8\otimes G$ parametrizes embeddings
    $C\subset\bP^2$ as a cubic curve plus 8 points $P_i\in C^0$, or
    equivalently embeddings $C\subset\bP^2$ plus 9 points $P_i\in C^0$
    such that $\prod_{i=1}^9 P_i=1$ in the group law of $C^0$. Thus,
    $P_9$ is the 9th base point of the pencil $|C|$ of cubic curves on
    $\bP^2$ through $P_1,\dotsc, P_8$.
  \end{enumerate}
\end{lemma}
\begin{proof}
  We have $A_n^*=\bZ^{n+1}/\diag\bZ$ and
  $A_n=\{(a_1,\dotsc,a_{n+1})\in\bZ^{n+1} \mid \sum a_i=0\}$, so (1)
  and (2) follow. Hence, $\Hom(A_8^*, G)$ parametrizes embeddings
  $C\subset\bP^2$ with 8 points and a choice of a flex, and
  $E_8\otimes G = \Hom(A_8^*,G)/G[3]$ forgets the flex.
\end{proof}

Thus, the torus $T_{A_8^*}$ parametrizes 8 smooth points
$P_1,\dotsc,P_8$ on a nodal cubic curve $C\subset\bP^2$ with a chosen
flex, and the torus $T_{E_8}$ the same, but forgetting the flex. 
We now take a 
concrete rational nodal cubic $C\subset\bP^2$ given by the equation
$g_0 = -uvw+v^3+w^3$ with a rational parametrization
$(u:v:w)=(t^3-1 : t : -t^2)$, so that the singular point of $C$ is
$(1:0:0)$ corresponding to $t=0$ or $\infty$.
Now consider a family over $\bA^8$ of cubics
\begin{displaymath}
  g_1 = \sum_{i,j\ge0, \ i+j\le3} a_{ij} u^{3-i-j} v^i w^j, 
  \quad a_{00} = 1, \ a_{11} = 0
\end{displaymath}
Then any pencil of cubic curves, parametrized by $x$, with a smooth generic fiber which has
$C$ at $x=\infty$ has a unique representation by a polynomial
$g(x;u,v,w) = x g_0+ g_1$. It is a simple exercise to put this pencil
into the Weierstrass form $\phi^2=y^3+ A(x)y + B(x)$ using Nagell's
algorithm or simply by using the \cite{sagemath} function
WeierstrassForm.  The polynomials $A(x)$, $B(x)$ have degrees 4 and 5
(not 6 since $C$ is singular).  The following is an easy explicit
computation:

\begin{lemma}\label{lem:change-coords}
  There is a unique change of coordinates of the form $x\mapsto x+d$,
  $y\mapsto y + ax^2 + bx + c$ which leaves the fiber $C$ at
  $x=\infty$ in the pencil intact and takes the polynomial
  $f(x,y) = y^3 +A(x)y + B(x)$ into the form of the equation
  \eqref{eq:E-families} for $E_8$.
\end{lemma}

We will use this to build a family over $(\bC^*)^8$ with the required properties. We
pick $t_1,\dotsc, t_8 $ in $\bC^*$ arbitrarily and then also $t_9$ so
that $\prod_{i=1}^9 t_i=1$. Using the rational parametrization of the
nodal cubic $C$, this gives 9 smooth points
$P_1,\dotsc, P_9 \in C$.

\begin{lemma}\label{lem:pass-thru-pts}
  The pencil $g(x;u,v,w)$ passes through the points $P_1,\dotsc, P_9$ iff
  \begin{align*}
    & a_{10} = \sigma_{8}\quad
      a_{01} = \sigma_{1}\quad  
      a_{21} = -\sigma_{2} + \sigma_{5} - \sigma_{8}\quad
      a_{12} = -\sigma_1 + \sigma_4 - \sigma_7  \\
    & a_{30} = -3 + \sigma_{6}\quad  
      a_{03} = -3 + \sigma_3\quad 
      a_{20} = -\sigma_{1} + \sigma_{7}\quad
      a_{02} = \sigma_2 - \sigma_8,
  \end{align*}
  where $\sigma_i$ are the elementary symmetric polynomials in
  $t_1,\dotsc, t_9$.
\end{lemma}
\begin{proof}
  We plug the rational parametrization
  $(u:v:w)=(t^3-1 : t : -t^2)$ into $g_1(u,v,w)$ to obtain a monic
  polynomial of degree 9 with constant coefficient $-1$ which we set
  equal to $\prod_{k=1}^9 (x-t_k) = \sum_{n=0}^9 (-1)^{n+1}\sigma_i x^i$. Then
  we solve the resulting linear equations for~$a_{ij}$.
\end{proof}

\begin{proof}[Proof of Thm.~\ref{thm:E-discr}]
  Define the pencil $g(x;u,v,w)$ as in Lemma~\ref{lem:pass-thru-pts},
  convert into Weierstrass form $\phi^2=y^3+ A(x)y + B(x)$,
  then apply  Lemma~\ref{lem:change-coords} to obtain 
   a polynomial $f(x,y)$ in the form of 
  equation~\eqref{eq:E-families}. The coefficients $c_i$ in the resulting
  expression for $f(x,y)$ satisfy $c_i\in \bC[A_8^*]^{S_9}$, so
  we obtain a family parametrized by the torus
  $T(A_8^*)=\Hom(A_8^*,\bC^*)$.
  The final very computationally intensive step, accomplished
  using \cite{sagemath}, is to
  rewrite it in terms of the characters of $E_8$.

  We now prove that the discriminant $\Discr(f)$ of this family of
  polynomials coincides with the discriminant $\Discr(E_8)$.
  We have a trivial family $\wt\cX^0=\bP^2\times T_{A_8^*}\to T_{A_8^*}$ with 9
  sections, call them $s_1,\dotsc, s_9$, corresponding to the points
  $P_i\in C^0$. Let $\wt\cX^n$, $1\le n\le 9$, be the family obtained
  by performing a smooth blowup of $\wt\cX^{n-1}$ along the strict
  preimage of $s_n$.

  On each fiber the points $P_1,\dotsc,P_8\in\bP^2$ are in an
  \emph{almost general position} because they lie on an irreducible
  cubic (see
  \cite[p.39]{demazure1980seminaire-sur-les-singularites}). This
  means that $-K_{\cX^8}$ is relatively nef and semiample, and
  defines a contraction to a family $\cX^8\to T(A_8^*)$ of del Pezzo
  surfaces with relatively ample $-K_{X^8}$ and with Du Val
  singularities.

  On the other hand, $\wt\cX^9$ is a family of Jacobian elliptic surfaces
  with a section $s_9$ corresponding to the last point $P_9$. 
  The linear system $|Ns_9|$, $N\gg0$ gives a contraction
  $\wt\cX^9\to\cX^9$ to a family  of surfaces with \ade singularities. 
  Let
  $\wt\iota_9$ be an elliptic involution $w\mapsto -w$ for this choice
  of a zero section. It descends to an involution $\wt\iota_8$ of
  $\wt\cX^8$ which in turn descends to an involution $\iota_8$ of
  $\cX^8$. It is easy to see that the quotients are families of the
  surfaces $X^9/\iota_9 = \bF_2$ and
  $Y^8=X^8/\iota_8 = \bF_2^0 = \bP(1,1,2)$. The families of the
  polynomials $f(x,y)$ written above are just the equations of the
  branch curves. On each fiber, the ramification curve passes through
  the singular point of the nodal cubic $C$. Blowing up the image of
  this point on $Y^8$ finally gives the toric $E_8$-surface $Y$ as in
  Fig.~\ref{fig:E} corresponding to the Newton polytope of $f(x,y)$.

  The branch curve $f=0$ is singular iff the double cover $X^8$
  is singular. This happens precisely when the points $P_1,\dotsc,
  P_8$ are not in general position:
  \begin{enumerate}
  \item some 3 out of 9 points $P_{i},P_{j},P_{k}$ lie on a line
    $\iff$ the complementary 6 points lie on a conic $\iff$
    $t_{i}t_{j}t_{k}=1$.
  \item some 2 out of 9 points $P_i=P_j$ ($i>j$) coincide $\iff$ the
    complementary 7 points lie on a cubic which also has a node at
    $P_j$ $\iff t_i=t_j$.
  \end{enumerate}
  These are precisely the root loci for the roots of $E_8$ in terms of
  the lattice $A_8^*$, with $t_i = e^{e_i - p}$. For our explicit parametrization of
  the nodal cubic $C$ this can be seen from 
    \begin{displaymath}
    \det
    \begin{vmatrix}
      t_i^3-1 & t_i & -t_i^2 \\
      t_j^3-1 & t_j & -t_j^2 \\
      t_k^3-1 & t_k & -t_k^2 
    \end{vmatrix}
    = (t_it_jt_k-1) (t_i-t_j)(t_i-t_k)(t_j-t_k).
  \end{displaymath}
  This shows that $\Discr(f)$ is a product of the equations $(e^\alpha-1)$
  of the root loci, and it is easy to see that they appear with
  multiplicity 1. Thus, $\Discr(f)=\pm \Discr(E_8)$. 

  This completes the proof in the $E_8$ case.  The $E_7$ and $E_6$
  cases are obtained as degenerations of this construction. In the
  $E_7$ we blow up 7 smooth points of the cubic $C$ and the node
  $P_8$. Then there exists a unique point $P_9$ which is infinitely
  near to $P_8$ such that all the cubics in the pencil
  $|C-P_1-\dotsb P_8|$ pass through $P_9$. In other words, $P_9$ is a
  point on the exceptional divisor $E_8$ of the blowup at $P_8$
  corresponding to a direction $t_9\ne0,\infty$ at $P_8$ for which we
  can write an explicit equation. Blowing up at $P_9$ gives an
  elliptic surface $\wt\cX^9\to\bP^1$ with a zero section and an elliptic
  involution.  The preimage of $C$ on $\wt\cX^9$ is an $I_2$ Kodaira
  fiber, instead of an $I_1$ fiber in the $E_8$ case.  In the same way
  as above, the discriminant locus is a union 
  of root loci for the roots of~$E_7$.

  The $E_6$ case is a further degeneration. We pick 6 smooth points on
  $C$ plus the node $P_7$ plus an infinitely near point $P_8\to P_7$
  corresponding to one of the directions at the node. Then there
  exists a unique infinitely near point $P_9\to P_8$ such that all the
  cubics in the pencil $|C-P_1-\dotsb P_8|$ pass through
  $P_9$. Blowing up at $P_9$ gives an elliptic surface $\wt\cX^9\to\bP^1$
  with a zero section and an elliptic involution.  The preimage of $C$
  on $\wt\cX^9$ is an $I_3$ Kodaira fiber.  In the same way as above,
  the discriminant locus is a union of root loci for the roots of
  $E_6$.

  For the uniqueness, write
  $c_i = \chi_i + \sum_{\lambda<\vpi_i} c_{i,\lambda} \chi(\lambda)$.
  The weights $\lambda<\vpi_i$ all lie below $\vpi_0$, there are 23 of
  them, and the partial order on them is described in
  Remark~\ref{rem:weights-below-fundamental}. Equating
  $\Discr(f) = \Discr(\Lambda)$ gives a system of polynomial
  equations in $c_{i,\lambda}$ which is upper triangular: 
  There is a linear equation 
  for the highest coefficient $c_{i,\lambda}$ with no other
  coefficients present, so with a unique solution. Then the equation
  for the next coefficient $c_{j,\lambda'}$ is linear with a unique
  solution once the higher coefficients are known, etc. The solutions
  are obtained recursively, in a unique way at every step.
\end{proof}

\subsection{Compactifications of the canonical families}
\label{sec:degn-canonical-families}

In this subsection we prove the remaining portion of 
Theorem~\ref{thm:canonical-families-summary}.

\begin{theorem}\label{thm:canonical-family-compactified}
  The canonical family extends to the compactifications $V_M\cox$ of
  Theorems~\ref{thm:compactified-family-A},
  \ref{thm:compactified-family-A'DE},
  \ref{thm:compactified-family-primed}. The restriction of the
  compactified canonical family to a boundary stratum is the canonical
  family for a smaller Dynkin diagram.
\end{theorem}
\begin{proof}
  For the compactification we use exactly the same formulas as in
  Theorems~\ref{thm:compactified-family-A},
  \ref{thm:compactified-family-A'DE},
  \ref{thm:compactified-family-primed}, and the proofs go through
  unchanged. Indeed, the only fact we used was that the leading
  monomial in each coefficient $c$ is $e^{\vpi}$, and that the other
  monomials are of the form $e^w$ for some weights of the form
  $w=\vpi - \alpha- \sum_{\beta} n_\beta \beta$. These are
  automatically satisfied if we modify $\chi(\vpi)$ only by adding
  characters of some lower weights $\lambda<\vpi$.

  For the fact that a canonical family restricts to canonical families
  on the boundary strata, a sketch of a possible proof, which can be
  made precise, is that the defining property of the canonical family
  is automatically satisfied for the restrictions.  Instead, we
  check the equations directly.

  For $A_n$ the check is immediate: the coefficients $\hchi_i$ of
  equation~\eqref{eq:A-families-uv} restrict to $\hchi_i$ by
  Lemma~\ref{lem:restrict-chi}, so \eqref{eq:A-families-uv} restricts
  to the $A_m$ family for a smaller $A_m$ diagram.

  For $D_n$ there are no dominant weights below $\vpi'$, $\vpi''$, and
  the dominant weights below $\vpi_i$ are $\vpi_{i+2}$, $\vpi_{i+4}$,
  etc., with the relations
  \begin{equation}\label{eq:D-relation}
    \vpi_i - \vpi_{i+2} = \alpha' + \alpha'' + 2\alpha_0 + \dotsb
    +2\alpha_i + \alpha_{i+1}.
  \end{equation}
  By Lemma~\ref{lem:restrict-chi}, the character
  $\chi(\lambda) = e^{\lambda} \hchi(\lambda)$ 
  under the degeneration $a_i=e^{-\alpha_i} \to 0$ goes to:
  \begin{enumerate}
  \item $0$ if in the expression $\lambda = \vpi - \sum
    n_\alpha\alpha$ one has $n_{\alpha_i} > 0$, or to 
  \item $\hchi(p(\lambda))$ if $n_{\alpha_i}=0$, where $p(\vpi_i) = 0$
    and $p(\vpi_j) = \vpi_j$ for $j\ne i$. 
  \end{enumerate}
  Thus, under the degenerations $a'=0$, resp. $a''=0$ \emph{all} the
  lower weights disappear, and we are left with an equation for the
  $\pA_{n-1}$, resp. $A_{n-1}$ family. Under the degeneration $a_i=0$,
  the limit surface has two components, and on the left, resp. right,
  surface the equation becomes the $D_{i+2}$, resp. $A_{n-i-3}$ family
  if $i>0$. For $i=0$ we get the equations of $A_1$ and $A_{n-3}$.

  \ifshortversion
  \else
\begin{table}[ht!]
  \setlength{\tabcolsep}{3pt}
  \centering
  \begin{tabular}{|lc|l|cccccccc|l|}
    \hline
    higher
    & &lower
    &$\alpha''$&$\alpha'_2$&$\alpha'_1$&$\alpha_0$&$\alpha_1$&$\alpha_2$&$\alpha_3$&$\alpha_4$$ $&\\
    \hline
    \rowcolor{\graycol}
    $\vpi_0$     &  &$\vpi'_2+\vpi_2$  &1 &  &1 &2 &1 &  &  &   &$D_4$ \\
    \multirow{2}{*}{$\vpi'_2+\vpi_2$} &\multirow{2}{*}{\Big\{}
    &$2\vpi'_2+\vpi_4$ &1 &  &1 &2 &2 &2 &1 &   &$D_6$\\
          && $\vpi''+\vpi_3$  &  &1 &1 &1 &1 &1 &  &   &$A_5$\\
    $2\vpi'_2+\vpi_4$&  &$\vpi'_1+\vpi_4$  &  &1 &  &  &  &  &  &   &$A_1$\\
    \multirow{2}{*}{$\vpi''+\vpi_3$} &\multirow{2}{*}{\Big\{}
    &$\vpi'_1+\vpi_4$  &1 &  &  &1 &1 &1 &1 &   &$A_5$\\
          && $2\vpi_3$     &2 &1 &2 &3 &2 &1 &  &   &$E_6$\\
    \multirow{2}{*}{$\vpi'_1+\vpi_4$} &\multirow{2}{*}{\Big\{}
      &$\vpi_2+\vpi_4$  &1 &1 &2 &2 &1 &  &  &   &$D_5$\\
    & &$\vpi''+\vpi'_2$  &  &  &1 &1 &1 &1 &1 &1  &$A_6$\\
    $2\vpi_3$    &  &$\vpi_2+\vpi_4$  &  &  &  &  &  &  &1 &   &$A_1$\\
    \multirow{2}{*}{$\vpi_2+\vpi_4$} &\multirow{2}{*}{\Big\{}
      &$\vpi'_2+2\vpi_4$ &1 &  &1 &2 &2 &2 &1 &   &$D_6$\\
    & &$\vpi_1$      &  &  &  &  &  &1 &1 &1  &$A_3$\\
    \multirow{2}{*}{$\vpi'_2+2\vpi_4$}&\multirow{2}{*}{\Big\{}
      &$3\vpi_4$     &2 &2 &3 &4 &3 &2 &1 &   &$E_7$\\
    & &$\vpi'_2+\vpi_3$  &  &  &  &  &  &  &  &1  &$A_1$\\
    $3\vpi_4$    &  &$\vpi_3+\vpi_4$  &  &  &  &  &  &  &  &1  &$A_1$\\
    $\vpi''+\vpi'_2$ &  &$\vpi_1$  &1 &1 &1 &1 &  &  &  &   &$A_4$\\
    \rowcolor{\graycol}
    $\vpi_1$     &  &$\vpi'_2+\vpi_3$  &1 &  &1 &2 &2 &1 &  &   &$D_5$\\
    \multirow{2}{*}{$\vpi'_2+\vpi_3$} &\multirow{2}{*}{\Big\{}
      &$\vpi''+\vpi_4$  &  &1 &1 &1 &1 &1 &1 &   &$A_6$\\
   & &$2\vpi'_2$     &1 &  &1 &2 &2 &2 &2 &1  &$D_7$\\
    \multirow{2}{*}{$\vpi''+\vpi_4$} &\multirow{2}{*}{\Big\{}
      &$\vpi_3+\vpi_4$  &2 &1 &2 &3 &2 &1 &  &   &$E_6$\\
    & &$\vpi'_1$      &1 &  &  &1 &1 &1 &1 &1  &$A_6$\\
    $\vpi_3+\vpi_4$ &  &$\vpi_2$      &  &  &  &  &  &  &1 &1  &$A_2$\\
    $2\vpi'_2$    &  &$\vpi'_1$      &  &1 &  &  &  &  &  &   &$A_1$\\
    \rowcolor{\graycol}
    $\vpi'_1$     &  &$\vpi_2$      &1 &1 &2 &2 &1 &  &  &   &$D_5$\\
    \rowcolor{\graycol}
    $\vpi_2$     &  &$\vpi'_2+\vpi_4$  &1 &  &1 &2 &2 &2 &1 &   &$D_6$\\
    \multirow{2}{*}{$\vpi'_2+\vpi_4$} &\multirow{2}{*}{\Big\{}
      &$2\vpi_4$     &2 &2 &3 &4 &3 &2 &1 &   &$E_7$\\
   & &$\vpi''$      &  &1 &1 &1 &1 &1 &1 &1  &$A_7$\\
    $2\vpi_4$    &  &$\vpi_3$      &  &  &  &  &  &  &  &1  &$A_1$\\
    \rowcolor{\graycol}
    $\vpi''$     &  &$\vpi_3$      &2 &1 &2 &3 &2 &1 &  &   &$E_6$\\
    \rowcolor{\graycol}
    $\vpi_3$     &  &$\vpi'_2$      &1 &  &1 &2 &2 &2 &2 &1  &$D_7$\\
    \rowcolor{\graycol}
    $\vpi'_2$     &  &$\vpi_4$      &2 &2 &3 &4 &3 &2 &1 &   &$E_7$\\
    \rowcolor{\graycol}
    $\vpi_4$     &  &$0$        &3 &2 &4 &6 &5 &4 &3 &2  &$E_8$\\
    \hline
  \end{tabular}
  \medskip
  \caption{Partial order on dominant weights of $E_8$ below $\wpi_0$}
  \label{tab:E8-weights}
\end{table}
\fi

\begin{table}[ht!]
  \setlength{\tabcolsep}{3pt}
  \centering
  \begin{displaymath}
  \begin{array}{|c|c|c|c|cc|}
    \hline
    &a'_2=0&a_2=0&a_3=0&\multicolumn{2}{c|}{a_4=0} \\
    \hline

c'' & & & \vpi_3 & \vpi_3 &\\
\hline
c_2' & & & & \vpi_4 &\\
\hline
c_1' & & \vpi_2 & \vpi_2 & \vpi_2 & 2\vpi_4 \\
 & & & & \vpi_2'+\vpi_4 &\\
\hline
c_0  & \vpi_2'+\vpi_2 & \vpi_2'+\vpi_2 & \vpi_2'+\vpi_2 &
                                                             \vpi_2'+\vpi_2 & 2\vpi_3 \\
 & 2\vpi_2'+\vpi_4 & & \vpi''+\vpi_3 & 2\vpi_2'+\vpi_4 & \vpi_2+\vpi_4 \\
 & & & 2\vpi_3 & \vpi''+\vpi_3 & \vpi_2'+2\vpi_4 \\
 & & & & \vpi_1'+\vpi_4 & 3\vpi_4 \\
\hline
c_1  & \vpi_2'+\vpi_3 & & \vpi_2'+\vpi_3 & \vpi_2'+\vpi_3 & \vpi_3+\vpi_4 \\
 & 2\vpi_2' & & & \vpi''+\vpi_4 &\\
\hline
c_2  & \vpi_2'+\vpi_4 & & & \vpi_2'+\vpi_4 & 2\vpi_4 \\
\hline
c_3  & \vpi_2' & & & &\\
\hline    
  \end{array}
  \end{displaymath}

  \medskip
  \caption{For $E_8$, the dominant weights $\lambda<\vpi$ in $c$ 
    which survive degenerations}
  \label{tab:E8-survive}
\end{table}

The $E_8$ case is the hardest to analyze.
\ifshortversion
We computed the weights that
do survive under the degenerations in Table~\ref{tab:E8-survive}.
\else
We computed the poset of dominant weights below $\vpi_0$ in
Table~\ref{tab:E8-weights}.  Every line is a ``cover'', a minimal step
in the partial order, and we write the difference as a positive
combination of simple roots. The difference in a cover is known to be
equal to the highest root of some connected Dynkin subdiagram, see
e.g. \cite[Thm.2.6]{stembridge1998partial-order}. We give this diagram
in the last column. The corollary of that table is
Table~\ref{tab:E8-survive} showing the weights that do survive under
degenerations. 
\fi
All
other lower weights under these and all other degenerations
vanish. From this table we immediately see for example that when
either of the coordinates $a''$, $a'_1$, $a_0$, $a_1$ is zero, then
\emph{all} the lower weights vanish and we are left with the equations
of the $A$ or $\pA$ shapes.

  In the degeneration $a'_2=0$ the $E_8$ equation of
  Theorem~\ref{thm:E-discr} reduces to
  $c(x) = (\chi_0 + \chi_4) + ( \chi_1-\chi_3+2)x +
  (\chi_2-2\chi_4)x^2 + (\chi_3-3)x^3 + \chi_4 x^4 + x^5$, which is
  precisely the equation of the canonical $D_7$ family from
  Example~\ref{ex:D7}.

  For the degeneration $a_4=0$ one can check that the $E_8$ equation
  reduces to the canonical $E_7$ equation of
  Theorem~\ref{thm:E-discr}, and for $a_3=0$ it reduces to the
  $E_6$ equation.  The other cases are checked similarly. The
  $E_7$ and $E_6$ cases now follow.
\end{proof}

\begin{remark}\label{rem:weights-below-fundamental}
  As we see, the poset of the dominant weights below the 8 fundamental
  weights of $E_8$ is very complicated. We make the following
  interesting observation. Associate to the 8 nodes of the Dynkin
  diagram the following points in $\bZ^3$: $p_i = (i,0,0)$,
  $p'_j = (0,j,0)$, $p''_k = (0,0,k)$, and choose the special point
  $p_* =(1,1,1)$. Consider the projection
  $\psi\colon E_8\to \bZ\oplus\bZ^3$ by the rule
  $\psi(\vpi) = (1,p - p_*)$.  Then for a fundamental weight $\varpi$,
   a dominant weight $\lambda$
  satisfies $\lambda<\vpi$ iff $\psi(\vpi-\lambda)$ is a non-negative
  combination of the 8 vectors $\psi(\vpi_i)$ and the vector
  $(-1,0,0,0)$.

  The same procedure works for $D_n$, $E_6$, $E_7$. In the $D_n$ case
  this becomes an especially easy way to see the 
  relation~\eqref{eq:D-relation}. Our two-dimensional projection of
  section~\ref{sec:2dim-models} is a further projection from $\bZ^3$
  to $\bZ^2$ obtained by ``completing the square in the $z$
  variable''.
\end{remark}

\subsection{Singularities of divisors in \ade pairs}
\label{sec:sings-ade-pairs} 

By Theorem~\ref{thm:two-discrs}, the singularities of
$B\cap (Y\setminus C)$ in the canonical families occur on the fibers
$Y_t$ for $t \in \cup_\alpha \{e^\alpha=1\}$, the union of root
hypertori. Generically, these are $A_1$ singularities.  On the
intersections of several hypertori some worse singularities occur.
Below we describe them explicitly. For each of the lattices
$\Lambda = A_n, D_n, E_n$ the singularity over the point
$1\in T_{\Lambda^*}$ is that same $A_n, D_n, E_n$.  However, there are
zero-dimensional strata of the hypertori arrangement different from
$1$. Some other maximal rank singularities occur on the fibers over
those points.

\begin{definition}
  Let $\Lambda$ be an \ade lattice with a root system $\Phi$ and
  Dynkin diagram $\Delta$, and let $G$ be some abelian group which we
  will write multiplicatively. Let $t\in\Hom(\Lambda,G)$ be a
  homomorphism. Define the sublattice
  \begin{displaymath}
    \Lambda_t = \la \alpha \mid t(\alpha) = 1 \ra \subset \Lambda
    \ \text{generated by the roots } \alpha\in\Phi\cap\ker(t).
  \end{displaymath}
\end{definition}

It is well known that a sublattice of an \ade lattice generated by
some of the roots is a direct sum of root lattices corresponding to
smaller \ade Dynkin diagrams. All such root sublattices can be
obtained by the Dynkin-Borel-de Siebenthal (DBS) algorithm, see
\cite[Thms. 5.2, 5.3]{dynkin1952semisimple-subalgebras}, as
follows. Make several of the steps (DBS1): replace a connected
component of the Dynkin diagram by an extended Dynkin diagram and then
remove a node; and then several of the steps (DBS2): remove a
node. Below, we determine which of these lattices are realizable as
$\Lambda_t$. 

All root sublattices are listed in \cite[Tables
9--11]{dynkin1952semisimple-subalgebras}. The
answer is as follows.  Recall that the lattice
$A_n\subset \bZ^{n+1}$ is generated by the roots $e_i-e_j$.  All root
sublattices of $A_n$ are of the form
$A_{|I_1|-1} \oplus \dotsb \oplus A_{|I_s|-1}$, where
$I_1\sqcup \dotsb \sqcup I_s = \{1,\dotsc, n+1\}$ is a partition,
$|I_i|\ge 1$. Here, $A_{|I_i|-1}=0$ if $|I_i|=1$.

The lattice $D_n\subset \bZ^n$ is generated by the roots $e_i\pm
e_j$. All root sublattices of $D_n$ are of the form
$A_{|I_1|-1} \oplus \dotsb \oplus A_{|I_s|-1} \oplus D_{|J_1|} \oplus
\dotsb \oplus D_{|J_r|}$, where
$I_1\sqcup \dotsb \sqcup I_s \sqcup J_1\sqcup \dotsb \sqcup J_r =
\{1,\dotsc, n\}$ is a partition, $|I_i|\ge1$ and $|J_j|\ge 2$. $D_2$
and $D_3$ are a special case. They are isomorphic to $2A_1$ and $A_3$
respectively as abstract lattices, but they are different 
as sublattices of $D_n$. 

The sublattices of $E_6,E_7,E_8$ are listed in \cite[Table
11]{dynkin1952semisimple-subalgebras} but note the typos: in the $E_8$
table one of the two $A_7+A_1$ is $E_7+A_1$, and $A_6+A_2$ should be
$E_6+A_2$.

\begin{definition}
  Let ${M}\subset\Lambda$ be two \ade lattices. Let
  $\Tors(\Lambda/{M})$ be the torsion subgroup of $\Lambda/{M}$ and
  $\im(\Phi\cap{M}_\bR)\subset \Tors(\Lambda/{M})$ be the image of the
  set of roots $\alpha\in\Phi\cap{M}_\bR$.  We define the closure
  $\overline{\im}(\Phi\cap{M}_\bR)$ to be the subset of
  $\Tors(\Lambda/{M})$ consisting of the elements $x\ne0$ such that
  $0\ne nx\in \im(\Phi\cap{M}_\bR)$ for some $n\in\bN$; plus $x=0$.
  Both $\im(\Phi\cap{M}_\bR)$ and $\overline{\im}(\Phi\cap{M}_\bR)$
  are finite sets, and a priori neither of them has to be a group.
\end{definition}

\begin{lemma}
  Let ${M}\subset\Lambda$ be two \ade lattices. Let $G$ be an
  abelian group containing $\bZ^r$, where
  $r=\rk\Lambda - \rk{M}$. Then ${M}=\Lambda_t$ for some
  $t\in\Hom(\Lambda, G)$ iff there exists a homomorphism
  $\phi\colon\Tors(\Lambda/{M}) \to G$ such that for any
  $0\ne x\in \overline{\im}(\Phi\cap{M}_\bR)$ one has $\phi(x)\ne0$.

\end{lemma}
\begin{proof}
  Of course one must have ${M}\subset\ker(t)$, so the question is
  whether there exists a homomorphism $\Lambda/{M}\to G$ which
  does not map any roots not lying in $M$ to zero. 
  We have $\Lambda/{M} = \bZ^r \oplus
  \Tors(\Lambda/{M})$. An embedding $\bZ^r\to G$ can always be
  adjusted by an element of $\GL(r,\bZ)$ so that the images
  of roots not in $\Tors(\Lambda/{M})$ do not map to zero. So the
  only condition is on $\im(\Phi\cap{M}_\bR)$ in $\Tors(\Lambda/{M})$ or,
  equivalently, on its closure.
\end{proof}

\begin{corollary}\label{cor:allowed-homs}
  Let ${M}\subset\Lambda$ be two \ade lattices and let $k$ be an
  algebraically closed field of characteristic zero. If the group
  $\Tors(\Lambda/{M})$ is cyclic then ${M}=\Lambda_t$ for
  some $t\in\Hom(\Lambda,\bC^*)$. In the opposite direction, if
  $\overline{\im}(\Phi\cap{M}_\bR)$ contains a non-cyclic subgroup then
  ${M}\ne \Lambda_t$ for any $t\in\Hom(\Lambda,\bC^*)$.
\end{corollary}
\begin{proof}
  This follows from the fact that any finite cyclic group can be
  embedded into $\bC^*$, and there are no non-cyclic finite
  subgroups in  $\bC^*$.
\end{proof}

\begin{theorem}\label{thm:allowed-sublattices}
  Let $\Lambda$ be an irreducible \ade lattice and ${M}$ be an
  \ade root sublattice.  Assume that the field $k$ is algebraically
  closed of characteristic zero.  Then ${M} = \Lambda_t$ for some
  $t\in \Hom(\Lambda, \bC^*)$ iff any of the following
  equivalent conditions holds:
  \begin{enumerate}
  \item $\Tors(\Lambda/{M})$ is cyclic.
  \item ${M}$ is obtained from $\Lambda$ by a single DBS1 step
    and then some DBS2 steps.
  \item ${M}$ corresponds to a proper subdiagram of the extended
    Dynkin diagram $\wt\Delta$. 
  \item ${M}$ corresponds to a subdiagram $\Delta$ of the
    following Dynkin diagrams:
    \begin{enumerate}
    \item[$A_n$:] $A_n$; \qquad
    $D_n$: $D_n$ or $D_aD_b\subset\wD_n$ with $a+b=n$, $a,b\ge2$. 
    \item[$E_6$:] $E_6$, $A_5A_1$, $3A_2$; \qquad
      $E_7$: $E_7$, $D_6A_1$, $A_7$, $A_5A_2$, $2A_3A_1$;
    \item[$E_8$:] $E_8$, $E_7A_1$, $E_6A_2$,
      $D_8$, $D_5A_3$, $A_8$, $A_7A_1$, $2A_4$, $A_5A_2A_1$. 
    \end{enumerate}
  \item ${M}$ is not one of the following forbidden sublattices:
    \begin{enumerate}
    \item[$D_n$:] a sublattice with $\ge3$ $D$-blocks;
      \qquad
      $E_7$: $D_4\,3A_1$, $7A_1$, $6A_1$;
    \item[$E_8$:] $4A_2$, $2D_4$, $D_62A_1$, $D_44A_1$,
      $2A_32A_1$, $8A_1$,$D_43A_1$,$7A_1$,$A_34A_1$,$6A_1$.
    \end{enumerate}
  \end{enumerate}
\end{theorem}
\begin{proof}
  We first prove the equivalence of the conditions (1-5).  For one
  direction, the identity
  $\sum_{\alpha\in\wt\Delta} m_\alpha \alpha=0$ implies that if the
  Dynkin diagram $\Delta(M)$ is obtained from $\wt\Delta$ by removing
  one node (i.e. by a single DBS1 step) then the cotorsion group is
  cyclic of the order equal to the multiplicity $m_\alpha$ of the
  removed node in the highest root of $\Delta$. Any sublattice of
  these lattices obtained by DBS2 steps also has cyclic
cotorsion. The lists in (4) are simply the lattices obtained by one DBS1
  step.  To complete the equivalence of (1-5) for $E_n$ we use
  Dynkin's lists of sublattices together with \cite[Table
  1]{persson1990configurations-of-kodaira} which gives the torsion
  groups, and check the finitely many cases. The $D_n$ case is easy.

  Now let ${M}$ be a sublattice as in (1). Then
  ${M}=\Lambda_t$ for some $t\in\Hom(\Lambda,\bC^*)$ by
  Cor.~\ref{cor:allowed-homs}.  Vice versa, let ${M}$ be one of
  the sublattices with a non-cyclic $\Tors(\Lambda/{M})$,
  which are listed
  in (5). If $\Lambda=D_n$ and ${M}$ has $r\ge3$
  $D$-blocks then $\Tors(\Lambda/{M})=\bZ_2^{r-1}$ and we easily
  calculate $\im(\Phi\cap{M}_\bR)$ to be $\{0,e_i, e_i+e_j \mid 1\le i,j \le
  r-1\}$. This set contains a non-cyclic subgroup
  $\bZ_2^2 = \{0, e_1,e_2, e_1+e_2\}$, so ${M}\ne\Lambda_t$ by
  Cor.~\ref{cor:allowed-homs}.

  For each sublattice of $E_7$ and $E_8$ listed in (5) we explicitly compute
  $\overline{\im}(\Phi\cap{M}_\bR)$.  We have
  $(\Lambda\cap{M}_\bR)/{M} \subset {M}^*/{M}$, so
  we find the images of the roots
  $\alpha \in \Phi\cap {M}_\bR$ in ${M}^*/{M}$.  The
  result is as follows. For $8A_1$ the set $\im(\Phi\cap{M}_\bR)$
  has 15 elements and contains $\bZ_2^3$; for $2A_3\,2A_1$ it has 7
  elements and its closure is $\bZ_4\oplus\bZ_2$; for
  $4A_2$ it has 8 elements and its closure is $\bZ_3^2$.
  In all the
  other cases, one has $\im(\Phi\cap{M}_\bR) =
  \Tors(\Lambda/{M})$. We conclude that ${M}\ne\Lambda_t$ by
  Cor.~\ref{cor:allowed-homs}.
\end{proof}

\begin{theorem}\label{thm:ade-on-root-tori}
  Consider a canonical family of \ade pairs of
  Theorems~\ref{thm:A-discr}, \ref{thm:D-discr}, \ref{thm:E-discr}.
  Then for a point $t\in T$, the singularities of the curve
  $B_t \cap (Y_t \setminus C_t)$ and of the double cover
  $X_t\setminus D_t$ near $R_t$ are Du Val of the type corresponding to the
  lattice~$\Lambda_t$.  In particular, a curve is singular iff $t$
  lies in a union of root hypertori $\{e^\alpha=1\}$, and for $t=1$ there
  is a unique singularity of the same Du Val type as the root lattice.

\end{theorem}
\begin{proof}
  The $A_n$ case is obvious: the curve curve $-y^2/4+c(x)$,
  $c(x)=\prod (x+t_i)$ has singularities
  $A_{m_1-1}, \dotsc, A_{m_s-1}$, each occurring when some $m_{k}$ of
  the $t_i$'s coincide, i.e. when several of the monomials
  $e^{t_i-t_j}$ vanish at the same time.

  \smallskip Let $\Discr_y(f)= \prod_{i=1}^n (x+t_i+t_i\inv)$ as in
  the proof of Thm.~\ref{thm:D-discr}. It is easy to see that for
  every root $x\ne\pm2$ of $\Discr_y$ of multiplicity $m$, the curve
  $f=0$ has an $A_{m-1}$-singularity, and if $x=\pm2$ is a root of
  $\Discr_y$ of multiplicity $m$ then $f$ has a
  $D_m$-singularity. This includes $D_3=A_3$, $D_2=2A_1$, and
  $D_1={\rm smooth}$. On the other hand, the root tori are of the form
  $\{t_it_j^{\pm1}=1\}$. The irreducible components of $\Lambda_t$
  correspond to the disjoint subsets $I\subset\{1,\dotsc,n\}$ of
  indices for which $t_i=t_j^{\pm1}$ for $i,j\in I$. If $t_i\ne \pm1$,
  i.e. $t_i+t_i\inv\ne\pm2$, then the component is of the
  $A_{|I|-1}$-type; otherwise it is of the $D_{|I|}$-type.

  \smallskip In the $E_n$ cases the singularities are Du Val by
  construction in the proof of \ref{thm:E-discr}.  Using notation as
  in the proof, let us fix a linear function $\varphi$ on
  $E_8\subset A_8^*$ such that
  $\varphi(p) > \varphi(e_1) > \dotsb > \varphi(e_8)$, and let the
  positive roots $\alpha$ be those with $\varphi(\alpha)>0$.  Then for
  any subroot system of $E_8$ the simple roots are exactly the roots
  that are realizable by irreducible $(-2)$-curves on $\wX^8$:
  $e_i-e_j$ for $i>j$ (preimages of the exceptional divisors $E_i$ of
  blowups at $P_i$), $\ell-e_i-e_j-e_k$ (preimages of lines passing
  through 3 points $P_i,P_j,P_k$), $2\ell-\sum_{k=1}^6 e_{i_k}$
  (preimages of conics through 6 points), and
  $3\ell-2e_j-\sum_{k=1}^7 e_{i_k}$ (preimages of nodal cubics through
  8 points). So for every $t\in \Hom(E_8,\bC^*)$, the simple roots in
  the lattice $\Lambda_t$ are realized by $(-2)$-curves on $\wX_8$
  which contract to a configuration of singularities on $X^8$ with the
  same Dynkin diagram as $\Lambda_t$. The $E_7$ and $E_6$ cases are
  done similarly.
\end{proof}

\begin{remark}
  By the proof of Theorem~\ref{thm:E-discr}, the surfaces in the
  $E_6,E_7,E_8$ families correspond to rational elliptic fibrations
  with an $I_3,I_2,I_1$ fiber respectively. The singularity type of
  the double cover $X_t\setminus D_t$ is obtained from the Kodaira type of the elliptic
  fibration by dropping one $I_3,I_2,I_1$ fiber respectively (it gives
  a singularity of $X_t$ lying in the boundary $D_t$; of type $A_2,A_1$, or none resp.)
  and converting the other Kodaira fibers into the \ade singularities.

  As a check, we note that the list of maximal sublattices in
  Theorem~\ref{thm:allowed-sublattices}(4) is equivalent to the list
  of the rational extremal non-isotrivial elliptic fibrations in
  \cite[Thm.~4.1]{miranda1986on-extremal-rational}, and that the full
  list of sublattices in Theorem \ref{thm:allowed-sublattices} is consistent
  with the full list of Kodaira fibers of rational elliptic fibrations
  in \cite{persson1990configurations-of-kodaira}.  Persson's list
  contains 6 surfaces with an $I_m$ fiber for which the corresponding
  sublattice of $E_8$ has non-cyclic cotorsion: $I^*_2\, 2I_2$
  ($D_6\,2A_1$), $I^*_0\,3I_2$ ($D_4\, 3A_1$), $2I_4\,2I_2$
  ($2A_3\,2A_1$), $I_4\,4I_2$ ($A_3\,4A_1$), $4I_3$ ($4A_2$), $6I_2$
  ($6A_1$). But $D_6A_1$, $D_4\,2A_1$, $2A_3\,A_1$, $A_3\,3A_1$ and
  $5A_1$ are sublattices of $E_7$ and $3A_2$ is a sublattice of $E_6$,
  all with cyclic cotorsion.
\end{remark}

\section{Applications and connections with other works}
\label{sec:connections}

\subsection{Toric compact moduli of rational elliptic surfaces}
\label{sec:elliptic}

Let $M_\rmel$ be the moduli space of smooth rational elliptic
relatively minimal surfaces $S\to\bP^1$ with a section~$E$. Let
$M_\rmel(I_1)$ be the moduli space of such surfaces $(S,E,F)$ together
with a fixed $I_1$ Kodaira fiber $F$ (i.e. a rational nodal curve). This
is a $12:1$ cover of a dense open subset of $M_\rmel$ since a generic
rational elliptic surface has 12 $I_1$ fibers.

\begin{theorem}\label{thm:ell-surfaces}
  There exists a moduli compactification of $M_\rmel(I_1)$ by stable
  slc pairs whose normalization is the quotient $V\semi_\Lambda/W_\Lambda$ 
  of the projective toric variety $V_{\Lambda}^{\mathrm{semi}}$
   for the generalized Coxeter fan by the Weyl group $W_{\Lambda}$,
   where $\Lambda$ is the root lattice $E_8$.
\end{theorem}
\begin{proof}
Let $j\colon S\to S$ be the elliptic involution with respect to the section $E$
and $E\sqcup R$ be the fixed locus of $j$. Contracting the
$(-2)$-curves in the fibers which are disjoint from the section $E$
and then $E$ itself gives a pair $(X,D+\epsilon R)$ which is an \ade
double cover of shape $E_8$. Vice versa, any pair $(X,D+\epsilon R)$ of $E_8$
shape is a del Pezzo surface of degree 1 with Du Val
singularities. Blowing up the unique base point of $|-K_X|$ and
resolving the singularities gives a rational elliptic fibration
$S\to\bP^1$ and the strict preimage of $D$ is an $I_1$ fiber of this
fibration. This theorem is now the $E_8$ case of
Theorem~\ref{thm:compact-ade-moduli}. 
\end{proof}

Similarly, the $E_7$ compactified family gives a moduli
compactification $\oM_\rmel(I_2)$ of the moduli space $M_\rmel(I_2)$ of rational
elliptic surfaces with an $I_2$ Kodaira fiber;
the $E_6$ family gives $\oM_\rmel(I_3)$; 
the $D_5^-$ family gives $\oM_\rmel(I_4)$; 
and the $\php A_4^-$ family gives $\oM_\rmel(I_5)$.

\subsection{Moduli of Looijenga pairs after Gross-Hacking-Keel}
\label{sec:GHK} 

A Looijenga pair is a smooth rational surface $(\wX, \wD)$ such that
$K_\wX + \wD \sim 0$ and $\wD$ is a cycle of rational curves.  In
\cite{gross2015moduli-of-surfaces}, Gross-Hacking-Keel construct
moduli of Looijenga pairs of a fixed type, given by the configuration
of the rational curves $\wD$. The result is as follows. First, one
defines the lattice $\Delta\subset\Pic\wX$ as the orthogonal to the
irreducible components of $\wD$, and the torus
$T_{\Delta} = \Hom(\Delta,\bC^*)$.  One glues several copies of this
moduli torus along dense open subsets into a nonseparated scheme $U$
and divides it by a group $\mathrm{Adm}$ of admissible monodromies,
including reflections in the $(-2)$-curves appearing on some
deformations of $(\wX, \wD)$.  The non-separatedness is expected since
$\wX$ in this setup are smooth surfaces without a polarization.  The
separated quotient of $[U/\mathrm{Adm}]$ is
$[T_{\Delta}/\mathrm{Adm}]$.

\smallskip

For an \ade double cover $(X,D + \epsilon R)$, the minimal resolution
of singularities $(\wX, \wD)$ is a Looijenga pair.  In
Theorems~\ref{thm:moduli-ade-pure}, \ref{thm:moduli-ade-toric-primed},
\ref{thm:moduli-ade-primed} we proved that the moduli space of \ade
pairs and of their double covers is a torus
$T_{\Lambda'} = \Hom(\Lambda', \bC^*)$ modulo a certain Weyl group
$W_\Lambda\rtimes W_0$. The lattices $\Lambda$, $\Lambda'$ and 
the Weyl groups $W_\Lambda$, $W_0$ were 
introduced in Section~\ref{sec:moduli}.
We now relate them to the lattices naturally
associated to Looijenga pairs with a nonsymplectic involution.

\begin{definition}
  Let $(\wX, \wD)$ be a Looijenga pair with an involution.  Let
  $\Delta = \wD^\perp$ be the sublattice of $\Pic\wX$ which is
  orthogonal to the curves in the boundary. Assume that there is an
  involution $\iota\colon \wX\to \wX$ with $\iota(\wD)=\wD$. We define
  $\dplus$ and $\dminus$ as the $(\pm1)$-eigensublattices of the
  induced involution $\iota^*\colon \Delta\to\Delta$.  Denote by
  $\dminustwo$ the set of $(-2)$-vectors in $\dminus$, and by $\wtwo$
  the group generated by reflections in them.
\end{definition}

\begin{theorem}\label{thm:ghk-connection}
  Let \ycb be an \ade pair and \xdr be its double cover, with the
  minimal resolution $(\wX,\wD)$. Then one has $\Lambda = \dminus$ and
  $W_\Lambda=\wtwo$. Further,
  $\Lambda' \subset \Delta/\dplus$, with equality if and only if the
  shape has no doubly primed sides. For a doubly primed shape $S''$
  (resp. $\ppS$), $\Delta/\dplus$ is the same as for the shape $S'$
  (resp. $\pS$); it thus contains $\Lambda'$ as a sublattice of index
  $2^N$, where $N$ is the number of sides on which the shape has a
  double prime.
\end{theorem}
\begin{proof}
  We prove the statement in representative $D$ cases, with the other
  cases done by similar computations.

  \smallskip
  
  ($D_{2n}$) The easiest model for a generic surface $X=\wX$ of this
  shape is as a blowup of $\bP^1\times\bP^1$ with a section $s$ and a
  fiber $f$ at $2n$ points lying on a curve in $|2s+f|$. Using $e_i$
  for the exceptional divisors in $\Pic X$, the boundary curves are
  $D_1 \sim 2s+f-\sum_{i=1}^{2n}e_i$, and $D_2\sim f$. Then $\Delta$
  is generated by the  roots $e_i - e_{i+1}$,
  $1\le i\le 2n-1$ and $f-e_1-e_2$ forming a $D_{2n}$ Dynkin
  diagram. The involution acts by $f\mapsto f$,
  $s\mapsto s + nf -\sum_{i=1}^{2n}e_i$, $e_i\mapsto f-e_i$. Thus, it acts as
  $(-1)$ on $\Delta$ and $\dminus$ is the root lattice $\Lambda$ of
  type $D_{2n}$. In this case $\dplus=0$ and $\Delta/\dplus = \dminus
  = \Lambda=\Lambda'$.

  \smallskip
  
  ($D_{2n}'$) The surface $X$ is obtained from the one for $D_{2n}$ by
  a blowup at one of the two points in $R\cap D_2$. Denoting by $g$
  the exceptional divisor, one has $D_1 \sim 2s+f-\sum_{i=1}^{2n}e_i$,
  and $D_2\sim f-g$. The lattice $\Delta$ is generated by the $2n$
  roots above and an additional root $\beta=s-e_1-g$. This forms a
  Dynkin diagram obtained by attaching an additional node $\beta$ to
  one of the short legs of $D_{2n}$, $\alpha'$ or $\alpha''$.
   Without loss of
  generality, let us say $\beta \alpha'=1$.  The involution $\iota$ acts on the
  vectors $s,f,e_i$ the same way as above, and $\iota^*g =g$. Thus,
  $\dplus$ is spanned by the vector $t=\beta+\iota^*\beta$ and
  $\dminus$ is the same $D_{2n}$ root lattice as before. 
  We have an orthogonal projection $p\colon \Delta \to \frac12\dminus$
  identifying $\Delta/\dplus$ with a sublattice of $\frac12\dminus$
  generated by $\dminus$ and the image $p(\beta)$. For a root
  $\alpha\in\dminus$ one has $p(\beta)\alpha = \beta \alpha$, so $\beta \alpha'=1$ and
  $\beta r=0$ for the other roots $\alpha$. Thus, $p(\beta)=\vpi'$, the
  fundamental weight $\vpi'$ for the root $\alpha'$, and
  $\Delta/\dplus = \Lambda + \vpi'$ is our $\Lambda'$.

  \smallskip
  
  ($\pD_{2n}$) The surface $X$ is obtained from the one for $D_{2n}$
  by a blowup at one of the two points in $R\cap D_1$. Denoting by $g$
  the exceptional divisor again, one has
  $D_1 \sim 2s+f-\sum_{i=1}^{2n}e_i - g$, and $D_2\sim f$. The lattice
  $\Delta$ is generated by the $2n$ roots above and an additional root
  $\beta=e_{2n}-g$. This forms a Dynkin diagram obtained by attaching
  an additional node $\beta$ to the long leg of $D_{2n}$, i.e. to
  $\alpha_{2n-3}$ in our notation. The $(-1)$-eigenspace $\dminus$ is
  again the $D_{2n}$ root lattice generated by the first $2n$ roots.
  The space $\dplus$ is generated by $t=\beta+\iota^*\beta = f-2g$.
  The orthogonal projection $p$ identifies $\Delta/\dplus$ with
  $\dminus + p(\alpha)$. And since one has $\beta \alpha_{2n-3}=~1$
  and $\beta$ is orthogonal to the other $2n-1$ roots,
  $p(\beta) = \vpi_{2n-3}$. So one has
  $\Delta/\dplus = \Lambda + \vpi_{2n-3} = \Lambda'$, as claimed.

  \smallskip

  ($\pD'_{2n}$) Similarly, in $\Delta$ one has two extra roots
  $\beta_1 = s-e_1-g_1$ and $\beta_2=e_{2n}-g_2$ whose images in
  $\frac12\dminus$ are $\vpi'$ or $\vpi''$ depending on the parity of
  $n$,  and $\vpi_{2n-3}'$, so $\Delta/\dplus = \Lambda'$ again.

  \smallskip
  
  When priming a surface of shape $S$ twice on the same side (say on
  the right), there are two exceptional divisors $g_1,g_2$. Then
  $\Delta(S'') = \Delta(S') \oplus \bZ(g_1-g_2)$,
  $\dplus(S'') = \dplus(S') \oplus \bZ(g_1-g_2)$,
  $\dminus(S'')=\dminus(S')$. Therefore, $\Delta/\dplus(S'') =
  \Delta/\dplus(S')$. This applies to $D''_{2n}$, $\ppD_{2n}$ and all
  the other doubly primed shapes.
\end{proof}

Next, we define an action of the Weyl group $W_0$ of the lattice
$\Lambda_0 = C^\perp \cap B^\perp$ introduced in Def.~\ref{def:CBperp}.

\begin{definition}
  Let $\pi\colon X\to Y$ be an double cover of a \ade pair with a
  branch divisor $B$. 
  Let $\tilde\pi\colon\wX\to \wY$ be a
  double cover of its resolution of singularities.  Let
  $e\in\Lambda_0^{(2)}$ be a cycle, so
  $e\in C^\perp\cap B^\perp$ and $e^2=-2$.  Then $\pi^*e = e_1 + e_2$
  with $\iota^*e_1=e_2$, $e_1^2=e_2^2=-2$ and $e_1e_2=0$. 

  We define $v_+= \pi^*(e) = e_1+e_2 \in \dplus$ and
  $v_-=e_1-e_2\in \dminus$. The composition
  of two reflections $w_{e_1}\circ w_{e_2} = w_{e_2}\circ w_{e_1}$
  acts on $\dminus$ as a reflection $w_{v_-}$ in the $(-4)$-vector
  $v_-$, and on $\dplus$ as a reflection $w_{v_+}$ in the $(-4)$-vector
  $v_+$.
\end{definition}

\begin{lemma}
  Given $e\in\Lambda_0$, $w_{e_1}\circ w_{e_2}$ is well defined up to
  a conjugation by $\wtwo$. 
\end{lemma}
\begin{proof}
  Suppose we have another decomposition $v_+=e_1+e_2 = e'_1 + e'_2$.  One
  has $e_1 = \frac12 (v_++v_-)$ and $e'_1 = \frac12 (v_++v'_-)$.
  Then $e_1e'_1 = -1 + \frac14 v_-v'_-$.  Since
  $\dminus \subset R^\perp$ and $R^2>0$, $\dminus$ is negative
  definite. Thus, $|v_-v'_-|<4$, and we conclude that
  $e_1e'_1=-1$. The elements $w_{e_1}\circ w_{e_2}$ and
  $w_{e'_1}\circ w_{e'_2}$ are conjugate by the reflection
  $w_{e_1-e'_1} = w_{e_2-e'_2}$. Finally, $e_1-e'_1 \in\dminus$ and
  $(e_1-e'_1)^2 = -2$, so $w_{e_1-e'_1} \in \wtwo$.
\end{proof}

\begin{definition}
  We define the Weyl group $\wtwofour$ as the group of reflections of
  $\Delta$ generated by $\wtwo$ and the elements $w_{e_1}\circ
  w_{e_2}$ for $e\in\Lambda_0$. By the above, it preserves both
  $\dminus$ and $\dplus$, with $\wtwo$ acting trivially on
  $\dplus$. Thus, we have the induced actions of $\wtwofour$ on
  $\dminus$ and of $\wtwofour/\wtwo$ on $\dplus$.
\end{definition}

\begin{theorem}\label{thm:ghk-connection2}
  One has $\wtwofour/\wtwo = W_0$. The subgroup $W_{00}$ from
  Definition~\ref{def:W00} is the
  subgroup of $W_0$ which acts trivially on $\dminus$.  
\end{theorem}

\begin{proof}
  We compute the action of
  $W_0$ in the representative $D$ cases using the same notation as in
  the proof of Theorem~\ref{thm:ghk-connection}.
  The lattice $\dplus\cap R^\perp$ can be identified with $\pi^*(\Lambda_0)$
  and the $(-4)$-vectors $v_+$ in $\dplus\cap R^\perp$ with the vectors
  $\pi^*(e)$ for $e\in\Lambda_0^{(2)}$.

  ($D_{2n}$) $\Lambda_0=0$ and $\dplus=0$; there is nothing to check.

  ($D_{2n}'$) One has $t^2= 2n-8$. This equals $-4$ only for $n=2$ and
  then $D'_4=\pD_4$.

  ($\pD_{2n}$) One has $\beta\cdot \iota^*\beta =0$. So for the
  generator $t=\beta+\iota^*\beta = f-2g$ of $\dplus$ one has
  $t^2=-4$. Indeed, $t=v_+ = \pi^*e$ for the generator $e$ of
  $\Lambda_0$.  Then $v_-=\beta-\iota^*\beta = 2e_{2n}-f$. Reflection
  $w_{v_-}$ in this vector fixes all roots of the $D_{2n}$ diagram
  except for $w_{v_-}(\alpha_{n-3}) =
  \alpha:=e_{2n-1}+e_{2n}-f$. Together with the other $2n$ roots,
  $\alpha$ forms the $\wD_{2n}$ diagram in which $\alpha_{n-3}$,
  $\alpha$ are two short legs. Thus, $w_{v_-}$ acts as an outer
  automorphism of $\Lambda(D_{2n})$ swapping two short legs. This is
  the same action for $W_0=S_2$ which we computed in
  subsection~\ref{sec:W0-action}.

  ($\ppD_{2n}$) One has $\dplus = \la f-2g_1, g_2-g_1\ra$. The only
  vectors $v_+$ of square $-4$ in $\dplus$ are $f-2g_1$ and $f-2g_2$,
  which are the pullbacks of the two vectors in $\Lambda_0^{(2)}$. For
  both of them we get the same vector $v_- = 2e_{2n}-f$. Thus,
  $w_{e_1}^{(1)}\circ w_{e_2}^{(1)}$ and
  $w_{e_1}^{(2)}\circ w_{e_2}^{(2)}$ for these two vectors act in the
  same way on $\dminus$ but differently on $\dplus$. We conclude that
  they generate $S_2\times S_2$ and their difference acts trivially on
  $\dminus$. This is the same description of $W_0=S_2\times S_2$ and
  $W_{00}=S_2$ as in~\ref{sec:W0-action}.

  ($\pD'_4$) $\pi^*\Lambda_0$ is generated by
  $v^1_+ = \beta_1 + \iota^*\beta_1$ and
  $v^2_+ = \beta_2 + \iota^*\beta_2$, $\beta_1=s-e_1-g_1$ and
  $\beta_2=e_4-g_2$. Then $v^1_- = -f-e_1+e_2+e_3+e_4$ and $v^2_- = -f
  + 2e_4$. Denote by $-\alpha$ the highest root, so that together with
  the other 4 roots it forms the $\wD_4$ diagram. Then
  $w_{v^1_-}$ swaps $\alpha'$ and $\alpha$, and 
  $w_{v^2_-}$ swaps $\alpha_1$ and $\alpha$. Thus, $W_0$ acts as the
  group $S_3$ of outer automorphisms of $\Lambda(D_4)$, the same as
  in~\ref{sec:W0-action}. 
\end{proof}

We now describe, without proof, how our moduli stack of \ade pairs
(equivalently, up to the $\mu_2$-cover, the stack of \ade double
covers with involution), which by Theorem~\ref{thm:moduli-ade-primed}
equals $[T_{\Lambda'} : W_\Lambda \rtimes W_0]$, is related to the
moduli of Looijenga pairs.  In the moduli torus $T_\Delta$ of
Looijenga pairs the subtorus $T_{\Delta/\dplus}$ corresponds to the
pairs admitting a nonsymplectic involution.  The moduli stack is the
quotient of it by the group of admissible monodromies of $T_\Delta$
leaving $T_{\Lambda'}$ invariant. A part of this group is obvious:
reflections $\wtwo$ in the vectors in $\Delta_-^{(2)}$. Also, for each
side which has a double prime there is a root $g_1-g_2$ which gives a
quotient by $\mu_2$ that forgets the ordering of the two primed
points.  This accounts for the fact that $\Lambda'$ is a sublattice of
$\Delta/\Delta_+$ for shapes with doubly primed sides.  Less
obviously, for each $e\in\Lambda_0^{(2)}$, with $\pi^*(e) = e_1+e_2$,
while the reflections $w_{e_1}$ and $w_{e_2}$ by themselves do not fix
$\dminus$, their composition $w_{e_1}\circ w_{e_2}$ does.

One thus takes a quotient of $T_{\Lambda'}$ by $\wtwo=W_\Lambda$ followed
by a quotient by $\wtwofour/\wtwo = W_0$.  The subgroup
$W_{00} \subset \wtwofour/\wtwo$ acts trivially on the coarse moduli
space $T_{\Lambda'}$ but nontrivially on the stack, giving extra
automorphisms of the pairs.  

\subsection{Involutions in the Cremona group} 
\label{sec:cremona}

Classically, the involutions in the Cremona group $\mathrm{Cr}(\bP^2)$, the group of
birational automorphisms of $\bP^2$, are of three types: De
Jonqui\`eres, Geiser, and Bertini. For a nice modern treatment that uses equivalent MMP, see
\cite{bayle2000birational-involutions}. For
a $(K+D)$-trivial polarized involution pair $(X, D, \iota)$, if $X$ is rational then $\iota$ is
an involution in $\mathrm{Cr}(\bP^2)$. 

\begin{theorem}
  Let $(X,D,\iota)$ be a $(K+D)$-trivial polarized involution pair with rational surface $X$
  and a smooth ramification curve $R$. Then
  \begin{enumerate}
  \item If $(X,D,\iota)$ is of shape $\wD$, $D$, or $A$ (pure or
    primed) then $\iota$ is De Jonqui\`eres.
  \item If it is of shape $\wE_7$, $E_7$, or $E_6$ (pure or primed)
    then $\iota$ is Geiser.
  \item If it is of shape $\wE_8$ or $E_8$ (pure or primed) then
    $\iota$ is Bertini.
  \end{enumerate}
\end{theorem}
\begin{proof}
  By \cite[Prop. 2.7]{bayle2000birational-involutions}, the type of
  the involution is uniquely determined by the normalization $\widetilde{R}$ of the
  ramification curve $R$: for De Jonqui\`eres $\wR$ is
  hyperelliptic, for Geiser it is non-hyperelliptic of genus 3, and
  for Bertini it is non-hyperelliptic of genus 4. In the $\wD$-$D$-$A$
  cases the branch curve $B\simeq R$ is a two-section of a ruling, so
  it is hyperelliptic. In the $\wE_7$-$E_7$-$E_6$ cases $R$ is a
  quartic curve in $\bP^2$, so a non-hyperelliptic curve of genus 3,
  and in the $\wE_8$-$E_8$ cases it is a section of $\cO(1)$ on the
  quadratic cone $\bF_2^0$, so a non-hyperelliptic curve of genus 4.
\end{proof}

\begin{remark}
  When $R$ has nodes, the involution may easily be of a different
  type. When it has $\ge2$ nodes, the involution is always De
  Jonqui\`eres. 
\end{remark}

We can give an alternative proof for the classification of the double
covers $(X,D)\to (Y,C)$ of log canonical non-klt surfaces using 
\cite{bayle2000birational-involutions} in some cases:

\begin{theorem}\label{thm:generic-pairs}
  Let $(X,D,\iota)$ be a $(K+D)$-trivial polarized involution pair
  with rational $X$. Suppose that $X$ is smooth outside of the
  boundary $D$, and in particular that the ramification curve $R$ is
  smooth.  Then the quotient $(Y,C)$ of this pair is an \ade or \wade
  surface defined in Section~\ref{sec:adepairs}.
\end{theorem}
\begin{proof}[Sketch of the proof]
  Let $\wX$ be the minimal resolution of $X$, it comes with an induced
  involution $\tilde\iota$.
  \cite[Thm. 1.4]{bayle2000birational-involutions} gives six
  possibilities for the pair $(\wX,\tilde\iota)$
  when it is minimal, i.e. there does not exist one or two
  $(-1)$-curves that can be equivariantly contracted to another smooth
  surface with an involution. In our case, $\wX$ is obtained from such
  a minimal surface by a sequence of single or double blowups which
  satisfy two conditions: they have to be involution-invariant, and
  there are no $(-2)$-curves disjoint from $B$. 

  It follows that $\wX$ is obtained by blowups at the points $B\cap
  R$, either one involution-invariant point or two points exchanged by
  the involution. We analyze them directly. The different cases of
  \cite[Thm. 1.4]{bayle2000birational-involutions} then lead to the
  following:
  \begin{enumerate}
    \renewcommand{\theenumi}{\roman{enumi}}
  \item impossible, i.e. does not lead to a $(K+D)$-trivial polarized
    involution pair with ample~$R$.
  \item $(ii)_{\rm sm}$ is impossible, and $(ii)_g$ gives the
    $\wD$-$D$-$A$ shapes.
  \item $A_0^-$ and $\wA_0^-$. 
  \item $\wA_1^\exo$, 
    $A_1$.
  \item $\wE_7$, $\phmi E_7$, $\phmi E_6^-$ and the primed shapes. 
  \item $\wE_8^-$, $\phmi E_8^-$ and the primed shapes. 
  \end{enumerate}
\end{proof} 

One could try to extend the results of this section to classify
\emph{families} of log del Pezzo pairs, in which the surface $Y$ may
acquire singularities away from the boundary. This would give an
alternative proof of Theorem~\ref{thm:logdP=ade}.  For this, we would
first need to know that the branch divisor $B$ can be smoothed. This
is known, see \cite[Cor.3.20]{nakayama2007classification-of-log}.
Secondly, we would also need to know that the singular points of the
surface $Y$ away from the boundary can be smoothed. For surfaces
without the boundary, this is
\cite[Prop. 3.1]{hacking2010smoothable-del-pezzo}. For the pairs
$(Y,C)$ with boundary this does not seem to be easy to prove
directly. This follows \emph{a posteriori} from the classification of
\emph{all} log del Pezzo surfaces with boundary given in
Sections~\ref{sec:adepairs} and \ref{sec:nakayama}.


\bibliographystyle{amsalpha}

\def\cprime{$'$}
\providecommand{\bysame}{\leavevmode\hbox to3em{\hrulefill}\thinspace}
\providecommand{\MR}{\relax\ifhmode\unskip\space\fi MR }
\providecommand{\MRhref}[2]{%
  \href{http://www.ams.org/mathscinet-getitem?mr=#1}{#2}
}
\providecommand{\href}[2]{#2}

\ifshortversion
\else
\listoftables
\listoffigures
\fi

\end{document}